\documentclass[12pt, a4paper]{amsart}

\usepackage[unicode]{hyperref}
\usepackage{color}

\usepackage[english]{babel}
\usepackage{amssymb,amsmath,amsthm}
\usepackage{verbatim}

\usepackage[all]{xy}

\usepackage[normalem]{ulem}

\usepackage{enumitem}
\setlist[itemize]{leftmargin=*, itemsep={2pt}}
\setlist[enumerate]{leftmargin=*, itemsep={2pt}}

\usepackage{tikz}

\numberwithin{equation}{section}

\theoremstyle{theorem}

\newtheorem{dummy}{dummy}[section]

\newtheorem{theorem}[dummy]{Theorem}
\newtheorem{corollary}[dummy]{Corollary}
\newtheorem{lemma}[dummy]{Lemma}
\newtheorem{proposition}[dummy]{Proposition}

\theoremstyle{remark}

\newtheorem{definition}[dummy]{Definition}
\newtheorem{remark}[dummy]{Remark}
\newtheorem{example}[dummy]{Example}

\def\A{\mathbb A}

\def\F{\mathbb F}

\def\P{\mathbb P}
\def\R{\mathbb R}

\def\Z{\mathbb Z}
\def\AA{\mathcal A}

\def\DD{\mathcal D}

\def\FF{\mathcal F}

\def\LL{\mathcal L}
\def\MM{\mathcal M}

\def\OO{\mathcal O}

\def\cT{\mathcal T}
\def\cD{\mathcal D}

\def\fm{\mathfrak m}
\def\fR{\mathfrak R}

\newcommand{\rh}{\mathbf{h}}

\newcommand{\bD}{\mathbf{D}}

\def\Spec{\mathrm{Spec}}

\def\Db{\DD^b}
\def\Dm{\DD^-}
\def\Dperf{\DD^{\mathrm{perf}}}

\def\={\;=\;}
\def\bal{\begin{aligned}}
\def\eal{\end{aligned}}
\def\be{\begin{equation}\label}
\def\ee{\end{equation}}

\def\wt{\widetilde}

\def\ol{\overline}

\def\tors{{\mathrm{tors}}}
\def\perf{{\mathrm{perf}}}
\def\hf{{\mathrm{fd}}}

\def\rB{\mathrm{B}}
\def\rG{\mathrm{G}}
\def\rK{\mathrm{K}}
\def\rM{\mathrm{M}}
\def\rN{\mathrm{N}}
\def\rP{\mathrm{P}}
\def\rT{\mathrm{T}}

\newcommand{\bmu}{\boldsymbol{\mu}}

\newcommand{\md}{\text{-}\mathrm{mod}}
\newcommand{\fd}{\mathrm{\,fd}}
\newcommand{\ac}{\mathrm{\,acycl}}
\newcommand{\op}{\mathrm{op}}

\DeclareMathOperator{\Res}{Res}
\DeclareMathOperator{\Ind}{Ind}

\newcommand{\et}{\textrm{\'et}}

\newcommand{\cA}{\mathcal{A}}

\newcommand{\cF}{\mathcal{F}}
\newcommand{\cG}{\mathcal{G}}
\newcommand{\cH}{\mathcal{H}}
\newcommand{\cL}{\mathcal{L}}
\newcommand{\cM}{\mathcal{M}}
\newcommand{\cO}{\mathcal{O}}
\newcommand{\cP}{\mathcal{P}}
\newcommand{\cR}{\mathcal{R}}
\newcommand{\cV}{\mathcal{V}}

\newcommand{\tx}{{\tilde{x}}}
\newcommand{\tX}{{\widetilde{X}}}
\newcommand{\tV}{{\widetilde{V}}}
\newcommand{\tcA}{{\widetilde{\cA}}}
\newcommand{\tcF}{{\widetilde{\cF}}}
\newcommand{\tcL}{{\widetilde{\cL}}}

\newcommand{\tcR}{{\widetilde{\cR}}}
\newcommand{\tcV}{{\widetilde{\cV}}}

\DeclareMathOperator{\ilim}{\mathop{lim}\limits_{\longrightarrow}}
\DeclareMathOperator{\IP}{IP}
\DeclareMathOperator{\MP}{MP}

\DeclareMathOperator{\id}{id}
\DeclareMathOperator{\rk}{rk}

\DeclareMathOperator{\Ext}{Ext}
\DeclareMathOperator{\rc}{c}
\DeclareMathOperator{\CH}{CH}
\DeclareMathOperator{\Cl}{Cl}
\DeclareMathOperator{\Pic}{Pic}
\DeclareMathOperator{\Br}{Br}
\DeclareMathOperator{\Ker}{Ker}
\let\Im\relax
\DeclareMathOperator{\Im}{Im}
\DeclareMathOperator{\Coker}{Coker}
\DeclareMathOperator{\Sing}{Sing}
\DeclareMathOperator{\Supp}{Supp}
\DeclareMathOperator{\Coh}{Coh}

\newcommand{\Gm}{\mathbb{G}_{\mathrm{m}}}

\newcommand{\dg}[1]{{#1}_{\mathrm{dg}}}

\DeclareMathOperator{\tridiag}{\mathsf{tridiag}}
\DeclareMathOperator{\tridet}{\mathsf{tridet}}

\DeclareMathOperator{\Aut}{Aut}
\DeclareMathOperator{\Hom}{Hom}
\DeclareMathOperator{\RHom}{RHom}
\DeclareMathOperator{\cRHom}{{\mathrm{R}\mathcal{H}\mathit{om}}}
\DeclareMathOperator{\End}{End}
\DeclareMathOperator{\CEnd}{{\mathcal{E}\!\mathit{nd}}}

\makeatletter  
\newcommand{\leqnomode}{\tagsleft@true}
\newcommand{\reqnomode}{\tagsleft@false}
\makeatother

\title[Derived categories of singular surfaces]{Derived categories of singular surfaces}
\author{Joseph Karmazyn, Alexander Kuznetsov, Evgeny Shinder}

\email{j.h.karmazyn@gmail.com}

\address{{\sloppy
\parbox{0.9\textwidth}{
Algebraic Geometry Section, Steklov Mathematical Institute of Russian Academy of Sciences,
8 Gubkin str., Moscow 119991 Russia
\\[5pt]
Laboratory of Algebraic Geometry, National Research University Higher School of Economics, Russian Federation
}\bigskip}}
\email{akuznet@mi-ras.ru}

\address{School of Mathematics and Statistics, University of Sheffield,
Hounsfield Road, S3 7RH, UK, and
National Research University Higher School of Economics, Russian Federation}
\email{e.shinder@sheffield.ac.uk}

\begin{document}

\maketitle

\begin{abstract}
We develop an approach that allows one to construct semiorthogonal decompositions
of derived categories of surfaces with cyclic quotient singularities
whose components are equivalent to derived categories of local finite dimensional algebras.

We first explain how to induce a semiorthogonal decomposition of a surface~$X$ with rational singularities
from a semiorthogonal decomposition of its resolution.
In the case when~$X$ has cyclic quotient singularities,
we introduce the condition of adherence for the components of the semiorthogonal decomposition of the resolution
that allows one to identify the components of the induced decomposition of~$X$
with derived categories of local finite dimensional algebras.
Further, we present an obstruction in the Brauer group of $X$ to the existence of such a semiorthogonal decomposition,
and show that in the presence of the obstruction a suitable modification of the adherence condition gives
a semiorthogonal decomposition of the twisted derived category of $X$.

We illustrate the theory by exhibiting a semiorthogonal decomposition for the untwisted or twisted derived category
of any normal projective toric surface depending on whether its Weil divisor class group is torsion-free or not. For weighted projective planes we compute the generators of the components explicitly and
relate our results to the results of Kawamata based on iterated extensions of reflexive sheaves of rank~1.
\end{abstract}

\tableofcontents

\section{Introduction}

\subsection{Overview}

In this paper we study bounded derived categories of coherent sheaves on singular surfaces with rational singularities
over an algebraically closed field $\Bbbk$ of characteristic zero.
We are primarily interested in rational surfaces $X$ but for most of the arguments it is sufficient to assume that
\begin{equation}
\label{eq:pg-q-x}
p_g(X) = q(X) = 0.
\end{equation}
Our aim is constructing semiorthogonal decompositions of the form
\begin{equation}\label{intro:main-dec}
\Db(X) = \langle \Db(K_1\md), \Db(K_2\md), \dots, \Db(K_n\md) \rangle,
\end{equation}
where $K_i$ are (possibly noncommutative) local finite-dimensional algebras.
We consider~\eqref{intro:main-dec} as a generalization for singular varieties of the notion of a full exceptional collection.

An instructive example is given by the nodal quadric $X = \P(1,1,2)$,
where a semior\-thogonal decomposition can be constructed using~\cite{K08b}:
$X$ admits an exceptional pair of line bundles
whose orthogonal is equivalent to the derived category of
the even part of a degenerate Clifford algebra which is Morita-equivalent to $\Bbbk[z]/z^2$;
thus we have a decomposition
\begin{equation*}
\Db(\P(1,1,2)) = \langle \Db(\Bbbk), \Db(\Bbbk), \Db(\Bbbk[z]/z^2) \rangle.
\end{equation*}
So, here $K_1 = K_2 = \Bbbk$ and $K_3 = \Bbbk[z]/z^2$.

A new feature that we want to emphasize is the appearance of the Brauer group~$\Br(X)$
as an obstruction to the existence of the decomposition \eqref{intro:main-dec}.
We give a complete answer for existence of such decompositions in the case of normal projective toric surfaces,
that is for a projective toric surface $X$ we construct the decomposition~\eqref{intro:main-dec} as soon as~$\Br(X) = 0$
and describe the algebras $K_i$ in terms of its
singular points explicitly, following \cite{KK}.

Our approach to construct \eqref{intro:main-dec}
is based on descending semiorthogonal decompositions from a resolution~$\wt{X}$ of~$X$.
Before going into details, we would like to mention that some results in this direction
were obtained earlier by Kawamata in~\cite{Kawamata}.
He used a completely different approach based on the study of deformations of so-called simple collections of reflexive sheaves.
In particular, he obtained decompositions of the above type for $\P(1,1,n)$ and $\P(1,2,3)$,
see Examples~\ref{Example:P11n} and~\ref{Example:P123}.
On the other hand, in~\cite{K17} a semiorthogonal decomposition of the same type
was constructed for any normal sextic del Pezzo surface
(there are six isomorphism classes of such, four of them, including~$\P(1,2,3)$ are toric, and two are non-toric).

\subsection{Descent and adherence}

Now let us explain our approach and results in more detail.
Let $X$ be a normal projective surface \emph{with rational singularities}
over an algebraically closed field $\Bbbk$ of characteristic zero and let
\begin{equation*}
\pi \colon \tX \to X
\end{equation*}
be a resolution of $X$.
Note that under this assumption the exceptional locus of $\pi$ is a disjoint union of trees
of smooth rational curves whose intersection matrix is negative definite.
The first result of this paper is a ``descent procedure'' that allows to construct a semiorthogonal decomposition of~$X$
from a semiorthogonal decomposition of~$\tX$ satisfying a certain compatibility condition.

To be more precise, we say (Definition~\ref{definition: compatible}) that a semiorthogonal decomposition
\begin{equation}
\label{intro:dbtx}
\Db(\tX) = \langle \tcA_1,\tcA_2,\dots,\tcA_n \rangle
\end{equation}
is {\sf compatible with $\pi$} if for every component $E$ of the exceptional divisor of $\pi$ (so that~$E$ is a smooth rational curve)
the sheaf $\cO_{E}(-1)$ is contained in one of the components of the decomposition, i.e., $\cO_E(-1) \in \tcA_i$ for some $i$.
Under this assumption we show (Theorem~\ref{theorem:dbx}) that the categories $\cA_i := \pi_*(\tcA_i)$ give a semiorthogonal decomposition
\begin{equation}
\label{intro:dbx}
\Db(X) = \langle \cA_1,\cA_2,\dots,\cA_n \rangle.
\end{equation}
Furthermore, in some cases (for instance, when $X$ is Gorenstein) we prove that the categories $\cA_i^\perf = \cA_i \cap \Dperf(X)$
give a semiorthogonal decomposition
\begin{equation*}
\Dperf(X) = \langle \cA_1^\perf,\cA_2^\perf,\dots,\cA_n^\perf \rangle
\end{equation*}
of the category of perfect complexes on $X$.

The next question we address is an explicit description of each component~$\cA_i$
of the induced semiorthogonal decomposition of~$\Db(X)$, which we provide under some additional hypotheses.
Let $X$ be a normal projective surface satisfying~\eqref{eq:pg-q-x}
with only \emph{cyclic quotient singularities}
and let $\pi \colon \tX \to X$ be its \emph{minimal} resolution.
Let $x_1,\dots,x_n \in X$ be a collection of points such that $\Sing(X) \subset \{x_1,\dots,x_n\}$.
Let $E_{i,1},\dots,E_{i,m_i}$ be the irreducible divisorial components of $\pi^{-1}(x_i)$;
this is a (possibly empty) chain of smooth rational curves.

Our main hypothesis is that $\Db(\tX)$ has a semiorthogonal decomposition~\eqref{intro:dbtx} in which each component
%
%
$\tcA_i$ is {\sf adherent to the chain of curves~$\{E_{i,p}\}_{p=1}^{m_i}$}
(see Definition~\ref{definition: adherent} for a slightly generalized version of this condition),
i.e., morally, the component~$\tcA_i$ is the smallest admissible subcategory in~$\Db(\tX)$
that contains all~$\cO_{E_{i,p}}(-1)$ for \mbox{$1 \le p \le m_i$}.
Explicitly this means that~$\tcA_i$ is generated by all $\cO_{E_{i,p}}(-1)$ and an additional line bundle~$\cL_{i,0}$
that has prescribed intersections with $E_{i,p}$ (depending on $d_{i,p} {} = -E_{i,p}^2$), which guarantees admissibility of $\tcA_i$.
In fact, the category $\tcA_i$ can also be generated by an exceptional collection of line bundles $\cL_{i,p}$, \mbox{$0 \le p \le m_i$},
where for $p \ge 1$ we define~$\cL_{i,p} = \cL_{i,p-1}(E_{i,p})$.
Note that any line bundle on $\wt{X}$ is exceptional
by~\eqref{eq:pg-q-x} and rationality of singularities of $X$.

A result of Hille and Ploog (Theorem~\ref{theorem:hille-ploog}) implies that
under these assumptions the category $\tcA_i$ is equivalent to the derived category of modules
over a certain finite dimensional algebra $\Lambda_i$ of finite global dimension.
The algebra $\Lambda_i$ has exactly $m_i + 1$ simple objects $S_{i,p}$, $0 \le p \le m_i$,
which under the above equivalence correspond to the line bundle $\cL_{i,0}$
and the sheaves $\cO_{E_{i,p}}(-1)$, \mbox{$1 \le p \le m_i$}, respectively.
Denoting by $P_{i,p}$ the corresponding indecomposable projective $\Lambda_i$-modules, we define a finite dimensional algebra
\begin{equation*}
K_i := \End_{\Lambda_i}(P_{i,0}).
\end{equation*}
These algebras, studied by Kalck and Karmazyn in~\cite{KK}, are finite dimensional local noncommutative monomial algebras
(see Lemma~\ref{lemma:kk} for an explicit description) that only depend on the type of the cyclic quotient singularity $(X,x_i)$.
Our second main result (Theorem~\ref{theorem:sod-cai}) is
that under the above assumptions we have an equivalence
\begin{equation*}
\cA_i \cong \Db(K_i\md)
\end{equation*}
of the component $\cA_i$ of the induced semiorthogonal decomposition~\eqref{intro:dbx} of $\Db(X)$
with the derived category of finite dimensional modules over the algebra $K_i$.
In fact, $\Db(\Lambda_i\md)$ is
a categorical resolution of singularities of $\Db(K_i\md)$
in the sense of~\cite{KL}, and moreover, the algebra~$\Lambda_i$
can be recovered from $K_i$ via the Auslander construction.

Combining the above results together we show (Corollary~\ref{corollary:dbx}) that
if $X$ is a normal projective surface satisfying~\eqref{eq:pg-q-x} with cyclic quotient singularities,
$\pi \colon \tX \to X$ is its minimal resolution, and there is a semiorthogonal decomposition~\eqref{intro:dbtx}
in which every component is adherent to a connected component of the exceptional divisor of $\pi$,
then the induced semiorthogonal decomposition~\eqref{intro:dbx} of $\Db(X)$ has the form
\begin{equation}
\label{intro:dbx-kmod}
\Db(X) = \langle \Db(K_1\md), \Db(K_2\md), \dots, \Db(K_n\md) \rangle,
\end{equation}
where all $K_i$ are the Kalck--Karmazyn algebras.
Furthermore, if an additional crepancy condition is satisfied, e.g., if $X$ is Gorenstein,
we check that~\eqref{intro:dbx-kmod} also induces a semiorthogonal decomposition of the category of perfect complexes
\begin{equation}
\label{intro:dperfx-kmod}
\Dperf(X) = \langle \Dperf(K_1\md), \Dperf(K_2\md), \dots, \Dperf(K_n\md) \rangle.
\end{equation}

\subsection{Brauer obstruction}

The next important observation we make is that the Brauer group $\Br(X)$ provides an obstruction to existence of such decompositions.
Indeed, if we have~\eqref{intro:dbx-kmod}, it is easy to see (Lemma~\ref{lemma:semiorth-torsion})
that the Grothendieck group $\rG_0(X)$ of the category $\Db(X)$ is torsion-free.
On the other hand, assuming that $X$ is rational, it is easy to show that $\rG_0(X)_\tors = \Cl(X)_\tors$,
where $\Cl(X)$ is the class group of Weil divisors, and that this group is trivial if and only if $\Br(X) = 0$.
To be more precise, there is a natural isomorphism
\begin{equation*}
\Br(X) \cong \Ext^1(\Cl(X),\Z) \cong \Ext^1(\rG_0(X),\Z),
\end{equation*}
(Proposition \ref{prop:torsion}), and so the Brauer group $\Br(X)$ of the surface $X$ provides an obstruction
to the existence of~\eqref{intro:dbx-kmod}.
One of the simplest examples of a surface $X$ with~$\Br(X) \ne 0$ is the toric cubic surface
\begin{equation}
\label{intro:cubic}
X = \P^2 / \bmu_3 \cong \{ z_0^3 - z_1z_2z_3 = 0 \} \subset \P^3
\end{equation}
that has three $A_2$-singularities, $\Cl(X) \cong \Z \oplus \Z/3$, and $\Br(X) \cong \Z/3$.
Consequently, its minimal resolution does not admit a semiorthogonal decomposition
with components adherent to the components of the exceptional divisor, and $\Db(X)$ does not have~\eqref{intro:dbx-kmod}.

More generally, one can see that for an arbitrary projective toric surface~$X$,
the Brauer group $\Br(X)$ is a finite cyclic group of order $r$,
where $r$ is the greatest common divisor of the orders $r_i$ of the toric points on $X$
(see Lemma~\ref{lemma:toric-br} and Remark~\ref{remark:toric-gcd}),
so that vanishing of $\Br(X)$ is equivalent to the orders $r_1, \dots, r_n$ being coprime.
In particular any weighted projective plane $\P(w_1,w_2,w_3)$ has vanishing Brauer group.

In Section~\ref{section:toric-surfaces} we show that the Brauer group provides \emph{the only} obstruction
to the existence of~\eqref{intro:dbx-kmod} in the toric case.
Specifically, we construct decompositions~\eqref{intro:dbx-kmod}
for any projective toric surface $X$ satisfying $\Br(X) = 0$.
These results can be considered as a generalization to singular projective toric surfaces
of a standard method of constructing
exceptional collections on smooth projective toric surfaces by iterative twisting of a line bundle
by the sequence of boundary divisors~\cite{Hille}.

For instance, for the weighted projective plane $\P(1,2,3)$ we obtain the semiorthogonal decomposition
\begin{align*}
\Db(\P(1,2,3))  &= \textstyle \Big\langle
\Db(\Bbbk),
\Db\left(\frac{\Bbbk[z]}{z^2}\right),
\Db\left(\frac{\Bbbk[z]}{z^3}\right)
\Big\rangle.
\end{align*}
Since $\P(1,2,3)$ is Gorenstein we also obtain a similar semiorthogonal decomposition of~$\Dperf(\P(1,2,3))$.
For a non-Gorenstein example, let us consider $\P(2,3,11)$, where we obtain
\begin{align*}
\Db(\P(2,3,11)) &= \textstyle \Big\langle
\Db\left(\frac{\Bbbk[z]}{z^2}\right),
\Db\left(\frac{\Bbbk[z_1,z_2]}{(z_1^2,z_1z_2,z_2^2)}\right),
\Db\left(\frac{\Bbbk\langle z_1,z_2 \rangle}{(z_1^4,z_1z_2,z_2^2z_1^2,z_2^3)}\right)
\Big\rangle.
\end{align*}
This time it does not induce a semiorthogonal decomposition of~$\Dperf(\P(2,3,11))$.

\subsection{Twisted adherence and twisted derived categories}

As we already observed, for surfaces~$X$ with $\Br(X) \ne 0$ there is
no semiorthogonal decomposition~\eqref{intro:dbx-kmod} with local finite-dimensional algebras $K_i$.
However, there are two things we can say.

First of all, in the toric case, resolving any of the singular points of $X$ will produce
a toric surface with trivial Brauer group, so that its structure can be analyzed with our methods.

On the deeper level, the same Brauer group that obstructs the existence of~\eqref{intro:dbx-kmod},
can be incorporated into the problem providing a generalization to the results described above.
Namely for every $\beta \in \Br(X)$ we can ask about semiorthogonal
decompositions for the \emph{twisted derived category} $\Db(X, \beta)$.
To analyze these we consider semiorthogonal decompositions of $\Db(\tX)$
with components that are adherent to the components of the exceptional divisor
up to a line bundle twist (individual for each component).
We call such decompositions {\sf twisted adherent},
show that they correspond to some explicit elements in the Brauer group~$\Br(X)$,
and prove that they induce semiorthogonal decompositions of the twisted derived category of~$X$:
\begin{equation}
\label{intro:dbx-twisted-kmod}
\Db(X,\beta) = \langle \Db(K_1\md), \Db(K_2\md), \dots, \Db(K_n\md) \rangle,
\end{equation}
(see Theorem~\ref{theorem:dbx-twisted}),
where $K_i$ are finite-dimensional algebras constructed from singularities of $X$ in the same way as in the untwisted case.
As usual, under the additional crepancy assumption there is a decomposition for~$\Dperf(X,\beta)$ analogous
to~\eqref{intro:dperfx-kmod}.
We also obtain a description of the Grothendieck group~$\rG_0(X,\beta)$ of $\Db(X,\beta)$ (Proposition~\ref{prop:g0-quotient}),
and check in Proposition~\ref{prop:g0-twisted} that
\begin{equation*}
\Ext^1(\rG_0(X,\beta),\Z) \cong \Br(X)/\langle \beta \rangle,
\end{equation*}
so the existence of~\eqref{intro:dbx-twisted-kmod} implies that $\beta$ is a generator of $\Br(X)$.

For example, for the cubic surface $X$ defined by~\eqref{intro:cubic} we obtain
\begin{equation*}
\Db(X,\beta) = \Big\langle
\Db\left(\tfrac{\Bbbk[z]}{z^3}\right), \Db\left(\tfrac{\Bbbk[z]}{z^3}\right), \Db\left(\tfrac{\Bbbk[z]}{z^3}\right)
\Big\rangle,
\end{equation*}
where $\beta$ is a generator of the Brauer group $\Br(X) \cong \Z/3$,
and this decomposition induces one for $\Dperf(X,\beta)$ (again, since $X$ is Gorenstein).

Let us point out that we do not have an answer to the following question:
is it true that decomposition \eqref{intro:dbx-twisted-kmod} exists for \emph{any} generator $\beta \in \Br(X)$?
In Lemma~\ref{lemma:delta-functions} we explicitly present the set of all generators $\beta$ of the Brauer group $\Br(X)$
of a toric surface for which we obtain a decomposition~\eqref{intro:dbx-twisted-kmod}
using twisted adherent exceptional collections on $\wt{X}$.
In particular, we check that in the Gorenstein case our construction produces just one $\beta$,
up to sign.

\subsection{Generators}

Finally, we discuss the relation of our approach to the one developed by Kawamata in~\cite{Kawamata}.
We show that our construction in many cases also produces some natural reflexive sheaves on $X$
and their \emph{versal noncommutative thickenings} (in the terminology of Kawamata).
We generalize Kawamata's results to all toric surfaces with torsion-free class group,
see Proposition~\ref{proposition: toric maximal extension}
and illustrate the result in the case on an arbitrary weighted projective plane (Example~\ref{ex:pabc}).

\subsubsection*{Structure of the paper}

In Section~\ref{section:semiorth} we explain how to induce a semiorthogonal decomposition
for a surface with rational singularities from a compatible semiorthogonal decomposition of its resolution.
The main result of this section is Theorem~\ref{theorem:dbx}.

In Section~\ref{section:components} we introduce the notion of (twisted) adherence,
define the algebras of Hille--Ploog and Kalck--Karmazyn, and under the adherence assumption
identify in Theorem~\ref{theorem:sod-cai} and Corollary~\ref{corollary:dbx}
the components of the induced decomposition of~$\Db(X)$.

In Section~\ref{section:brauer} we discuss the Brauer group of a rational surface $X$,
construct the Brauer class~$\beta$ corresponding to a twisted adherent semiorthogonal
decomposition of its resolution~$\tX$,
and describe in Theorem~\ref{theorem:dbx-twisted} the induced semiorthogonal
decomposition of the twisted derived category $\Db(X,\beta)$.

In Section~\ref{section:toric-surfaces} we apply our constructions in the case of toric surfaces.
The main results of this section are Theorem~\ref{theorem:toric} and Corollary~\ref{cor:toric}.

In Section~\ref{section:reflexive} we discuss the relation of our results to the approach of Kawamata,
in particular we investigate under which conditions generators of the constructed above components $\cA_i \subset \Db(X)$
of~\eqref{intro:dbx} are reflexive or locally free sheaves on $X$.
Moreover, we describe explicitly the reflexive generators of the components
of the semiorthogonal decomposition for any weighted projective plane $\P(w_1,w_2,w_3)$.

\subsubsection*{Notation and conventions}

We work over an algebraically closed field $\Bbbk$ of characteristic zero.
All varieties and categories are assumed to be $\Bbbk$-linear.
All surfaces are assumed to be irreducible.

For a $\Bbbk$-scheme $X$ we denote by $\Db(X)$ the bounded derived category of coherent sheaves on $X$
and by $\Dperf(X)$ the category of perfect complexes on $X$, i.e., the full subcategory of $\Db(X)$
consisting of objects that are locally quasi-isomorphic to finite complexes of locally free sheaves of finite rank.

Similarly, for a $\Bbbk$-algebra~$R$ we denote by $R\md$ the category of finitely generated right $R$-modules,
by~$\Db(R\md)$ its bounded derived category,
and by~$\Dperf(R\md)$ the category of bounded complexes of finitely generated projective $R$-modules.

We denote by $\tau^{\le t}$ and $\tau^{\ge t}$ the \emph{canonical truncation functors} at degree~$t$.
For a set~$\{F_i\}$ of objects of a triangulated category $\cT$
we denote by $\langle \{F_i\} \rangle$ the minimal triangulated subcategory of~$\cT$ containing all $F_i$,
and by $\langle \{F_i\} \rangle^{\oplus}$ the minimal triangulated subcategory of $\cT$
closed under all direct sums that exist in $\cT$ and containing all $F_i$.
For a subcategory $\cA \subset \cT$ we denote by $\cA^\perp$ and~${}^\perp\cA$ its \emph{right and left orthogonals}:
\begin{equation*}
\cA^\perp = \{ F \in \cT \mid \Ext^\bullet(\cA,F) = 0 \},
\qquad
{}^\perp\cA = \{ F \in \cT \mid \Ext^\bullet(F,\cA) = 0 \}.
\end{equation*}
For a morphism $f$ we denote by $f_*$ the \emph{derived} pushforward functor, and by $f^*$ the \emph{derived} pullback.
Similarly, $\otimes$ is used for the \emph{derived} tensor product.
If we need underived functors, we use notation $R^0f_*$ and $L_0f^*$ respectively.

For an abelian group $A$ we denote by $A_{\tors}$ its torsion part.

\subsubsection*{Acknowledgements}

We thank Alexander Efimov, Sergey Gorchinskiy, Dima Orlov, Yuri Prokhorov, and Damiano Testa for useful discussions,
and Agnieszka Bodzenta, Martin Kalck and Mykola Shamaiev for their comments about the paper.

J.K. was supported by EPSRC grant EP/M017516/2.
A.K. was partially supported by the HSE University Basic Research Program and
the Russian Academic Excellence Project~\hbox{`5-100'}.
E.S. was partially supported by Laboratory of Mirror Symmetry NRU HSE, RF government grant, ag. N~14.641.31.0001.

\section{Inducing a semiorthogonal decomposition from a resolution}
\label{section:semiorth}

The main goal of this section is to set up a framework in which  a semiorthogonal decomposition of the bounded derived category
of a singular surface $X$ can be constructed from a semiorthogonal decomposition of the bounded derived category of its resolution.
For this we use the approach developed in~\cite{K17}, with a suitable modification.
The main difference between the situation of~\cite{K17} and the situation considered here,
is that~$X$ is not necessarily Gorenstein, and the resolution is not necessarily crepant;
because of that we have to modify some arguments of~\cite{K17} that were using these assumptions.

\subsection{Resolutions of rational surface singularities}
\label{subsection:resolution}

Let $X$ be a normal surface, and let $\pi\colon \tX \to X$ be its resolution of singularities.
We assume that $X$ has {\sf rational singularities}, i.e. we have an isomorphism
$\pi_* \OO_{\wt{X}} \simeq \OO_X$.
In this case every irreducible component of the exceptional divisor of $\pi$ is a smooth rational curve,
and every connected component is a tree of rational curves with transverse intersections;
see e.g. \cite[Lemma 1.3]{BrieskornKomplexerFlachen}.

Below we discuss what such a resolution does on the level of derived categories.
Most conveniently this is expressed on the level of bounded from above categories of
coherent sheaves $\Dm(\tX)$ and $\Dm(X)$.
This could be also done on the level of the unbounded derived category,
but we prefer to work with~$\Dm$.

The derived pushforward and pullback provide an adjoint pair of functors
\begin{equation*}
\pi_* \colon \Dm(\tX) \to \Dm(X)
\qquad\text{and}\qquad
\pi^* \colon \Dm(X) \to \Dm(\tX),
\end{equation*}
and, as $X$ has rational singularities, by projection formula we have
\begin{equation}
\label{eq:pi-pi}
\pi_* \circ \pi^* \cong \id_{\Dm(X)}.
\end{equation}
Consequently, we have a semiorthogonal decomposition
\begin{equation}
\label{eq:dbtx-dbx}
\Dm(\tX) = \langle \Ker \pi_*, \pi^*(\Dm(X)) \rangle.
\end{equation}
Using~\cite[Lemma 3.1]{BridgelandFlops}
and~\cite[Theorem~6.13]{BB-new}
the category $\Ker \pi_*$ can be described as follows.

\begin{lemma}[{\cite[Lemma~2.3]{K17}}]
\label{lemma:ker-pi}
Let $X$ be a normal surface with rational singularities and let $\pi \colon \tX \to X$ be its resolution.
An object $\cF \in \Dm(\tX)$ is contained in~$\Ker \pi_*$ if and only if every cohomology sheaf $\cH^i(\cF)$
is an iterated extension of the sheaves~$\cO_E(-1)$ for irreducible exceptional divisors $E$ of $\pi$.
\end{lemma}

However, we are mostly interested in the bounded derived category and the category of perfect complexes,
which are not preserved by the adjoint pair of functors $(\pi^*,\pi_*)$;
in fact $\pi^*$ only preserves the category of perfect complexes, while $\pi_*$ only preserves the bounded derived category.
On the level of these categories we have no semiorthogonal decomposition analogous to~\eqref{eq:dbtx-dbx},
but using small dimension effects we can deduce many results for them.


\begin{lemma}[{\cite[Lemma~2.4]{K17}}]
\label{lemma:pis-tau}
Let $X$ be a normal surface with rational singularities and let $\pi \colon \tX \to X$ be its resolution.
If $\cG$ is concentrated in degrees $\ge k$, then
\begin{equation*}
\pi_*(\tau^{\le k-2} \pi^* \cG) = 0
\qquad\text{and}\qquad
\cG \cong \pi_*(\tau^{\ge k-1}\pi^*(\cG)).
\end{equation*}
\end{lemma}

The above implies that $\pi_*$ is essentially surjective on bounded derived categories.

\begin{corollary}[{\cite[Corollary~2.5]{K17}}]
\label{lemma:im-pi}
Under the assumptions of Lemma~\textup{\ref{lemma:pis-tau}}
for any object~\mbox{$\cG \in \Db(X)$} there exists $\cF \in \Db(\tX)$ such that $\cG \cong \pi_*(\cF)$.
\end{corollary}

Lemma \ref{lemma:descent} below is very useful, in particular we will often use the (1)~$\implies$~(3) and~(2)~$\implies$~(3) implications
to descend vector bundles from~$\wt{X}$ to~$X$.
In the proof we need the following standard result in commutative algebra, (cf.~\cite[X.3, Proposition~4]{Bou}).

\begin{lemma}
\label{lemma:comm-alg}
Let $(A,\fm)$ be a noetherian local ring and $\Bbbk = A/\fm$.
If $M \in \Dm(A)$ is a bounded above complex of finitely generated $A$-modules,
then $M$ is a perfect complex if and only if $\Ext^t_A(M, \Bbbk) = 0$ for $|t| \gg 0$.
\end{lemma}
\begin{proof}
Let $F^\bullet$ be the \emph{minimal} free resolution of $M$, i.e.,
a bounded above complex of finitely generated free $A$-modules quasiisomorphic to $M$
such that all its differentials are zero modulo $\fm$
(such resolution can be constructed by a standard procedure, see~\cite[\S1.3]{BH} for the case when $M$ is a module).
Then the complex $\Hom_A(F^\bullet, \Bbbk)$ has zero differentials so that we have
\begin{equation*}
\Ext^t_A(M, \Bbbk) \simeq (F^{-t} \otimes_A \Bbbk)^\vee.
\end{equation*}
Since $F^{-t}$ is free, $F^{-t} \otimes_A \Bbbk = 0$ if and only if $F^{-t} = 0$,
hence $\Ext^t_A(M, \Bbbk) = 0$ for~$|t| \gg 0$ if and only if $F^\bullet$ is bounded,
and since $F^{-t}$ are finitely generated and free this holds if and only if $M$ is perfect.
\end{proof}

Below $\cF\vert_E$ stands for the derived pullback of $\cF$ along the embedding $E \hookrightarrow \tX$.

\begin{lemma}
\label{lemma:descent}
Let $X$ be a normal surface with rational singularities and let~$\pi \colon \tX \to X$ be its resolution.
Let $\cF \in \Db(\tX)$.
The following properties are equivalent:
\begin{enumerate}
\item
for any irreducible exceptional divisor $E$ of $\pi$ one has $\cF\vert_E \in \langle \cO_E \rangle$;
\item
for any irreducible exceptional divisor $E$ of $\pi$ one has $\Ext^\bullet(\cF\vert_E, \cO_E(-1)) = 0$;
\item
there exists $\cG \in \Dperf(X)$ such that $\cF \cong \pi^*\cG$;
\item
one has $\pi_*\cF \in \Dperf(X)$ and $\cF \cong \pi^*(\pi_*\cF)$.
\end{enumerate}
If additionally $\cF$ is a pure sheaf, or a locally free sheaf, then so is $\pi_*\cF$.
\end{lemma}

\begin{proof}
We prove (1) $\implies$ (2) $\implies$ (3) $\implies$ (4) $\implies$ (1).

(1) $\implies$ (2) This implication is trivial, since $E \simeq \P^1$ and
$\Ext^\bullet_{\P^1}(\OO, \OO(-1)) = 0$.

(2) $\implies$ (3)
It follows from~\eqref{eq:dbtx-dbx} and Lemma~\ref{lemma:ker-pi}
that $\cF$ belongs to the image of the functor~$\pi^* \colon \Dm(X) \to \Dm(\tX)$,
that is~$\cF \cong \pi^*\cG$ for some $\cG \in \Dm(X)$.
Let $x \in X$ be a point and let $\tx \in \tX$ be a point over $x$.
Then $\pi_*\cO_\tx \cong \cO_x$, hence by adjunction
\begin{equation*}
\Ext^\bullet(\cG,\cO_x) \cong
\Ext^\bullet(\cG,\pi_*\cO_\tx) \cong
\Ext^\bullet(\pi^*\cG,\cO_\tx) \cong
\Ext^\bullet(\cF,\cO_\tx),
\end{equation*}
which is finite-dimensional by smoothness of $\tX$.
By Lemma~\ref{lemma:comm-alg}, $\cG$ is perfect in a neighborhood of $x$.
Since this holds for each point of $X$, we conclude that $\cG$ is perfect.

(3) $\implies$ (4) This is clear since if $\cF \cong \pi^* \cG$ then
\begin{equation*}
\pi_* \cF \cong \pi_* \pi^* \cG \cong \cG.
\end{equation*}

(4) $\implies$ (1) Let $\cG = \pi_* \cF$. Then $\cF \cong \pi^* \cG$ by assumption.
The restriction $\cF\vert_E \cong \pi^*(\cG)\vert_E$ is isomorphic
to $p_E^*(\cG\vert_x)$ where $p_E\colon E \to \Spec(\Bbbk)$ is the structure morphism
and $x  = \pi(E) \in X$.
We have $p_E^*(\Db(\Bbbk)) {}={} \langle \OO_E \rangle$ and the result follows.

Now assume that $\cF \in \Db(\tX)$ is an object for which all the equivalent conditions hold.
Let $\tx \in \tX$ and $x = \pi(\tx)$.
For every $t \in \Z$ by adjunction and property (4) we have
\begin{multline*}
\Hom(\pi_*\cF,\cO_x[t]) =
\Hom(\pi_*\cF,\pi_*\cO_\tx[t]) \cong
\Hom(\pi^*\pi_*\cF,\cO_\tx[t]) \cong
\Hom(\cF,\cO_\tx[t]).
\end{multline*}
If $\cF$ is a pure sheaf, then the $\Hom$-space in the right-hand side is zero for all $\tx \in \tX$ and $t < 0$.
Therefore, the left side is zero for all $x \in X$ (since $\pi$ is surjective) and~\mbox{$t < 0$},
hence the complex $\pi_*\cF$ is concentrated in nonpositive degrees.
But since $\pi_*$ is left exact, it is a pure sheaf.

If $\cF$ is locally free, then the $\Hom$-space in the right-hand side is also zero for all~\mbox{$\tx \in \tX$} and~\mbox{$t > 0$}.
Therefore, the left side is zero for all $x \in X$ and $t > 0$ as well,
hence $\pi_*\cF$ is locally free by \cite[X.3, Proposition~4]{Bou}.
\end{proof}

We will also use the following corollary.

\begin{corollary}\label{cor:descent}
Let $\wt{X} \to X$ be a resolution of a normal surface with rational singularities.
If $\cG \in \Dm(X)$ and $\pi^* \cG \in \Db(\wt{X})$, then
$\cG \in \Dperf(X)$.
\end{corollary}
\begin{proof}
Set $\cF = \pi^*\cG$.
By~\eqref{eq:dbtx-dbx} and Lemma~\ref{lemma:ker-pi} the property~(2) of Lemma~\ref{lemma:descent} is satisfied for $\cF$.
Therefore, by property~(4) the object $\pi_*\cF \cong \pi_*\pi^*\cG \cong \cG$ is perfect.
\end{proof}

Later we will state an analog of Lemma~\ref{lemma:descent}
for categorical resolutions of finite-dimensional algebras (see Lemma~\ref{lemma:descent-algebraic})
and twisted derived categories (see Lemma~\ref{lemma:descent-twisted}).

\subsection{Compatibility with contraction}
\label{subsection:exceptional-collection-tx}

Let $X$ be an irreducible normal surface with rational singularities.
Let $\pi \colon \tX \to X$ be its resolution of singularities.
Let $D$ be the exceptional locus of $\pi$.
Recall that each irreducible component of $D$ is a smooth rational curve.

\begin{definition} \label{definition: compatible}
A semiorthogonal decomposition $\Db(\tX) = \langle \tcA_1, \tcA_2, \dots, \tcA_n \rangle$ is
{\sf compatible with the contraction $\pi$} if for each irreducible component $E$ of $D$ one has
\begin{equation*}
\cO_E(-1) \in \tcA_i
\end{equation*}
for one of the components $\tcA_i$ of the decomposition.
\end{definition}

Note that for an irreducible component $E$ of $D$ the component $\tcA_i$ to which $\cO_E(-1)$ belongs is uniquely determined.
Let
\begin{equation}
\label{eq:def-di}
D_i := \bigcup \left\{ E \mid \cO_E(-1) \in \tcA_i \right\}
\end{equation}
be the union of those irreducible components $E \subset D$ for which $\cO_E(-1)$ belongs to $\tcA_i$.

We will need the following simple observation.

\begin{lemma}
\label{lemma:e-e-orthogonal}
Let $\Db(\tX) = \langle \tcA_1, \tcA_2, \dots, \tcA_n \rangle$ be a semiorthogonal decomposition.
If~$E$ and~$E'$ are components of $D$ such that
\begin{equation*}
\cO_E(-1) \in \tcA_i
\qquad\text{and}\qquad
\cO_{E'}(-1) \in \tcA_{i'}
\qquad\text{with $i \ne i'$},
\end{equation*}
then $E \cap E' = \varnothing$ and so $\cO_E(-1)$ and $\cO_{E'}(-1)$ are completely orthogonal.
In particular,
\begin{equation*}
D_i \cap D_{i'} = \varnothing
\qquad
\text{for $i \ne i'$.}
\end{equation*}

\end{lemma}
\begin{proof}
Let $i' > i$.
Then $\Ext^\bullet(\cO_{E'}(-1),\cO_E(-1)) = 0$ by semiorthogonality of~$\tcA_i$ and~$\tcA_{i'}$.
Since $E$ and $E'$ are irreducible curves on a smooth surface $\tX$,
an easy computation shows
that in fact
\begin{equation*}
E \cap E' = \varnothing.
\end{equation*}
But then $\Ext^\bullet(\cO_{E}(-1),\cO_{E'}(-1)) = 0$ as well.
\end{proof}

Thus, a semiorthogonal decomposition compatible with $\pi$ induces a decomposition
\begin{equation}
\label{eq:decomposition-d}
D = D_1 \sqcup \dots \sqcup D_n
\end{equation}
of the exceptional divisor $D$ of $\pi$ into $n$ pairwise disjoint components,
where $n$ is the number of components in the semiorthogonal decomposition (some of $D_i$ may be empty).

Recall that a morphism $\pi \colon \tX \to X$ is {\sf crepant} if the canonical line bundle $\cO_\tX(K_\tX)$
is isomorphic to a pullback from $X$, i.e., by Lemma~\ref{lemma:descent} if $K_\tX \cdot E = 0$
for each irreducible component~$E$ of the exceptional divisor $D$ of $\pi$.
This condition can be reformulated in many ways.

\begin{lemma}
\label{lemma:crepancy}
Let $(X,x)$ be an isolated rational surface singularity and let $\pi\colon \wt{X} \to X$ be its resolution.
The morphism $\pi$ is crepant if and only if $(X,x)$ is Gorenstein and $\pi$ is its minimal resolution.
Furthermore, in this case $K_{\wt X} = \pi^*(K_X)$.
\end{lemma}

\begin{proof}
Assume that $\pi$ is crepant.
By Lemma~\ref{lemma:descent} we have $\cO_\tX(K_\tX) \cong \pi^*\cL$ for a line bundle~$\cL$ on $X$.
Furthermore, if $D$ is the exceptional locus of $\pi$, we have
\begin{equation*}
\cO(K_{X\setminus x}) = \cO(K_{\tX \setminus D}) {} \cong {} \pi^*\cL\vert_{\tX \setminus D} = \cL\vert_{X \setminus x},
\end{equation*}
hence $\cO(K_X) \simeq \cL$ and $K_X$ is a Cartier divisor, so $x$ is Gorenstein.
It also follows that~$K_{\wt X} = \pi^*(K_X)$.
Furthermore, if $\pi$ is not minimal, there exists a $(-1)$-component~$E$ of~$D$.
Then by adjunction formula $K_\tX \cdot E = -1$, hence $\pi$ is not crepant.
This proves one direction.

For the other direction just note that if $x$ is Gorenstein
then the singularity $(X,x)$ is Du Val by~\cite[Corollary~5.24 and Theorem~4.20]{KM};
in particular it has a crepant resolution
(see~\cite[Definition~4.24]{KM}).
Since crepant resolution is minimal and the minimal resolution is unique,
we conclude that $\pi$ is crepant.
\end{proof}

\begin{definition}
\label{definition:crepancy}
Assume that a decomposition~\eqref{eq:decomposition-d} of the exceptional divisor of a birational morphism $\pi$ is given.
We will say that $\pi$ is {\sf crepant along $D_i$} if $K_\tX \cdot E = 0$ for each irreducible component $E$ of $D_i$.
\end{definition}

According to Lemma~\ref{lemma:crepancy}, when~$X$ is a surface with rational singularities
its minimal resolution $\pi$ is crepant along a connected component $D'$ of the exceptional divisor $D$ of $X$
if and only if the point \mbox{$x = \pi(D') \in X$} is Gorenstein.

In the next lemma we apply Serre duality on $\wt{X}$ to rotate the components of the
semiorthogonal decomposition; for that we need to assume that $X$ is projective.

\begin{lemma}
\label{lemma:twisted-compatible}
Let $X$ be a normal projective surface and let $\pi \colon \tX \to X$ be a resolution of singularities.
Assume that $\Db(\tX) = \langle \tcA_1, \tcA_2, \dots, \tcA_n \rangle$ is a semiorthogonal decomposition.
For each $1 \le k \le n$ we have semiorthogonal decompositions
\begin{align}
\label{eq:dbtx-ktx}
\begin{aligned}
\Db(\tX) &= \langle \tcA_{{k+1}}(K_\tX), \dots, \tcA_n(K_\tX), \tcA_1, \dots, \tcA_{{k}} \rangle,\\
\Db(\tX) &= \langle \tcA_{{k}}, \dots, \tcA_n, \tcA_1(-K_\tX), \dots, \tcA_{{k-1}}(-K_\tX) \rangle.
\end{aligned}
\end{align}
Assume further that the original decomposition is compatible with $\pi$,
and let~\eqref{eq:decomposition-d} be the induced decomposition of its exceptional divisor.

If $\pi$ is crepant along $D_j$ for all~\mbox{$j > k$} the first
decomposition in~\eqref{eq:dbtx-ktx} is compatible with $\pi$,
and if~$\pi$ is crepant along~$D_j$ for all $j < k$ then so is the second.
In both cases
the induced decomposition~\eqref{eq:decomposition-d} of the exceptional divisor is obtained from the original one by
an appropriate cyclic permutation of indices.
\end{lemma}
\begin{proof}
The fact that~\eqref{eq:dbtx-ktx} are semiorthogonal decompositions follows easily from Serre duality.
To prove compatibility with $\pi$ of the first of them,
let $E$ be an irreducible component of the exceptional divisor of $\pi$ and assume that $\cO_E(-1) \in \tcA_j$.
If $j \le {k}$ there is nothing to check.
If~$j \ge {k + 1}$, we twist the containment by $K_\tX$, and since $K_\tX \cdot E = 0$ (by crepancy of $\pi$ along $D_j$),
we conclude that $\cO_E(-1) \in \tcA_j(K_\tX)$.
Compatibility of the second semiorthogonal decomposition is proved analogously.
\end{proof}

Recall that a functor $\Phi \colon \cT_1 \to \cT_2$ between triangulated categories endowed with t-structures
{\sf has finite cohomological amplitude} if there is a pair of integers $a_- \le a_+$ such that for any $k_- \le k_+$ one has
\begin{equation*}
\Phi\Big(\cT_1^{[k_-,k_+]}\Big) \subset \cT_2^{[k_- + a_-, k_+ + a_+]},
\end{equation*}
where $\cT_1^{[k_-,k_+]}$ denotes the subcategory of $\cT_1$ consisting of objects
whose cohomology with respect to the t-structure is supported in degrees between $k_-$ and $k_+$,
and similarly for~$\cT_2^{[k_- + a_-, k_+ + a_+]}$.
Finiteness of cohomological amplitude
is a useful finiteness condition (see, e.g., \cite{kuznetsov2011base}).

Recall that a triangulated subcategory $\cA \subset \cT$ is called {\sf left} (resp. {\sf right}) {\sf admissible}
if the embedding functor of $\cA$ has left (resp.\ right) adjoint functor,
or equivalently there is a semiorthogonal
decomposition $\cT = \langle \cA, {}^\perp \cA \rangle$ (resp. $\cT = \langle \cA^\perp, \cA \rangle$).
If $\cA$ is both left and right admissible it is called {\sf admissible}.

The main result of this section is the next theorem.

\begin{theorem}
\label{theorem:dbx}
Let $X$ be a normal projective surface
with rational singularities
and let $\pi \colon \tX \to X$ be its resolution.
Assume that $\wt{X}$ admits
a semiorthogonal decomposition
\begin{equation}
\label{eq:dbtx}
\Db(\tX) = \langle \tcA_1, \tcA_2, \dots, \tcA_n \rangle
\end{equation}
compatible with $\pi$
and let~\eqref{eq:decomposition-d} be the induced decomposition of the exceptional divisor.

\noindent{$(i)$} There is a unique semiorthogonal decomposition
\begin{equation}
\label{eq:dbx}
\Db(X) = \langle \cA_1, \cA_2, \dots, \cA_n \rangle,
\end{equation}
with $\pi_*(\tcA_i) = \cA_i$.
The projection functors of~\eqref{eq:dbx} have finite cohomological amplitude.
Moreover, the functor $\pi_*\colon \tcA_i \to \cA_i$ induces an equivalence of $\cA_i$ with a Verdier quotient
\begin{equation*}
\cA_i \simeq \tcA_i / \langle \OO_{E}(-1) \rangle_{E \subset D_i},
\end{equation*}
where $E$ runs over the set of irreducible components of $D_i$.

\noindent{$(ii)$}
If $\pi$ is crepant along $D_j$ for $j > i$, $\cA_i$ is right admissible in~$\Db(X)$,
and if $\pi$ is crepant along $D_j$ for $j < i$, $\cA_i$ is left admissible in~$\Db(X)$.

In particular,
if $\pi$ is crepant, then all $\cA_i$ are admissible in $\Db(X)$.

\noindent{$(iii)$}
Setting
$\cA_i^\perf := \cA_i \cap \Dperf(X)$
we have
\begin{equation}
\label{eq:pis-aiperf}
\pi^*(\cA_i^\perf) \subset \langle \tcA_i, \tcA_{i+1} \cap \Ker \pi_*, \dots, \tcA_n \cap \Ker \pi_* \rangle.
\end{equation}

\noindent{$(iv)$}
If $\pi$ is crepant along $D_j$ for $j > i$, we have
\begin{equation*}
\pi^*(\cA_i^\perf) \subset \tcA_i,
\end{equation*}
Furthermore, if $\pi$ is crepant along $D_j$ for all $j \ge 2$, there is a semiorthogonal decomposition
of the category of perfect complexes
\begin{equation}
\label{eq:dpxperf}
\Dperf(X) = \langle \cA_1^\perf, \cA_2^\perf, \dots, \cA_n^\perf \rangle.
\end{equation}
Finally, if $\pi$ is crepant, then all components $\cA_i^\perf$ of \eqref{eq:dpxperf} are admissible.
\end{theorem}

\begin{remark}
In fact, a semiorthogonal decomposition of the category $\Db(X)$ for a projective scheme~$X$
always induces \emph{some} semiorthogonal decomposition of $\Dperf(X)$,
see Theorem~\ref{theorem:efimov}.
However, we are mostly interested in the case when the intersections of the components of $\Db(X)$ with $\Dperf(X)$
form a semiorthogonal decomposition; this is why we state a criterion for this.
\end{remark}

The proof of Theorem~\ref{theorem:dbx} takes Sections~\ref{subsection:dm-decomposition} and~\ref{subsection:db-decomposition}.
This proof is rather technical, and the reader not interested in technicalities can easily bypass it
and go directly to Section~\ref{section:components}.
Before we give the proof,
let us illustrate the theorem in an example,
at the same time motivating results in the further sections.

\begin{example}\label{ex:one-sing}
Let $X$ be a projective surface satisfying~\eqref{eq:pg-q-x} with a single cyclic quotient singularity of
type $\frac1{d}(1,1)$ (see Section~\ref{subsection:cyclic-singularity} for a discussion of cyclic quotient singularities),
and let $\pi \colon \wt{X} \to X$ be its minimal resolution, so that~$\pi$ contracts a
smooth rational curve $E \subset \wt{X}$
with self-intersection~\mbox{$-d$}.
Assume that there exists a line bundle $\cL \in \Pic(\wt{X})$ such that $\cL \cdot E = d - 1$
(this holds if and only if~$E$ is primitive in~$\Pic(\tX)$).
Under these assumptions $(\cL,\cL(E))$ is an exceptional pair on~$\wt{X}$.
Then we have a semiorthogonal decomposition
\begin{equation*}
\Db(\wt{X}) = \langle \tcA_1, \tcA_2 \rangle,
\qquad\text{where}\qquad
\tcA_1 = \langle \cL, \cL(E) \rangle
\text{\quad and\quad }
\tcA_2 = {}^\perp \tcA_1.
\end{equation*}
It is easy to see that $\OO_{E}(-1)$ is isomorphic to the cone of a morphism $\cL \to \cL(E)$,
hence it belongs to~$\tcA_1$, so that the semiorthogonal decomposition is compatible with~$\pi$.
In this case Theorem~\ref{theorem:dbx} gives a semiorthogonal decomposition
\begin{equation*}
\Db(X) = \langle \cA_1, \cA_2 \rangle
\end{equation*}
where $\cA_1 = \pi_*(\tcA_1) = \langle \pi_*(\cL) \rangle {} = \langle \pi_*(\cL(E)) \rangle$
(however, $\pi_*(\cL)$ is not exceptional)
and~$\cA_2 = \pi_*(\tcA_2)$.
Furthermore there is an induced decomposition for $\Dperf(X)$.

In Section~\ref{section:components} we
show that the above category~$\tcA_1$ is the simplest example of a category \emph{adherent} to $E$ and
explain how to describe the constructed category~$\cA_1$ explicitly.
In Section~\ref{section:reflexive} we study under which conditions is $\pi_*(\cL)$ a reflexive sheaf on $X$.

In Section~\ref{section:brauer} we show that
primitivity of $E$, i.e., the existence of a line bundle~$\cL$ as above,
is in general controlled by the Brauer group~$\Br(X)$ of the surface~$X$
(this follows from Remark~\ref{remark:brpic-general}).
See also~\cite[Example 5 in Chapter 4]{Srinivas} for an example of a projective rational surface with a single
$\frac14{(1,1)}$-point and the exceptional curve $E$ of the minimal resolution divisible by~2
(in this case $\Br(X) \cong \Z/(2)$).
Note also that if~$d$ is square-free, $E$ is always primitive.
\end{example}

\subsection{Decomposition of the bounded from above category}
\label{subsection:dm-decomposition}

We keep the assumptions and notation from the previous subsection.
Moreover, we assume that $X$ is projective and fix a semiorthogonal decomposition~\eqref{eq:dbtx} compatible with $\pi$.
Recall the decomposition~\eqref{eq:decomposition-d} of the exceptional divisor $D$ of $\pi$
with components $D_i$ defined by~\eqref{eq:def-di}.
For each $D_i$ we denote by~$E_{i,1}$, \dots, $E_{i,m_i}$ its irreducible components, so that
\begin{equation}
\label{eq:decomposition-di}
D_i = E_{i,1} \cup \dots \cup E_{i,m_i}.
\end{equation}
Recall that by definition of $D_i$ we have
\begin{equation}
\label{eq:oe-ai-other}
\cO_{E_{i,p}}(-1) \in \tcA_i
\qquad
\text{for each $1 \le p \le m_i$}.
\end{equation}
Note that a combination of~\eqref{eq:decomposition-d} and~\eqref{eq:decomposition-di} shows that
any irreducible component of the exceptional divisor of $\pi$ is equal to one of $E_{i,p}$.

Since $\tX$ is smooth and projective, every component $\tcA_i$ of $\Db(\tX)$ is admissible,
hence~\eqref{eq:dbtx} is a strong semiorthogonal decomposition in the sense of~\cite[Definition~2.6]{kuznetsov2011base}.
Therefore, by~\cite[Proposition~4.3]{kuznetsov2011base} it extends to a semiorthogonal decomposition
of the bounded from above derived category
\begin{equation}
\label{eq:dbtx-}
\Dm(\tX) = \langle \tcA_1^-, \tcA_2^-, \dots, \tcA_n^- \rangle,
\end{equation}
where $\tcA_i \subset \tcA_i^-$, $\tcA_i^-$ is closed under arbitrary direct sums that exist in~$\Dm(\tX)$,
and~$\tcA_i = \tcA_i^- \cap \Db(\wt{X})$.
We define a sequence of subcategories $\cA_i^- \subset \Dm(X)$ by
\begin{equation}
\label{eq:def-caim}
\cA_i^- = \pi_*(\tcA_i^-).
\end{equation}
In Proposition~\ref{proposition:sod-dbx-} below we will show that these subcategories
are triangulated and form a semiorthogonal decomposition of $\Dm(X)$.
After that, in Proposition~\ref{proposition:sod-dbx} we will check that this decomposition induces a decomposition of $\Db(X)$.

We start with a lemma that describes the intersections of the categories $\tcA_i^-$
with the kernel category of the pushforward functor $\pi_*$.

\begin{lemma}
\label{lemma:ker-pis}
{$(i)$} For each $1 \le k \le n$ we have
\begin{equation}
\label{eq:tcaim-kerpi}
\tcA_k^- \cap \Ker \pi_* = \langle \cO_{E_{k,1}}(-1), \dots, \cO_{E_{k,m_k}}(-1) \rangle^\oplus,
\end{equation}
where $\langle - \rangle^\oplus$ denotes the minimal triangulated subcategory
closed under arbitrary direct sums that exist in $\Dm(\tX)$.

\noindent{$(ii)$} For any $\cF \in \Ker\pi_* \subset \Dm(\tX)$ there is a canonical direct sum decomposition
\begin{equation*}
\cF = \bigoplus_{i=1}^n \cF_i,
\qquad \text{where} \quad
\cF_i \in \tcA_i^- \cap \Ker \pi_*.
\end{equation*}
\noindent{$(iii)$} We have a semiorthogonality
\begin{equation}
\label{eq:semiorthogonality-tcai-tcaj-ker}
\Ext^\bullet(\tcA_i^-,\tcA_j^- \cap \Ker \pi_*) = 0
\end{equation}
either if
$i < j$ and $\pi$ is crepant along $D_j$, or if $i > j$.
\end{lemma}
\begin{proof}
Let $\cF \in \Ker \pi_* \subset \Dm(\tX)$.
By Lemma~\ref{lemma:ker-pi}
every cohomology sheaf $\cH^t(\cF)$ is an iterated extension of sheaves $\cO_{E_{i,p}}(-1)$, where \mbox{$1 \le i \le n$} and~\mbox{$1 \le p \le m_i$}.
In particular,
\begin{equation*}
\Supp(\cF) {}\subset{} \bigcup E_{i,p} = D = D_1 \sqcup \dots \sqcup D_n.
\end{equation*}
We denote by $\cF_i$ the summand of $\cF$ supported on $D_i$.
Then $\cF = \oplus \cF_i$, and the summands are completely orthogonal.
Moreover, the summand $\cF_i$ is an iterated extension of shifts of sheaves~$\cO_{E_{i,p}}(-1)$,
where \mbox{$1 \le p \le m_i$}, so in view of~\eqref{eq:oe-ai-other} we have~\mbox{$\cF_i \in \tcA_i^- \cap \Ker \pi_*$}.
This proves~$(ii)$.

Since $\pi_*$ commutes with infinite direct sums,
we have
\begin{equation*}
\langle \cO_{E_{i,p}}(-1) \rangle^\oplus_{1 \le p \le m_i} \subset \tcA^{-}_k \cap \Ker(\pi_*).
\end{equation*}
Conversely, if $\cF \in \tcA_i^- \cap \Ker \pi_*$ then in the direct sum decomposition of part~$(ii)$
all summands of $\cF$ distinct from $\cF_i$ vanish (because of semiorthogonality of~\eqref{eq:dbtx-}).
Thus~\mbox{$\cF = \cF_i$} and from the above argument we conclude that $\cF$ is an iterated extension of sheaves~$\cO_{E_{i,p}}(-1)$,
where \mbox{$1 \le p \le m_i$},
hence belongs to $\langle \cO_{E_{i,p}}(-1) \rangle^\oplus_{1 \le p \le m_i}$.
This proves~$(i)$.

Finally, let us prove~$(iii)$.
If $i > j$, semiorthogonality of $\tcA_i^-$ and $\tcA_j^- \cap \Ker\pi_*$ follows from semiorthogonality of~\eqref{eq:dbtx-}.
So, let $i < j$, $\cF \in \tcA_i^-$, and $\cF' \in \tcA_{j}^- \cap \Ker \pi_*$.
Using~\eqref{eq:tcaim-kerpi} for $k = j$ and crepancy of $\pi$ along $D_j$, we deduce
\begin{equation*}
\cF' \cong \cF'(K_\tX) \in \tcA_j^-(K_\tX).
\end{equation*}
Therefore the required vanishing $\Ext^\bullet(\cF,\cF') = 0$ follows
from semiorthogonality of the decomposition of~$\Dm(\tX)$ obtained from the first line in~\eqref{eq:dbtx-ktx} with $k = i$
by an application of~\cite[Proposition~4.3]{kuznetsov2011base}.
\end{proof}

Denote by $\tilde\alpha_i\colon \Dm(\tX) \to \Dm(\tX)$
the projection functors of the decomposition~\eqref{eq:dbtx-};
so that the essential image of each $\tilde\alpha_i$ is $\tcA_i^- \subset \Dm(\tX)$.

\begin{remark}
\label{remark:tilde-alpha}
By~\cite[Proposition~4.3 and Lemma~3.1]{kuznetsov2011base} the projection functors of~\eqref{eq:dbtx}
are given by the restrictions of $\tilde\alpha_i$ to $\Db(\tX)$.
In particular, the functors $\tilde\alpha_i$ preserve boundedness.
\end{remark}

\begin{proposition}
\label{proposition:sod-dbx-}
{$(i)$} The subcategories $\cA_i^- \subset \Dm(X)$ defined by~\eqref{eq:def-caim} are triangulated and
form a semiorthogonal decomposition
\begin{equation*}
\Dm(X) = \langle \cA_1^-, \cA_2^-, \dots, \cA_n^- \rangle
\end{equation*}
with projection functors given by
\begin{equation}
\label{eq:alpha}
\alpha_i := \pi_* \circ \tilde\alpha_i \circ \pi^*.
\end{equation}
In particular, $\alpha_i\vert_{\cA_i^-} \cong \id_{\cA_i^-}$.

\noindent{$(ii)$}
For each $1 \le i \le n$ we have
\begin{equation}
\label{eq:pis-ai}
\pi^*(\cA_i^-) \subset \langle \tcA_i^-, \tcA_{i+1}^- \cap \Ker \pi_*, \dots, \tcA_n^- \cap \Ker \pi_* \rangle.
\end{equation}

\noindent{$(iii)$}
If
$\pi$ is crepant along $D_j$ for all $j > i$,
a stronger property~$\pi^*(\cA_i^-) \subset \tcA_i^-$ is fulfilled.
\end{proposition}
\begin{proof}
Let us prove~\eqref{eq:pis-ai}.
First, take any object $\cF \in \tcA_i^-$ and define $\cF'$ from the standard triangle
\begin{equation*}
\pi^*(\pi_*\cF) \to \cF \to \cF'.
\end{equation*}
Then by~\eqref{eq:pi-pi} we have $\cF' \in \Ker\pi_*$.
By Lemma~\ref{lemma:ker-pis}$(ii)$ we have $\cF' = \oplus \cF'_j$ with
\begin{equation*}
\cF'_j \in \tcA_j^- \cap \Ker \pi_*.
\end{equation*}
If~$j < i$ we have~$\Ext^\bullet(\cF,\cF'_j) = 0$ by~\eqref{eq:semiorthogonality-tcai-tcaj-ker}.
On the other hand, by adjunction we have~$\Ext^\bullet(\pi^*(\pi_*\cF),\cF'_j) = 0$ since $\cF'_j \in \Ker  \pi_*$.
Combining these vanishings with the above triangle we deduce~$\Hom(\cF',\cF'_j) = 0$ for $j < i$.
Therefore,~$\cF'_j = 0$ for $j < i$ since it is a direct summand of~$\cF'$.

This proves that $\cF' \in \langle \tcA_i^- \cap \Ker \pi_*, \tcA_{i+1}^- \cap \Ker \pi_*, \dots, \tcA_n^- \cap \Ker \pi_* \rangle$,
and from the triangle we conclude that
$\pi^*(\pi_*\cF) \in \langle \tcA_i^-, \tcA_{i+1}^- \cap \Ker \pi_*, \dots, \tcA_n^- \cap \Ker \pi_* \rangle$.
Since~$\cA_i^-$ is formed by the pushforwards $\pi_*\cF$ for $\cF \in \tcA_i^-$, this proves~\eqref{eq:pis-ai}
and part $(ii)$.

Further,
if $\pi$ is crepant along $D_j$ for $j > i$, then $\Ext^\bullet(\cF,\cF'_j) = 0$ for all $j \ne i$
by~\eqref{eq:semiorthogonality-tcai-tcaj-ker},
hence the above argument shows that $\cF'_j = 0$ for all $j \ne i$.
Thus, in this case we have $\cF' = \cF'_i \in \tcA_i^- \cap \Ker \pi_*$, hence also $\pi^*(\pi_*\cF) \in \tcA_i^-$,
and arguing as above we conclude that $\pi^*(\cA_i^-) \subset \tcA_i^-$.
This proves part $(iii)$.

To prove $(i)$ we take $\cG \in \cA_i^-$.
It follows from~\eqref{eq:pis-ai} that there is a distinguished triangle
\begin{equation*}
\cG' \to \pi^*\cG \to \tilde\alpha_i(\pi^*\cG)
\end{equation*}
with $\cG' \in \langle \tcA_{i+1}^- \cap \Ker \pi_*, \dots, \tcA_n^- \cap \Ker \pi_* \rangle \subset \Ker \pi_*$.
Applying~$\pi_*$ to the triangle we obtain
\begin{equation*}
\cG \cong \pi_*(\pi^*\cG) \cong \pi_*(\tilde\alpha_i(\pi^*\cG)) = \alpha_i(\cG).
\end{equation*}
Thus, the functor~\eqref{eq:alpha} is isomorphic to the identity functor when restricted to $\cA_i^-$.
Since the functor $\alpha_i$ is triangulated, identity on $\cA_i^-$,
and by~\eqref{eq:def-caim} the image of~$\alpha_i$ is contained in $\cA_i^-$,
it follows that the subcategory~$\cA_i^- \subset \Dm(X)$ is triangulated as well.

Next, note that for $\cF \in \cA_i^-$ and $\cG \in \tcA_j^-$ with $i > j$ we have
\begin{equation*}
\Hom(\cF,\pi_*\cG) = \Hom(\pi^*\cF,\cG) = 0
\end{equation*}
by~\eqref{eq:pis-ai}.
Since $\cA_j^-$ is formed by the pushforwards $\pi_*\cG$ for $\cG \in \tcA_j^-$,
this proves that the subcategories $\cA_i^-$ and $\cA_j^-$ are semiorthogonal for $i > j$.

Finally, take any $\cF \in \Dm(X)$.
Then $\tcF = \pi^*\cF \in \Dm(\tX)$, so we can decompose it with respect to~\eqref{eq:dbtx-}.
This means that there is a chain of maps
\begin{equation*}
0 = \tcF_n \rightarrow \dots \rightarrow \tcF_2 \rightarrow \tcF_1 \rightarrow \tcF_0 = \tcF,
\end{equation*}
whose cones are $\tilde\alpha_i(\tcF) \in \tcA_i^-$.
Pushing this forward to $X$, we obtain a chain of maps
\begin{equation*}
0 = \pi_*(\tcF_n) \rightarrow \dots \rightarrow \pi_*(\tcF_2) \rightarrow \pi_*(\tcF_1) \rightarrow \pi_*(\tcF_0) = \pi_*(\tcF) \cong \cF,
\end{equation*}
whose cones are $\pi_*(\tilde\alpha_i(\tcF)) \cong \pi_*(\tilde\alpha_i(\pi^*\cF)) = \alpha_i(\cF) \in \cA_i^-$.
This proves the semiorthogonal decomposition and shows that its projection functors are given by~$\alpha_i$.
\end{proof}

\begin{corollary}
\label{cor:ap-p-adjunction}
The functor $\tilde\alpha_i \circ \pi^* \colon \cA_i^- \to \tcA_i^-$ is fully faithful
and is left adjoint to the functor $\pi_* \colon \tcA_i^- \to \cA_i^-$.
Moreover, we have a semiorthogonal decomposition
\begin{equation}
\label{eq:semiorth-A-minus}
\tcA_i^{-} = \langle \tcA_i^{-} \cap \Ker \pi_*,
\tilde\alpha_i(\pi^*(\cA_i^{-}) \rangle
\end{equation}
and in particular
\begin{equation}
\label{eq:image-ap}
\tilde\alpha_i(\pi^*(\cA_i^-)) = {}^\perp\langle \cO_{E_{i,1}}(-1), \dots, \cO_{E_{i,m_i}}(-1) \rangle \subset \tcA_i^-.
\end{equation}
\end{corollary}
\begin{proof}
The adjunction follows from the adjunction between the pullback $\pi^*$ and the pushforward $\pi_*$,
and between the embedding $\tcA_i^- \hookrightarrow \langle \tcA_i^-, \tcA_{i+1}^-, \dots, \tcA_n^- \rangle$
and the projection functor $\tilde\alpha_i \colon \langle \tcA_i^-, \tcA_{i+1}^-, \dots, \tcA_n^- \rangle \to \tcA_i^-$.
Full faithfulness follows from the adjunction and the isomorphism
$\pi_* \circ \tilde\alpha_i \circ \pi^* \cong \operatorname{id}\vert_{\tcA_i^-}$ on $\tcA_i^-$
proved in Proposition~\ref{proposition:sod-dbx-}$(i)$.
This proves~\eqref{eq:semiorth-A-minus}, and~\eqref{eq:image-ap} follows from~\eqref{eq:semiorth-A-minus} and~\eqref{eq:tcaim-kerpi}.
\end{proof}

Note that if
$\pi$ is crepant along $D_j$ for $j > i$,
we have $\tilde\alpha_i \circ \pi^* \cong \pi^*$ on~$\cA_i^-$ by Proposition~\ref{proposition:sod-dbx-}$(iii)$.

\subsection{Decomposition of the bounded category}
\label{subsection:db-decomposition}

Now we show that the semiorthogonal decomposition of $\Dm(X)$ constructed in Proposition~\ref{proposition:sod-dbx-}
induces a semiorthogonal decomposition of $\Db(X)$.

\begin{proposition}
\label{proposition:sod-dbx}
{$(i)$}
The subcategories
\begin{equation*}
\cA_i := \cA_i^- \cap \Db(X)
\end{equation*}
provide a semiorthogonal decomposition~\eqref{eq:dbx} with the projection functors $\alpha_i$ given by~\eqref{eq:alpha}.

\noindent{$(ii)$}
The functors~$\alpha_i$ preserve boundedness and have finite cohomological amplitude.

\noindent{$(iii)$
We have an equality}
$\cA_i = \pi_*(\tcA_i)$ and $\pi_*$ induces an equivalence of triangulated categories
\begin{equation*}
\cA_i \simeq \tcA_i \big/ (\tcA_i \cap \Ker \pi_*),
\end{equation*}
where the right hand side is a Verdier quotient.
\end{proposition}
\begin{proof}
For $(i)$ it is enough to check that the projection functors~$\alpha_i$ preserve boundedness.
Take any object $\cF \in \DD^{[k_-,k_+]}(X)$.
Then $\pi^*(\cF) \in \DD^{(-\infty,k_+]}(\tX)$.
By~\eqref{eq:pi-pi} we have~$\pi_*(\pi^*(\cF)) \cong \cF$,
hence by Lemma~\ref{lemma:pis-tau} we have $\tau^{\le {k_--2}}(\pi^*(\cF)) \in \Ker \pi_*$.
Consider the triangle
\begin{equation*}
\tilde\alpha_i(\tau^{\le {k_--2}}(\pi^*(\cF))) \to \tilde\alpha_i(\pi^*(\cF)) \to \tilde\alpha_i(\tau^{\ge {k_- - 1}}(\pi^*(\cF)))
\end{equation*}
obtained by applying the projection functor $\tilde\alpha_i$ to the canonical truncation triangle.
By Lemma~\ref{lemma:ker-pis}$(ii)$ the functor $\tilde\alpha_i$ preserves~$\Ker\pi_*$, hence the first term of the triangle is in $\Ker \pi_*$.
Therefore, applying the pushforward and using~\eqref{eq:alpha} we obtain an isomorphism
\begin{equation}
\label{eq:alpha-cf}
\alpha_i(\cF)  = \pi_*(\tilde\alpha_i(\pi^*(\cF))) \cong \pi_*(\tilde\alpha_i(\tau^{\ge {k_- - 1}}(\pi^*(\cF))).
\end{equation}
It remains to note that we have~$\tau^{\ge {k_- - 1}}(\pi^*(\cF)) \in \DD^{[{k_- - 1},k_+]}(\tX)$,
hence the object in the right-hand side of~\eqref{eq:alpha-cf} is bounded, since both~$\tilde\alpha_i$
(see Remark~\ref{remark:tilde-alpha}) and~$\pi_*$ preserve boundedness.
This completes the proof of~$(i)$.

Moreover, if the cohomological amplitude of $\tilde\alpha_i$ is $(a_-,a_+)$ (it is finite since $\tX$ is smooth,
see~\cite[Proposition~2.5]{K08}), then from~\eqref{eq:alpha-cf} and from cohomological amplitude of $\pi_*$ being $(0,1)$,
we deduce
\begin{equation*}
\alpha_i(\cF) \cong \pi_*(\tilde\alpha_i(\tau^{\ge {k_- - 1}}(\pi^*(\cF))) \in \DD^{[{k_- + a_- - 1},k_++a_++1]}(X).
\end{equation*}
In particular, $\alpha_i$ has finite cohomological amplitude.
This proves~$(ii)$.

Let us prove~$(iii)$.
By~\eqref{eq:def-caim} we have $\pi_*(\tcA_i) \subset \cA_i^-$, and since $\pi_*$ preserves boundedness, we have $\pi_*(\tcA_i) \subset \cA_i$.
To check that this inclusion is an equality, take any~$\cF \in \cA_i$.
By Corollary~\ref{lemma:im-pi} there exists $\tcF \in \Db(\tX)$ such that $\cF \cong \pi_*(\tcF)$.
Let~$\cG$ be the cone of the natural morphism $\pi^*\cF \to \tcF$.
Then $\cG \in \Ker  \pi_*$.
Moreover, by Lemma~\ref{lemma:ker-pis}($ii$) we have $\tilde\alpha_i(\cG) \in \Ker \pi_*$, hence applying the functor $\pi_* \circ \tilde\alpha_i$
to the distinguished triangle~$\pi^*\cF \to \tcF \to \cG$, we deduce an isomorphism $\cF \cong \alpha_i(\cF) \cong \pi_*(\tilde\alpha_i(\tcF))$,
and it remains to note that $\tilde\alpha_i(\tcF) \in \tcA_i$.

Finally, as we have already shown that $\pi_*\colon \tcA_i \to \cA_i$
is essentially surjective, it remains to show that the induced functor
$\pi_* \colon \tcA_i \big/ (\tcA_i \cap \Ker \pi_*)
\to \cA_i$ is fully faithful; to show this
we use the argument from~\cite[Lemma 2.31]{Pavic-Shinder}
which goes as follows.
From the commutative diagram
\begin{equation*}
\xymatrix@C=4em{
\tcA_i^- \big/ (\tcA_i^- \cap \Ker \pi_*)\ar[r]^-{\simeq} & \cA_i^- \\
\tcA_i \big/ (\tcA_i \cap \Ker \pi_*) \ar[u] \ar[r] & \cA_i \ar@{^{(}->}[u]  \\
}
\end{equation*}
(the top arrow is an equivalence by Corollary~\ref{cor:ap-p-adjunction})
we deduce that bottom horizontal functor is fully faithful if and only if the left vertical functor
\begin{equation*}
\tcA_i \big/ (\tcA_i \cap \Ker \pi_*) \to \tcA_i^- \big/ (\tcA_i^- \cap \Ker \pi_*)
\end{equation*}
is fully faithful. To show that this is the case, we use the Verdier criterion
\cite[Theorem 2.4.2]{verdier}:
it suffices to show that every
morphism $\cG \to \cG'$ in $\tcA_i^-$ with
\mbox{$\cG \in \tcA_i^- \cap \Ker \pi_*$} and $\cG' \in \tcA_i$
factors through an object from $\tcA_i \cap \Ker \pi_*$.
It is easily seen that for such an object we can take
$\alpha_i (\tau^{\ge -N} \cG)$, for $N$ large enough.
\end{proof}

Now we can prove the theorem.

\begin{proof}[Proof of Theorem~\textup{\ref{theorem:dbx}}]
Part $(i)$ follows from Proposition~\ref{proposition:sod-dbx}.

$(ii)$
Let $1 \le i \le n$ and assume that $\pi$ is crepant along~$D_j$ for~$j > i$.
By Lemma~\ref{lemma:twisted-compatible} besides the semiorthogonal decomposition~\eqref{eq:dbtx}
also the first of the semiorthogonal decompositions~\eqref{eq:dbtx-ktx} for $k = i$ is compatible with~$\pi$.
Hence, as we proved in part~$(i)$ of the theorem, it gives rise to a semiorthogonal decomposition
\begin{align*}
\Db(X) &= \langle \pi_*(\tcA_{i+1}(K_\tX)), \dots, \pi_*(\tcA_n(K_\tX)), \pi_*(\tcA_1), \dots, \pi_*(\tcA_{i}) \rangle,
\end{align*}
of $\Db(X)$.
The category $\cA_i = \pi_*(\tcA_i)$ is its rightmost component,
hence is right admissible in $\Db(X)$.

Similarly, assuming that $\pi$ is crepant along $D_j$ for $j < i$ and using
the second of the semiorthogonal decompositions~\eqref{eq:dbtx-ktx} for $k = i$,
we deduce that~$\cA_i$ is left admissible.
Thus, if $\pi$ is crepant for all $j \ne i$, $\cA_i$ is an admissible subcategory in~$\Db(X)$,
and if $\pi$ is crepant, all $\cA_i$ are admissible.

$(iii)$ Follows from Proposition~\ref{proposition:sod-dbx-}$(ii)$.

$(iv)$ The inclusion $\pi^*(\cA_i^{\perf}) \subset \tcA_i$ under the appropriate
crepancy assumptions
follows from Proposition~\ref{proposition:sod-dbx-}$(iii)$.

Assume now that $\pi$ is crepant along all $D_j$ for $j \ge 2$.
Let $\cF \in \Dperf(X)$ and let~$\cF_1,\dots,\cF_n \in \Db(X)$ be its components with respect to the semiorthogonal
decomposition~\eqref{eq:dbx}.
By Proposition~\ref{proposition:sod-dbx-} we have $\pi^*\cF_i \subset \tcA_i^-$,
hence $\pi^*\cF_i$ are the components of $\pi^*\cF$ with respect to the semiorthogonal decomposition~\eqref{eq:dbtx-}.
But~$\pi^*\cF \in \Db(\tX)$, hence by Remark~\ref{remark:tilde-alpha} we conclude that $\pi^*\cF_i \in \Db(\tX)$,
and Corollary~\ref{cor:descent} finally shows that~$\cF_i \in \Dperf(X)$.
This proves that the projection functors $\alpha_i$ of~\eqref{eq:dbx} preserve perfectness,
hence we obtain~\eqref{eq:dpxperf}.

Finally, if $\pi$ is crepant, then rotating the decomposition left or right as in the proof of~$(ii)$
we show that each $\cA_i^\perf \subset \Dperf(X)$ is left and right admissible.
\end{proof}

Note that for each $i$ the category $\tcA_i$ provides (via the functors $\pi_*$ and $\tilde\alpha_i \circ \pi^*$)
a categorical resolution for the category $\cA_i$ in the sense of~\cite{KL}.

\section{Components of the induced semiorthogonal decomposition}
\label{section:components}

In this section we will provide a description of the components
of the semiorthogonal decomposition $\Db(X)= \langle \cA_1, \cA_2, \dots, \cA_n \rangle$
constructed in Theorem~\ref{theorem:dbx} under an additional assumption
on the components of the decomposition $\Db(\tX)= \langle \tcA_1, \tcA_2, \dots, \tcA_n \rangle$ we started with.
If this assumption holds for a component $\tcA_i$, then we will show that this component
$\tcA_i$ is equivalent to the bounded derived category $\Db(\Lambda_i\md)$ of finitely generated right modules
for a noncommutative algebra~$\Lambda_i$ and the corresponding component $\cA_i$ can be realized as $\Db(K_i\md)$
for a finite dimensional algebra~$K_i$, which we describe explicitly.

The additional assumptions we impose will require the exceptional divisor $D_i$ associated with the component $\tcA_i$ by~\eqref{eq:def-di}
to be a \emph{chain} of rational curves,
and this restricts us to the situation where the singularity obtained by the contraction of~$D_i$ is
a cyclic quotient singularity (see Proposition~\ref{prop:cyclic}).

\subsection{Cyclic quotient singularities}
\label{subsection:cyclic-singularity}

An isolated singularity $(X,x)$ is a cyclic quotient singularity if it is \'etale-locally isomorphic to the quotient $\A^n/\bmu_r$,
where $\bmu_r$ is the group of roots of unity of order $r$ that acts on $\A^n$ linearly and freely away from the origin.
The latter condition means that the action is given by a collection $(a_1,\dots,a_n)$ of characters
\begin{equation*}
a_i \in \Hom(\bmu_r,\Gm) \cong \Z/r
\end{equation*}
such that each $a_i$ is invertible in $\Z/r$.
Thus we can think of the~$a_i$ as integers such that~\mbox{$0 \le a_i < r$} and~\mbox{$\gcd(a_i,r) = 1$}.
Moreover, the $a_i$ are only well-defined up to the action of $\Aut(\bmu_r) \cong (\Z/r)^\times$,
so we can assume that $a_1 = 1$.
The singularity corresponding to the $\bmu_r$-action with weights $(1,a_2,\dots,a_n)$ will be denoted by
\begin{equation*}
\tfrac1r(1,a_2,\dots,a_n).
\end{equation*}

In this paper we stick to the case of dimension 2.
Accordingly, any cyclic quotient singularity of a surface is isomorphic to one of the singularities
\begin{equation*}
\tfrac1r(1,a),
\qquad
0 \le a < r,
\qquad
\gcd(a,r) = 1.
\end{equation*}
The following well-known result is crucial for us.

\begin{proposition}[{See e.g.~\cite[Satz 2.10 and 2.11]{BrieskornKomplexerFlachen}}]
\label{prop:cyclic}
Let $(X,x)$ be a rational singularity of a surface and let $\pi \colon \tX \to X$ be its minimal resolution.
The following properties are equivalent:
\begin{enumerate}
\item $(X,x)$ is a cyclic quotient singularity;
\item the irreducible components $E_i$ of the exceptional divisor of $\pi$
are smooth rational curves forming a chain,
i.e., after a possible reordering each $E_i$ intersects $E_{i+1}$ transversally at a single point
and $E_i \cap E_j = \varnothing$ for $|i - j| \ge 2$.
\end{enumerate}
\end{proposition}

If $(X,x)$ is a cyclic quotient singularity of a surface and $\pi \colon \tX \to X$ is its minimal resolution,
the self-intersections of the components $E_i$ (ordered as in Proposition~\ref{prop:cyclic}(2))
of the exceptional divisor of $\pi$ are encoded in a {\sf Hirzebruch--Jung continued fraction}.

For a collection of positive integers $d_1,\dots,d_m \ge 2$ we denote
\begin{equation}
\label{eq:hj-fraction}
[d_1, \dots, d_m]:= d_1 - \frac1{d_{2} - \frac1{d_{3} - \dots}}.
\end{equation}
This is a rational number greater than~1.
Conversely, every rational number greater than~1 can be written as a Hirzebruch--Jung continued fraction $[d_1,\dots,d_m]$
with $d_i \ge 2$ in a unique way.

\begin{proposition}[{See e.g.~\cite[Satz 2.11]{BrieskornKomplexerFlachen}}]
\label{proposition:sing-frac}
Assume that $(X,x)$ is a cyclic quotient singularity of type~$\frac1r(1,a)$ and let
\begin{equation}
\label{eq:ra-hj-cf}
\frac{r}{a} = [d_1,\dots,d_m]
\end{equation}
be the Hirzebruch--Jung continued fraction representation for $r/a$.
Then the intersection matrix of\/ \textup(ordered in a chain\textup) irreducible components of the exceptional divisor
of a minimal resolution $\pi \colon \tX \to X$ of $(X,x)$ is the tridiagonal matrix
\begin{equation}
\label{eq:tridiag}
\tridiag(d_1,d_2,\dots,d_m) :=
\left[\begin{array}{ccccc}
-d_1 & 1 & 0 & \dots & 0 \\
1 & -d_2 & 1 & \dots & 0 \\
& & \dots & & \\
0 & \dots & 1 & -d_{m-1} & 1 \\
0 & \dots & 0 & 1 & -d_m \\
\end{array}
\right].
\end{equation}
\end{proposition}

Conversely, the type of the singularity can be recovered from the intersection matrix.
Denote by
\begin{equation}
\label{eq:tridet}
\tridet(d_1,d_2,\dots,d_m) := (-1)^m \det(\tridiag(d_1,d_2,\dots,d_m)),
\end{equation}
the determinant of a tridiagonal matrix (\emph{continuant}) with the sign.
Note that the assumption $d_i \ge 2$ implies that $\tridet(d_1,d_2,\dots,d_m) > 0$.

\begin{lemma}
\label{lemma:frac-sing}
For any $i \le j$ we have
\begin{equation*}
[d_i,\dots,d_j] = \frac{\tridet(d_i,\dots,d_j)}{\tridet(d_{i+1},\dots,d_j)}\,,
\end{equation*}
where the numerator and denominator in the right side are  coprime.
In particular, if~$\gcd(r,a) = 1$ and $r/a = [d_1,\dots,d_m]$ then
\begin{equation}
\label{eq:r-a-d}
r = \tridet(d_1,\dots,d_m)
\qquad\text{and}\qquad
a = \tridet(d_2,\dots,d_m).
\end{equation}
\end{lemma}
\begin{proof}
This can be proved by an elementary induction on $j - i$.
\end{proof}

\begin{remark}
Inverting the order of the chain of exceptional curves
leads to the continued fraction
\begin{equation*}
\frac{r}{a'} = [d_{m}, \dots, d_1],
\end{equation*}
where $0 \le a' < r$ satisfies $a \cdot a' \equiv 1 \mod{n}$.
Note that $\frac1{r}(1,a)$
and $\frac1{r}(1,a')$ are isomorphic singularities.
\end{remark}

\subsection{Adherent components}
\label{subsection:adherent}

Assume we are in the setup of Theorem~\ref{theorem:dbx}, fix some component $\tcA_i$ of~\eqref{eq:dbtx},
and assume that the corresponding divisor~$D_i \subset \tX$ defined by the condition~\eqref{eq:def-di}
is a chain $E_{i,1},\dots,E_{i,m_i}$ of rational curves.
Recall that this means that~$\cO_{E_{i,p}}(-1) \in \tcA_i$, and $E_{i,p}$ are the only exceptional curves with this property.
Note that the category generated by the sheaves $\cO_{E_{i,p}}(-1)$, $1 \le p \le m_i$ is \emph{not} admissible in~$\Db(\tX)$.
The assumption we want to make is that $\tcA_i$ is, in a sense,
the smallest possible admissible subcategory of~$\Db(\tX)$ containing all $\cO_{E_{i,p}}(-1)$.

For convenience we slightly generalize this setup:
we fix a sequence $b_{i,p}$, $1 \le p \le m_i$, of integers,
and instead of considering line bundles $\cO_{E_{i,p}}(-1)$,
replace them by the twisted collection $\{\cO_{E_{i,p}}(-1+b_{i,p})\}_{1 \le p \le m_i}$.
We denote
\begin{equation}
\label{eq:dip}
d_{i,p} := -E_{i,p}^2,
\end{equation}
so that the intersection matrix of $E_{i,p}$ is the tridiagonal matrix $\tridiag(d_{i,1},\dots,d_{i,m_i})$ as in~\eqref{eq:tridiag}.
We set $x := \pi(D_i) \subset X$.
A combination of Proposition~\ref{proposition:sing-frac} and Lemma~\ref{lemma:frac-sing}
shows that $(X,x)$ is a cyclic quotient singularity of type $\frac1{r_i}(1,a_i)$, where
\begin{equation*}
r_i = \tridet(d_{i,1},\dots,d_{i,m_i}),
\qquad\text{and}\qquad
a_i = \tridet(d_{i,2},\dots,d_{i,m_i}).
\end{equation*}

The main definition of this section will be given in terms of the following lemma.

\begin{lemma}
\label{lemma:adherence-equivalence}
Let $\tX$ be a smooth projective surface such that
\begin{equation}
\label{eq:hotx}
H^\bullet(\tX,\cO_\tX) \cong \Bbbk.
\end{equation}
Let $D = \cup_{p=1}^{m} E_{p}$ be a chain of smooth rational curves on $\tX$ with $E_p^2 = -d_p \le -2$
and let~$\tcA \subset \Db(\tX)$ be a triangulated subcategory.
The following conditions
are equivalent:

\begin{enumerate}
\item
There exist a line bundle $\cL_0$ on $\tX$ and integers $b_1,\dots,b_m$ such that
\begin{equation}
\label{eq:tcai-l-o}
\tcA
\quad\text{is generated by sheaves}\quad
\cL_{0}, \cO_{E_{1}}(-1 + b_{1}), \dots, \cO_{E_{m}}(-1 + b_{m}),
\end{equation}
and
\begin{equation}
\label{eq:cli0-oeip}
\cL_{0} \cdot E_{p} =
\begin{cases}
d_{1} + b_{1} - 1, & \text{if $p = 1$},\\
d_{p} + b_{p} - 2, & \text{if $2 \le p \le m$}.
\end{cases}
\end{equation}
\item
The category $\tcA$ is generated by an exceptional collection of line bundles
\begin{equation}
\label{eq:tcai-l-l}
\tcA = \langle \cL_{0}, \cL_{1}, \dots, \cL_{m} \rangle,
\end{equation}
where for all $1 \le p \le m$ we have
\begin{equation}
\label{eq:sequence-lip}
\cL_{p} := \cL_{0}(E_{1} + \dots + E_{p}).
\end{equation}
\end{enumerate}
\end{lemma}

\begin{proof}
Assume~(2).
The standard exact sequence $0 \to \cO_\tX(-E_{p}) \to \cO_\tX \to \cO_{E_{p}} \to 0$ after tensoring by $\cL_{p}$ gives
\begin{equation}
\label{eq:lip-lip-bip}
0 \to \cL_{p-1} \to \cL_{p} \to \cO_{E_{p}}(-1 + b_{p}) \to 0,
\end{equation}
where $b_{p}$ are defined by
\begin{equation}
\label{eq:bip-lip}
b_{p} = \cL_{p} \cdot E_{p} + 1.
\end{equation}
This shows that the subcategory~\eqref{eq:tcai-l-l}
is at the same time generated by the line bundle~$\cL_{0}$ and the sheaves $\cO_{E_{p}}(-1 + b_{p})$.
Furthermore,
\begin{equation*}
\cL_{0} \cdot E_{1} = \cL_{1} \cdot E_{1} - E_{1}^2 = -1 + b_{1} + d_{1},
\end{equation*}
and for each $2 \le p \le m$
\begin{align*}
\cL_{0} \cdot E_{p} &=
\cL_{p} \cdot E_{p} - E_{p-1} \cdot E_{p} - E_{p}^2 =
-1 + b_{p} - 1 + d_{p}.
\end{align*}
This proves~(1).

Conversely, assume~(1) and define the sequence of line bundles by~\eqref{eq:sequence-lip}.
A computation similar to the above shows that $\cL_{p} \cdot E_{p} = -1 + b_{p}$,
hence we have an exact sequence~\eqref{eq:lip-lip-bip} for each $1 \le p \le m$.
This proves that $\tcA$ is generated by the sequence~$\cL_{p}$ with $0 \le p \le m$.
Finally, the sequence of line bundles
is exceptional by~\cite[Lemma~2.1]{HP}, hence~(2) holds.
\end{proof}

\begin{definition}
\label{definition: adherent}
Let $\tX$ be a smooth projective surface such that~\eqref{eq:hotx} holds.
We will say that a triangulated subcategory $\tcA \subset \Db(\tX)$ is {\sf twisted adherent}
to a chain $D {} = \cup_{p=1}^{m} E_{p}$ of smooth rational curves if
any of the equivalent conditions of Lemma~\ref{lemma:adherence-equivalence} hold.

The sequence $(b_1,\dots,b_m)$ is called the {\sf twist} of an adherent category.
When we do not want to specify the twist we will just say that $\tcA$ is {\sf adherent} to $D$,
and in the case when all~$b_{p}$ are zero, we will say that $\tcA$ is {\sf untwisted adherent} to $D$.
\end{definition}

In~\cite{HP} Hille and Ploog suggest to think of an adherent category $\tcA$
as of a ``categorical neighborhood'' of the sheaves $\cO_{E_{p}}(-1+b_p)$.

\begin{example}
Let us illustrate the concept of untiwsted adherence for the singularities~$\frac1{r}(1,1)$ and $\frac1{r}(1,r-1)$:
\begin{itemize}
\item
in the $\frac1{r}(1,1)$ case the exceptional divisor is a single $(-r)$-curve $E$,
and the adherence condition for the line bundle $\cL_0$ from Lemma~\ref{lemma:adherence-equivalence}
is $\cL_0 \cdot E = r - 1$ (cf.\ Example~\ref{ex:one-sing});
\item
in the $\frac1{r}(1,r-1)$ case the exceptional divisor is a chain
of $(-2)$-curves $E_1, \dots, E_{r-1}$, and the adherence condition is
$\cL_0 \cdot E_1 = 1$, $\cL_0 \cdot E_2 = \dots = \cL_0 \cdot E_{r-1} = 0$.
\end{itemize}
\end{example}

Actually, for the main results of this section (see subsection~\ref{subsection:cai}) we use only untwisted adherence.
However, the twisted version will become important in Section~\ref{section:brauer},
so in this subsection we work in the more general case of twisted adherence.

\begin{remark}
\label{remark:adherent-empty}
We want to consider the case of \emph{empty} $D$ as a special case of the above situation:
explicitly, a category adherent to the empty divisor is just a category generated
by a single line bundle $\cL_{0}$ with no conditions imposed.
\end{remark}

Below we provide
a reformulation of a result of Hille and Ploog
that is crucial for us;
it allows to relate adherent exceptional collections to
finite-dimensional noncommutative algebras.

When $\Lambda$ is a noncommutative algebra we write $\Lambda\md$ for the abelian
category of finitely generated right $\Lambda$-modules.
Recall that a finite-dimensional algebra $\Lambda$ is called {\sf basic} if~$\Lambda/\fR \cong \Bbbk \times \dots \times \Bbbk$,
where $\fR$ is the Jacobson radical of $\Lambda$.
Assume the copies of the field $\Bbbk$ in the right-hand side of the above equality are indexed by the integers $0,1,\dots,m$.
The corresponding idempotents $\bar{e}_0,\bar{e}_1,\dots,\bar{e}_m$ in $\Lambda/\fR$
lift to unique idempotents $e_0,e_1,\dots,e_m$ in $\Lambda$.
Then for each $0 \le p \le m$
\begin{equation*}
S_p = \bar{e}_p(\Lambda/\fR)
\end{equation*}
is a simple (one-dimensional) $\Lambda$-module, and every finite-dimensional $\Lambda$-module has a filtration with factors $S_p$;
in particular every simple $\Lambda$-module is isomorphic to one of~$S_p$.
Similarly, for each $0 \le q \le m$
\begin{equation*}
P_q = e_q\Lambda
\end{equation*}
is a projective $\Lambda$-module and we have
\begin{equation}
\label{eq:ext-p-s}
\Ext^\bullet(P_q,S_p) = S_p e_q = \bar{e}_p (\Lambda/\fR) \bar{e}_q =
\begin{cases}
\Bbbk, & \text{if $p = q$},\\
0, & \text{if $p \ne q$}.
\end{cases}
\end{equation}
In particular, all $P_q$ are indecomposable, pairwise non-isomorphic, and every projective $\Lambda$-module is a sum of $P_q$.

\begin{theorem}[{\cite[Theorem~2.5]{HP}}]
\label{theorem:hille-ploog}
Let $\tX$ be a smooth projective surface satisfying~\eqref{eq:hotx},
let $D_i = E_{i,1} \cup E_{i,2} \cup \dots \cup E_{i,m_i} \subset \tX$ be
a chain of rational curves of length~$m_i$ with $E_{i,p}^2 \le -2$ for each $1 \le p \le m_i$,
and assume that a subcategory $\tcA_i \subset \Db(\tX)$ is adherent to~$D_i$.
Then

\noindent{$(i)$}
The subcategory $\tcA_i \subset \Db(\tX)$ is admissible.

\noindent{$(ii)$}
There is a sequence $\cP_{i,p}$, $0 \le p \le m_i$ of vector bundles on $\tX$ such that
\begin{equation}
\label{eq:def-lambda}
\Lambda_i := \End_\tX(\cP_{i,0} \oplus \cP_{i,1} \oplus \dots \oplus \cP_{i,m_i})
\end{equation}
is a finite dimensional basic
algebra of finite global dimension and
\begin{equation}
\label{eq:def-tgamma}
\tilde\gamma_i \colon \Db(\Lambda_i\md) \xrightarrow{\ \simeq\ } \tcA_i,
\qquad
M \mapsto M \otimes_{\Lambda_i} \left( \bigoplus_{p=0}^{m_i} \cP_{i,p} \right)
\end{equation}
is an equivalence of categories.

\noindent{$(iii)$}
The algebra $\Lambda_i$ is quasi-hereditary
with $m_i+1$ simple modules $S_{i,p}$ and indecomposable projective modules
$P_{i,p}$, $0 \le p \le m_i$,
and
\begin{equation}
\label{eq:tgamma-p-s}
\begin{aligned}
\tilde\gamma_i(P_{i,p}) &\cong \cP_{i,p},
&& 0 \le p \le m_i,
\\
\tilde\gamma_i(S_{i,0}) &\cong \cL_{i,0},
\\
\tilde\gamma_i(S_{i,p}) &\cong \cO_{E_{i,p}}(-1 + b_{i,p}),\quad && 1 \le p \le m_i,
\end{aligned}
\end{equation}
where $(b_{i,p})$ is the twist of $\tcA_i$.
\end{theorem}

\begin{proof}
By Lemma~\ref{lemma:adherence-equivalence}
the category $\tcA_i$ is generated by an exceptional collection
formed by~$\cL_{i,0}$ and
$\cL_{i,p} = \cL_{i,0}(E_{i,1} + \dots + E_{i,p})$, $1 \le p \le m_i$, for an appropriate line bundle~$\cL_{i,0}$,
hence is admissible.
Furthermore, the sequence of line bundles $\cL_{i,p}$
is obtained from the sequence in~\cite[Theorem~2.5]{HP}
by a line bundle twist, and so all the properties~\mbox{$(i)$--$(iii)$} for the category $\tcA_i$ have been established in~\cite{HP}.
Namely, the vector bundle~$\cP_{i,p}$ is constructed as the iterated universal extension
of~$\cL_{i,p}$ by~$\cL_{i,p+1}, \dots, \cL_{i,m_i}$,
and the functor
\begin{equation}
\label{eq:tgamma-radjoint}
\RHom\left(\bigoplus_{p=0}^{m_i} \cP_{i,p}, - \right) \colon \tcA_i \xrightarrow{\ \simeq\ } \Db(\Lambda_i\md)
\end{equation}
is shown to be an equivalence of categories.
The functor $\tilde\gamma_i$ is its left adjoint, hence defines the inverse equivalence.
Furthermore, in the proof of~\cite[Proposition~1.3]{HP} it is shown
that the sheaves $\cP_{i,p}$ correspond to indecomposable projective modules~$P_{i,p}$
and the fact that the sheaves $\cL_{i,0}$ and $\cO_{E_{i,p}}(-1+b_{i,p})$ correspond to simple modules~$S_{i,0}$ and~$S_{i,p}$
was explained just before~\cite[Theorem~2.5]{HP} (see also the proof of~\cite[Corollary~1.9]{HP}).

The algebra $\Lambda_i$ is basic because it is isomorphic to the endomorphism algebra
of the sum of its indecomposable projective modules~$P_{i,p}$.
\end{proof}

\begin{remark}
\label{remark:tgammai}
Note that the pair of functors defined by~\eqref{eq:def-tgamma} and~\eqref{eq:tgamma-radjoint} extend to an adjoint pair of functors
between the bounded above categories~$\Dm(\Lambda_i\md)$ and~$\tcA_i^-$ (defined by~\eqref{eq:dbtx-})
and a standard argument shows that they induce an equivalence~$\Dm(\Lambda_i\md) \cong \tcA_i^-$.
\end{remark}

In what follows we call $\Lambda_i$ the {\sf Hille--Ploog algebra}.
In case of a category adherent to the empty divisor as in Remark~\ref{remark:adherent-empty}, we have~$\Lambda_i \cong \Bbbk$.

\subsection{Hille--Ploog algebras as resolutions of singularities}
\label{subsection:hille-ploog}

Following~\cite[Definition~4.6]{KK} from the algebra $\Lambda_i$ we construct the algebra
\begin{equation}
\label{eq:def-k}
K_i := \End_\tX(\cP_{i,0}) \cong \End_{\Lambda_i}(P_{i,0}).
\end{equation}
We call $K_i$ the {\sf Kalck--Karmazyn algebra}.
Every $K_i$ is a local finite dimensional algebra
with explicit generators and relations (see Lemma~\ref{lemma:kk} for its explicit description).

The projective $\Lambda_i$-module $P_{i,0}$ is a $K_i$-$\Lambda_i$-bimodule, hence defines a pair of functors
\begin{align}
\label{eq:def-rho}
\begin{aligned}
\rho_{i*} &\colon \Dm(\Lambda_i\md) \to \Dm(K_i\md), \qquad&
M &\mapsto \RHom_{\Lambda_i}(P_{i,0},M),\\
\rho_i^* &\colon \Dm(K_i\md) \to \Dm(\Lambda_i\md), &
N &\mapsto N \otimes_{K_i} P_{i,0},
\end{aligned}
\end{align}
(where the tensor product is derived).
In case of a category adherent to the empty divisor as in Remark~\ref{remark:adherent-empty}, we have $K_i \cong \Lambda_i \cong \Bbbk$
and the functors $\rho_i^*$ and $\rho_{i*}$ are equivalences.

In general, we suggest to think about the pair of functors $(\rho_i^*,\rho_{i*})$ as
making $\Lambda_i$ into a noncommutative, or rather categorical,
resolution of $K_i$, see~\cite{KL} for a discussion of this concept.
We expect that the algebra $\Lambda_i$ is the Auslander resolution of the algebra~$K_i$ as defined in~\cite[\S5]{KL}.
Below we show that the pair of functors $(\rho_i^*,\rho_{i*})$ satisfies the same properties as
the resolution~$(\pi^*,\pi_*)$ of a rational surface singularities does, see Section~\ref{subsection:resolution}.

First of all, by~\eqref{eq:def-rho}
the functor $\rho_i^*$ is the left adjoint of $\rho_{i*}$ and preserves perfectness, while the functor $\rho_{i*}$ preserves boundedness:
\begin{equation*}
\rho_i^*(\Dperf(K_i\md)) \subset \Dperf(\Lambda_i\md),
\qquad
\rho_{i*}(\Db(\Lambda_i\md)) \subset \Db(K_i\md).
\end{equation*}
Moreover, on the bounded above categories these functors satisfy the property
\begin{equation}
\label{eq:rho-rho}
\rho_{i*} \circ \rho_i^* \cong \id;
\end{equation}
in particular the functor $\rho_i^*$ is fully faithful.
It follows immediately from~\eqref{eq:ext-p-s} and the definition~\eqref{eq:def-rho} of $\rho_{i*}$ that
the kernel of the functor~$\rho_{i*}$ is spanned by the simple modules
(except $S_{i,0}$), hence
\begin{align}
\label{eq:ker-rho}
\begin{aligned}
\Ker \rho_{i*} &= \langle S_{i,1},\dots,S_{i,m} \rangle^{\oplus} \subset \Dm(\Lambda_i\md),\\
\Im \rho_i^* &= {}^\perp\langle S_{i,1},\dots,S_{i,m} \rangle \subset \Dm(\Lambda_i\md),
\end{aligned}
\end{align}
where $\langle - \rangle^\oplus$ denotes the minimal triangulated subcategory
closed under arbitrary direct sums that exist in $\Dm(\Lambda_i\md)$.

The following lemma is an analog of Lemma~\ref{lemma:descent} and can be proved by a similar (in fact, even simpler) argument.
\begin{lemma}
\label{lemma:descent-algebraic}
Let $M \in \Db(\Lambda_i\md)$.
The following properties are equivalent:
\begin{enumerate}
\item
one has $M \in \langle P_{i,0} \rangle$;
\item
for any $1 \le p \le m_i$ one has $\Ext^\bullet(M, S_{i,p}) = 0$;
\item
there exists $N \in \Dperf(K_i\md)$ such that $M \cong \rho_i^*N$;
\item
one has $\rho_{i*}M \in \Dperf(K_i\md)$ and $M \cong \rho_i^*(\rho_{i*}M)$.
\end{enumerate}
\end{lemma}

Finally, we will need the following simple consequence of the above facts.

\begin{lemma}
\label{lemma:rho-surjective}
The restriction of the functor $\rho_{i*}$ to the bounded derived category is essentially surjective, i.e.,
\begin{equation*}
\rho_{i*}(\Db(\Lambda_i\md)) = \Db(K_i\md).
\end{equation*}
Moreover,
\begin{equation}
\label{eq:rho-s0}
\rho_{i*}(S_{i,0}) \cong \Bbbk,
\qquad
\rho_{i*}(P_{i,0}) \cong K_i.
\end{equation}
\end{lemma}
\begin{proof}
As before for $a \in \Z$ let $\tau^{\ge a}$ denote the canonical truncation of a complex in degree~$a$.
The functor $\rho_{i*}$ is exact with respect to the standard t-structures,
because the module $P_{i,0}$ is projective over $\Lambda_i$, hence
\begin{equation*}
\rho_{i*}(\tau^{\ge a}(M)) \cong \tau^{\ge a}(\rho_{i*}(M))
\end{equation*}
for any $M \in \Dm(\Lambda_i\md)$ and any $a \in \Z$.

If $N \in \Db(K_i\md)$ then $\tau^{\ge a}(N) = N$ for~$a \ll 0$
and if we set $M = \tau^{\ge a}(\rho_i^*(N))$ then
\begin{equation*}
\rho_{i*}(M) {} =
\rho_{i*}(\tau^{\ge a}(\rho_i^*(N))) \cong
\tau^{\ge a}(\rho_{i*}(\rho_i^*(N)) \cong
\tau^{\ge a}(N) = N.
\end{equation*}
It remains to note that $M \in \Db(\Lambda_i\md)$ since $\rho_i^*$ is right exact.

The first isomorphism in~\eqref{eq:rho-s0} follows from~\eqref{eq:ext-p-s}
and $\RHom_{\Lambda_i}(P_{i,0},P_{i,0}) = K_i$ by definition of $K_i$.
\end{proof}

For completeness we also include a brief explicit description of the algebras $K_i$
using the intersection data of the corresponding chain $D_i$ of rational curves $E_{i,1}, \dots, E_{i,m_i}$
(and for details we refer to~\cite{KK}).
We look at one such algebra, and to ease notation we drop the $i$ subscripts.
Recall that the intersection data of the chain is encoded in the Hirzebruch--Jung continued
fraction~{\eqref{eq:ra-hj-cf}},
where $E_{p}^2 = -d_{p}$, and there is also an associated ``dual''
(see~\cite[Section~3]{R74} or~\cite[Section~6]{KK})
Hirzebruch--Jung fraction
\begin{equation*}
\frac r{r-a} = [c_{1}, \dots, c_{l}]
\end{equation*}
determining the integers $l$, $c_1, \dots, c_l$.
This data can be used to give a presentation of the algebra $K=K(r,a)$.

\begin{lemma}[{\cite[Corollary 6.27, Proposition~6.28]{KK}}]
\label{lemma:kk}
The Kalck--Karmazyn algebra $K=K(r,a)$ associated to the cyclic surface singularity $\frac1{r}(1,a)$
is a local, monomial algebra,
$\dim_\Bbbk(K(r,a)) = r$, and
\begin{equation*}
K(r,a) \cong \frac{\Bbbk \langle z_1, \dots, z_l \rangle}
{\left\langle
\begin{array}{rl}
z_j^{c_j} = 0 & \text{for all $j$} \\
z_j z_k=0 & \text{for all $j<k$} \\
\Big(  z_j^{c_j-1} \Big)
\Big(  z_{j-1}^{c_{j-1}-2} \Big)
 {\dots} 
 \Big(  z_{k+1}^{c_{k+1}-2} \Big)
 \Big(  z_k^{c_k-1} \Big)
 =0 & \text{for all $j > k$}
\end{array}
\right\rangle}\ ,
\end{equation*}
where $\Bbbk \langle z_1, \dots, z_l \rangle$ is the free associative algebra on generators~$z_1,\dots,z_l$ and
the parameters $l \ge 1$ and the $c_j \ge 2$ are  defined by
the dual Hirzebruch--Jung continued fraction expansion $r/(r-a)=[c_1, \dots, c_l]$.
\end{lemma}


\begin{example}
\label{example:k}
The Hirzebruch--Jung continued fractions
\begin{equation*}
\frac{r}{r-1}=\underbrace{[2,\dots, 2]}_{r-1} \quad \quad \quad \quad
\frac{r}{1} = [r].
\end{equation*}
are dual.
\begin{enumerate}
\item
The intersection data $[2, \dots, 2]$ corresponds to the $\frac1{r}(1,r-1)$ singularity,
the dual continued fraction is~$[r]$, and the corresponding Kalck--Karmazyn algebra is
\begin{equation*}
K(r,r-1)\cong \Bbbk[z]/z^{r}.
\end{equation*}
\item
The intersection data $[r]$ corresponds to the $\frac1{r}(1,1)$ singularity,
the dual continued fraction is~$[2, \dots, 2]$ and the algebra is
\begin{equation*}
K(r,1) \cong \Bbbk[z_1, \dots, z_{r-1}]/(z_1, \dots, z_{r-1})^2.
\end{equation*}
\end{enumerate}
In all other cases the algebra $K(r,a)$ is not commutative. One example is the following:
\begin{enumerate}
\item[(3)] The intersection data $7/5=[2,2,3]$ corresponds to the $\frac1{7}(1,5)$ singularity;
it has dual fraction $7/(7-5)=[4,2]$, and the algebra is
\begin{equation*}
K(7,5) \cong \frac{\Bbbk\langle z_1, z_2\rangle}{\langle z_1^4, z_2^2,  z_1z_2, z_2z_1^3 \rangle}.
\end{equation*}
\end{enumerate}
\end{example}

\begin{remark}
\label{remark:kk-duality}
As can be seen from its presentation, the algebra $K(r,a)$ depends on the directional ordering
of the chain of divisors $E_{1}, \dots, E_{m}$ determining the intersection data $r/a = [d_1, \dots, d_m]$
and dual fraction $r/(r-a)=[c_1, \dots, c_l]$.
Swapping the direction of the chain yields the fraction $r/a'=[d_m, \dots, d_1]$ and dual fraction~$r/(r-a')=[c_l, \dots, c_1]$
where $a a' \equiv 1 \mod r$, and it is explicit from the algebra presentation that~$K(r,a) \cong K(r,a')^{\textnormal{opp}}$.
Note that the corresponding cyclic quotient singularities~$\frac{1}{r}(1,a)$ and $\frac{1}{r}(1,a')$ are isomorphic;
at the same time the vector spaces duality induces an equivalence
of abelian categories
\begin{equation*}
K(r,a')\md \cong K(r,a')^{\textnormal{opp}}\md \cong  K(r,a)\md.
\end{equation*}
\end{remark}

\subsection{Kalck--Karmazyn algebras and the components $\cA_i$}
\label{subsection:cai}

Assuming that a component $\tcA_i \subset \Db(\tX)$ of a semiorthogonal
decomposition~\eqref{eq:dbtx} compatible with $\pi$
is \emph{untwisted} adherent to a chain of rational curves $D_i$,
we describe the corresponding subcategory $\cA_i \subset \Db(X)$.
Recall that the functors $\tilde\gamma_i$ from~\eqref{eq:def-tgamma}
are well-defined on the category $\Dm(\Lambda_i\md)$, see Remark~\ref{remark:tgammai}.

\begin{theorem}
\label{theorem:sod-cai}
Let $X$ be a normal projective surface satisfying~\eqref{eq:hotx} with rational singularities
and let $\pi \colon \tX \to X$ be its minimal resolution.
Let~\eqref{eq:dbtx} be a semiorthogonal decomposition compatible with $\pi$
and let~\eqref{eq:decomposition-d} be the corresponding decomposition of the exceptional divisor~$D$ of~$\pi$.
Assume that one of its components~$D_i \subset \tX$ is a chain of rational curves
and the corresponding subcategory $\tcA_i \subset \Db(\tX)$ is untwisted adherent to $D_i$.
Let~$\Lambda_i$, $K_i$, $\tilde\gamma_i$, $\rho_i^*$, and~$\rho_{i*}$ be the corresponding algebras and functors
defined by~\eqref{eq:def-lambda}, \eqref{eq:def-k}, \eqref{eq:def-tgamma}, and~\eqref{eq:def-rho}.
Then the functor
\begin{equation}
\label{def:gamma}
\gamma_i := \pi_* \circ \tilde\gamma_i \circ \rho_i^* \colon \Dm(K_i\md) \to \Dm(X).
\end{equation}
is fully faithful, preserves boundedness, and
induces an equivalence
\begin{equation*}
\gamma_i \colon \Db(K_i\md) \xrightarrow{\ \simeq\ } \cA_i := \pi_*(\tcA_i) \subset \Db(X)
\end{equation*}
onto the component $\cA_i$ of the induced semiorthogonal decomposition of~$\Db(X)$.
\end{theorem}

Note that in case of empty $D_i$ (see Remark~\ref{remark:adherent-empty}),
so that $\tcA_i$ is generated by an exceptional line bundle $\cL_{i,0}$,
the theorem just says that $\cA_i$ is generated by an exceptional object $\pi_*(\cL_{i,0})$.

\begin{proof}
We keep the notation introduced in the proof of Theorem~\ref{theorem:dbx}.
In particular, we denote by $\tilde\alpha_i$ the projection functor to $\tcA_i^-$.
Consider the diagrams
\begin{equation}
\label{eq:gamma-tgamma}
\vcenter{\xymatrix@C=3em{
\Dm(\Lambda_i\md) \ar[r]^-{\tilde\gamma_i} \ar[d]_{\rho_{i*}} &
\Dm(\tX) \ar[d]^{\pi_{*}}
\\
\Dm(K_i\md) \ar[r]^-{\gamma_i} &
\Dm(X)
}}
\quad\text{and}\quad
\vcenter{\xymatrix@C=3em{
\Dm(\Lambda_i\md) \ar[r]^-{\tilde\gamma_i} &
\Dm(\tX)
\\
\Dm(K_i\md) \ar[r]^-{\gamma_i} \ar[u]^{\rho_i^{*}} &
\Dm(X) \ar[u]_{\tilde\alpha_i \circ \pi^{*}}
}}
\end{equation}
We will show that both of them are commutative.

By definition $\gamma_i \circ \rho_{i*} = \pi_* \circ \tilde\gamma_i \circ \rho_i^* \circ \rho_{i*}$,
so for commutativity of the first diagram it is enough to check that for any object $M \in \Dm(\Lambda_i\md)$
the cone $M'$ of the canonical morphism $\rho_i^*(\rho_{i*}M) \to M$ is killed by the functor $\pi_* \circ \tilde\gamma_i$.
Indeed, by~\eqref{eq:rho-rho} we have~$\rho_{i*}(M') = 0$,
hence by~\eqref{eq:ker-rho} we have $M' \in \langle S_{i,1}, \dots, S_{i,m_i} \rangle^{\oplus}$,
hence by~\eqref{eq:tgamma-p-s} we have $\tilde\gamma_i(M') \in \langle \cO_{E_{i,1}}(-1), \dots, \cO_{E_{i,m_i}}(-1) \rangle^{\oplus}$,
and hence by~\eqref{eq:tcaim-kerpi} we finally have~$\pi_*(\tilde\gamma_i(M')) = 0$.

For commutativity of the second diagram note that by~\eqref{eq:ker-rho} the
image of the functor~$\rho_i^*$
is the left orthogonal ${}^\perp\langle S_{i,1},\dots, S_{i,m_i} \rangle \subset \Dm(\Lambda_i\md)$
of the subcategory generated by the simple modules~$S_{i,1},\dots, S_{i,m_i}$.
Therefore, from~\eqref{eq:tgamma-p-s} and full faithfulness of the functors $\tilde\gamma_i$
we deduce that the image of the functor $\tilde\gamma_i \circ \rho_i^*$
is contained in the left orthogonal ${}^\perp \langle \cO_{E_{i,1}}(-1), \dots, \cO_{E_{i,m_i}}(-1) \rangle \subset \tcA_i^-$,
which by~\eqref{eq:image-ap} is equal to~$\tilde\alpha_i(\pi^*(\cA_i^-))$.
Since $\pi_* \circ \tilde\alpha_i \circ \pi^* \cong \id$
on~$\cA_i^-$ by Proposition~\ref{proposition:sod-dbx-}$(i)$,
we conclude that $\tilde\alpha_i \circ \pi^* \circ \pi_* \cong \id$ on $\tilde\alpha_i(\pi^*(\cA_i^-))$, hence
\begin{equation*}
\tilde\gamma_i \circ \rho_i^* \cong
\tilde\alpha_i \circ \pi^* \circ \pi_* \circ \tilde\gamma_i \circ \rho_i^* \cong
\tilde\alpha_i \circ \pi^* \circ \gamma_i,
\end{equation*}
so the second diagram commutes.

Now let us show that the functor $\gamma_i$ is fully faithful.
Indeed, by commutativity of the second diagram in~\eqref{eq:gamma-tgamma}, this follows from full faithfulness
of the functors $\rho_i^*$ (see~\eqref{eq:rho-rho}), $\tilde\gamma_i$ (Theorem~\ref{theorem:hille-ploog}),
and $\tilde\alpha_i \circ \pi^*$ (Corollary~\ref{cor:ap-p-adjunction}).

Next, let us check that $\gamma_i$ preserves boundedness.
Indeed, take any $N \in \Db(K_i\md)$.
By Lemma~\ref{lemma:rho-surjective} there exists $M \in \Db(\Lambda_i\md)$ such that~$N \cong \rho_{i*}(M)$.
Then by commutativity of the first diagram in~\eqref{eq:gamma-tgamma} we have
$\gamma_i(N) \cong \gamma_i(\rho_{i*}(M)) \cong \pi_*(\tilde\gamma_i(M))$.
But~$\tilde\gamma_i(M)$ is bounded by Theorem~\ref{theorem:hille-ploog}, hence so is $\pi_*(\tilde\gamma_i(M)) {{}\cong{}} \gamma_i(N)$.

It remains to show that $\gamma_i(\Db(K_i\md)) = \cA_i$.
For this we restrict the first commutative diagram to the bounded derived categories,
and note that $\tilde\gamma_i$ is essentially surjective onto $\tcA_i$ by Theorem~\ref{theorem:hille-ploog},
and $\pi_*$ restricted to $\tcA_i$ is essentially surjective onto~$\cA_i$ by Theorem~\ref{theorem:dbx}.
\end{proof}

Now with this new technique we develop Example~\ref{ex:one-sing} in a special case.

\begin{example}\label{ex:hirzebruch}
Let $X = \P(1,1,d)$ be a weighted projective plane and let
\begin{equation*}
\pi \colon \tX = \F_d \to \P(1,1,d) = X
\end{equation*}
be its minimal resolution by the Hirzebruch surface $\F_d = \P_{\P^1}(\cO \oplus \cO(d))$.
The summands of~$\cO \oplus \cO(d)$ provide two disjoint divisors $E,C \subset \F_d$
with self-intersections~$-d$ and~$d$ respectively (so that $E$ is the exceptional divisor of $\pi$).
Let also $H$ be the pullback of the point class from $\P^1$.
Using the projective bundle formula and Lemma~\ref{lemma:twisted-compatible} (note that~$K_{\F_d} = -E-C-2H$)
we can write the following full exceptional collection on~$\F_d$:
\begin{equation*}
\Db(\F_d) = \langle \OO(-H-E), \OO(-H), \OO, \OO(C) \rangle.
\end{equation*}
We use this exceptional collection to describe the derived category of $\P(1,1,d)$.
We let
\begin{equation*}
\tcA_1 = \langle \OO(-H-E), \OO(-H) \rangle,
\qquad
\tcA_2 = \langle \OO \rangle,
\qquad
\tcA_3 = \langle \OO(C) \rangle;
\end{equation*}
then the semiorthogonal decomposition
\begin{equation*}
\Db(\F_d) = \langle \tcA_1, \tcA_2, \tcA_3 \rangle
\end{equation*}
is compatible with $\pi$ and untwisted adherent to the components $D_1 = E$, $D_2 = D_3 = \varnothing$ of the exceptional divisor of~$\pi$.
Using Theorem \ref{theorem:sod-cai} we deduce that
\begin{equation*}
\Db(\P(1,1,d)) = \langle \cA_1, \cA_2, \cA_3 \rangle,
\end{equation*}
with components $\cA_1 \simeq \Db(K(d,1)\md)$, $\cA_2 \simeq \cA_3 \simeq \Db(\Bbbk)$.
Explicitly, $\cA_2$ and~$\cA_3$ are generated by the exceptional line bundles $\OO$ and $\OO(d)$ on $\P(1,1,d)$
and $\cA_1$ is generated by~$R = \pi_*(\OO(-H-E)) \simeq \pi_*(\OO(-H)) {} \cong \cO(-1)$,
a reflexive sheaf of rank~1
(this agrees up to a line bundle twist with the decomposition described in Example~\ref{ex:p1ab}).
\end{example}

When we have a semiorthogonal decomposition of $\Db(\tX)$ into adherent components
a twist by a line bundle keeps the adherence property but changes the adherence twists.
It is a natural question whether in this way all components may be made untwisted,
and hence compatible with the contraction, thus providing a finite dimensional algebra description for all of the $\cA_i$.
This question is addressed in Section~\ref{section:brauer},
and now we summarize the results of Theorems~\ref{theorem:dbx} and~\ref{theorem:sod-cai}
in the case when such untwisting is possible.

\begin{corollary}
\label{corollary:dbx}
Let $X$ be a normal projective surface satisfying~\eqref{eq:hotx} with cyclic quotient singularities
and let $\pi \colon \tX \to X$ be its minimal resolution of singularities.
Let
\begin{equation*}
\Db(\tX) = \langle \tcA_1, \tcA_2, \dots, \tcA_n \rangle
\end{equation*}
be a semiorthogonal decomposition compatible with $\pi$,
such that every component $\tcA_i$ is untwisted adherent to a chain of rational curves $D_i$
where $D = \sqcup D_i$ is the exceptional locus of $\pi$.
Let $K_i$ and $\gamma_i$ be the corresponding algebras and functors defined by~\eqref{eq:def-k} and~\eqref{def:gamma}.
Then the functors~$\gamma_i$ induce a semiorthogonal decomposition
\begin{equation}
\label{eq:dbx-dbk}
\Db(X) = \langle \Db(K_1\md), \Db(K_2\md), \dots, \Db(K_n\md) \rangle.
\end{equation}
Let $\cA_i = \pi_*(\tcA_i) = \gamma_i(\Db(K_i\md))$ and $\cA_i^\perf = \cA_i \cap \Dperf(X)$.
For each $1 \le i \le n$ if~$\pi$ is crepant along $D_j$ for $j > i$ then $\gamma_i$ also induces an equivalence
\begin{equation*}
\Dperf(K_i\md) \xrightarrow{\ \simeq\ } \cA_i^\perf.
\end{equation*}
In particular, if $\pi$ is crepant along $D_j$ for $j \ge 2$
there is a semiorthogonal decomposition
\begin{equation}
\label{eq:perfx-perfk}
\Dperf(X) = \langle \Dperf(K_1\md), \Dperf(K_2\md), \dots, \Dperf(K_n\md) \rangle.
\end{equation}
\end{corollary}

Note that we allow some components $D_i$ to be empty, see Remark~\ref{remark:adherent-empty} and Example~\ref{ex:hirzebruch}.

\begin{proof}
The semiorthogonal decomposition~\eqref{eq:dbx-dbk} follows from
a combination of Theorem~\ref{theorem:dbx} together with Theorem~\ref{theorem:sod-cai}
applied separately for each component $\cA_i$.

Now assume that $\pi$ is crepant along $D_j$ for $j > i$.
For any $N \in \Db(K_i\md)$ we have a chain of equivalences
\begin{align*}
N \in \Dperf(K_i\md)
&\iff \rho_i^*N \in \Db(\Lambda_i\md)
&& \text{(by Lemma~\ref{lemma:descent-algebraic})}\\
&\iff \tilde\gamma_i(\rho_i^*N) \in \Db(\tX)
&& \text{(by Theorem~\ref{theorem:hille-ploog})}\\
&\iff \pi^*(\gamma_iN) \in \Db(\tX)
&& \text{(by~\eqref{eq:gamma-tgamma} and Proposition~\ref{proposition:sod-dbx-})}\\
&\iff \gamma_iN \in \Dperf(X)
&& \text{(by Corollary~\ref{cor:descent}).}
\end{align*}
To be more precise, in the second line we use the fact that the functor $\tilde\gamma_i$ and its adjoint
provide an equivalence between $\Dm(\Lambda_i\md)$ and $\tcA_i^-$, see Remark~\ref{remark:tgammai},
restricting to an equivalence between $\Db(\Lambda_i\md)$ and $\tcA_i$.
Similarly,
in the third line we use the fact that $\pi^*(\gamma_iN) \subset \tcA_i^-$ by
Proposition~\ref{proposition:sod-dbx-}$(iii)$,
hence $\pi^*(\gamma_iN) \cong \tilde\alpha_i(\pi^*(\gamma_iN)) \cong \tilde\gamma_i(\rho_i^*N)$
by commutativity of~\eqref{eq:gamma-tgamma},
and in the fourth line we use that~$\gamma_iN \in \Db(X)$ by Theorem~\ref{theorem:sod-cai}.

Since~$\gamma_i$ is an equivalence of $\Db(K_i\md)$ and $\cA_i$,
we conclude that it also induces an equivalence of $\Dperf(K_i\md)$ and $\cA_i^\perf$.
Thus,~\eqref{eq:dpxperf} gives~\eqref{eq:perfx-perfk}.
\end{proof}

\begin{remark}
Assume that $\pi \colon \tX \to X$ is a minimal resolution of a normal projective surface~$X$
satisfying~\eqref{eq:hotx} with cyclic quotient singularities,
and that $D = \sqcup_{i=1}^n D_i$ is the decomposition of its exceptional divisor such that $D_i = E_{i,1} \cup \dots \cup E_{i,m_i}$
is a chain of smooth rational curves with self-intersections $E_{i,p}^2 = -d_{i,p}$.
To apply Corollary~\ref{corollary:dbx} and get the required semiorthogonal decomposition of $\Db(X)$
we need to construct a semiorthogonal decomposition of $\Db(\tX)$ with components $\tcA_i$ adherent to $D_i$.
For this we need to find a sequence of line bundles $\{\cL_{i,0}\}_{i=1}^n$ on $\tX$ with prescribed intersections with~$E_{i,p}$
such that the concatenation of the exceptional collections~$\cL_{i,0}, \cL_{i,1}, \dots, \cL_{i,m_i}$, where
\begin{equation*}
\cL_{i,p} = \cL_{i,0}(E_{i,1} + \dots + E_{i,p})
\end{equation*}
is a full exceptional collection on $\tX$.
These conditions (in a slightly generalized from to include twisted adherence) can be spelled out as
\begin{enumerate}
\item $\Ext^\bullet(\cL_{i,0},\cL_{j,0}) = 0$ {for $i > j$};
\item
$\cL_{i,0} \cdot E_{j,p} =
\begin{cases}
d_{j,p} + b_{j,p} - 2, & \text{if $i < j$};\\
d_{j,p} + b_{j,p} - 2 + \delta_{p,1}, & \text{if $i = j$};\\
b_{j,p}, & \text{if $i > j$};\\
\end{cases}$
\item the collection of line bundles $\{\{\cL_{i,p}\}_{p=0}^{m_i}\}_{i=1}^n$ is full.
\end{enumerate}
As we will see in Corollary~\ref{corollary:restriction-exists} below,
the second condition can be always fulfilled if the class group $\Cl(X)$ is torsion-free.
Moreover, the collection of line bundles $\cL_{i,0}$ satisfying this condition is unique
up to a twist of each $\cL_{i,0}$ by a line bundle pulled back from~$X$.
The hard question is to choose these twists in a way to satisfy the first and the last conditions above.
In Section~\ref{section:toric-surfaces} (see Proposition~\ref{proposition:toric-sequence})
we will explain how to do this when $X$ is a toric surface.
\end{remark}


\section{Brauer group of singular rational surfaces}
\label{section:brauer}

In this section we show that for a projective normal rational surface with rational singularities
the torsion subgroup of $\rG_0(X) = \rK_0(\Db(X))$ is dual to the Brauer group~$\Br(X)$
and give an explicit identification of elements of the Brauer group in terms of vector bundles
on the resolution~$\wt{X}$ of~$X$.
We also explain that semiorthogonal decompositions of~$\Db(\tX)$ with components
twisted adherent to connected components of the exceptional divisor
give rise to semiorthogonal decompositions of twisted derived categories~$\Db(X, \beta)$
for $\beta \in \Br(X)$ depending on the twist~$(b_{i,p})$.

\subsection{Grothendieck groups}
\label{subsection:grothendieck-preliminaries}

Let $\cT$ be a triangulated category.
The Grothendieck group $\rK_0(\cT)$ is defined as the quotient of the free abelian group on isomorphism classes of objects
by the subgroup generated by relations $[\cF_2] = [\cF_1] + [\cF_3]$ for all distinguished triangles
\begin{equation*}
\cF_1 \to \cF_2 \to \cF_3
\end{equation*}
in $\cT$.
Note that these relations imply that~\mbox{$[\cF \oplus \cF'] = [\cF] + [\cF']$} and $[\cF[1]] = -[\cF]$
for all objects $\cF$, $\cF' \in \cT$.

For a $\Bbbk$-scheme $X$ we write
\begin{equation*}
\rG_0(X) = \rK_0(\Db(X)),
\qquad
\rK_0(X) = \rK_0(\Dperf(X)).
\end{equation*}
When $X$ is smooth we have $\Db(X) = \Dperf(X)$ and $\rG_0(X) = \rK_0(X)$.

The Grothendieck group is additive with respect to semiorthogonal decompositions:
if $\cT = \langle \cA_1, \cA_2, \dots, \cA_n \rangle$, then
\begin{equation*}
\rK_0(\cT) = \rK_0(\cA_1) \oplus \rK_0(\cA_2) \oplus \dots \oplus \rK_0(\cA_n).
\end{equation*}
This implies a very simple necessary condition for the existence of~\eqref{eq:dbx-dbk}.

\begin{lemma}
\label{lemma:semiorth-torsion}
If a $\Bbbk$-scheme $X$ admits a semiorthogonal decomposition
\begin{equation*}
\Db(X) = \langle \Db(K_{1}\md), \Db(K_{2}\md), \dots, \Db(K_{n}\md) \rangle
\end{equation*}
into derived categories of finitely generated modules over local finite-dimensional algebras,
then $\rG_0(X) \cong \Z^n$.
In particular, $\rG_0(X)$ is torsion-free.
\end{lemma}
\begin{proof}
Since the Grothendieck group is additive for semiorthogonal decompositions,
and the Grothendieck group of finitely-generated modules over a finite-dimensional local algebra is isomorphic to $\Z$
(via the dimension function),
the lemma follows.
\end{proof}

Below we relate torsion in $\rG_0(X)$ with more basic invariants of $X$,
namely with the class group $\Cl(X)$ of Weil divisors.
Recall that there is a natural group homomorphism
\begin{equation*}
\rc_1 \colon \rK_0(X) \to \Pic(X)
\end{equation*}
and if $X$ is normal, then we can also define the homomorphism $\rc_1 \colon \rG_0(X) \to \Cl(X)$ as a composition
\begin{equation}
\label{eq:c1-G0}
\rG_0(X) \to \rG_0(X \setminus \Sing(X)) {} = \rK_0(X \setminus \Sing(X)) \xrightarrow{\ \rc_1\ } \Pic(X  \setminus \Sing(X)) = \Cl(X)
\end{equation}
where $\Sing(X) \subset X$ is the singular locus of $X$ and the first map is induced by the restriction.

For a coherent sheaf $\cF$ we denote by $\chi(\cF)$ its Euler characteristic;
for a complete variety it gives homomorphisms $\chi \colon \rG_0(X) \to \Z$ and $\chi \colon \rK_0(X) \to \Z$.

\begin{lemma}
\label{lem:k0-g0-rational}
Let $X$ be a quasi-projective rational normal surface.
If $X$ is complete, then we have an isomorphism
\begin{equation}
\label{eq:g0}
(\rk, \rc_1, \chi)\colon \rG_0(X) {}\xrightarrow{\ \cong\ }{} \Z \oplus \Cl(X) \oplus \Z,
\end{equation}
and if $X$ is not complete, then we have an isomorphism
\begin{equation}
\label{eq:g0-open}
(\rk, \rc_1)\colon \rG_0(X) {}\xrightarrow{\ \cong\ }{} \Z \oplus \Cl(X).
\end{equation}
\end{lemma}

If $X$ is a smooth projective surface, we call the right side of~\eqref{eq:g0} the {\sf Mukai lattice of $X$}.
In Subsection~\ref{subsec:Groth-twisted} we introduce the Mukai pairing on this lattice.

\begin{proof}
Let us first assume that $X$ is complete.
We compare $\rG_0(X)$ to the Chow groups~$\CH_i(X)$ of $i$-dimensional cycles on~$X$, see~\cite{Fulton-Chow}.
We have $\CH_2(X) = \Z$, $\CH_1(X) = \Cl(X)$ by definition, and since $X$ is rationally connected, $\CH_0(X) = \Z$.

Consider the morphism~\eqref{eq:g0}.
Both its source and target come with a three-step filtration
(the filtration $F^\bullet \rG_0(X)$ by the codimension of support on $\rG_0(X)$
and the filtration induced by the direct sum decomposition on the target).
The map is compatible with these filtrations (since the rank vanishes on objects supported in codimension~1,
and $\rc_1$ vanishes on objects supported in codimension~2).
Therefore, we obtain maps between the factors
\begin{align*}
\rG_0(X) / F^1 \rG_0(X) 	&\xrightarrow{\hbox to 5em{\hfill$\scriptstyle\rk$\hfill}} \Z \\
F^1 \rG_0(X) / F^2 \rG_0(X) &\xrightarrow{\hbox to 5em{\hfill$\scriptstyle\rc_1$\hfill}} \Cl(X) \\
F^2 \rG_0(X) 		&\xrightarrow{\hbox to 5em{\hfill$\scriptstyle\chi$\hfill}} \Z. \\
\intertext{On the other hand, we have maps in the opposite direction defined by}
\Z 	&\xrightarrow{\hbox to 5em{\hfill$\scriptstyle[\cO_X]$\hfill}} \rG_0(X) / F^1 \rG_0(X) \\
\Cl(X) 	&\xrightarrow{\hbox to 5em{\hfill$\scriptstyle D \mapsto [\cO_D]$\hfill}} F^1 \rG_0(X) / F^2 \rG_0(X) \\
\Z	&\xrightarrow{\hbox to 5em{\hfill$\scriptstyle[\cO_x]$\hfill}} F^2 \rG_0(X)
\end{align*}
(in the middle row the map takes the class of an effective Weil divisor $D$ to the class of its structure sheaf;
this map when taken modulo $F^2\rG_0(X)$ respects linear equivalence of Weil divisors).
The second collection of maps is surjective~\cite[Example 15.1.5]{Fulton-Chow}.
On the other hand, it is evident that $\rk(\cO_X) = 1$, $\rc_1(\cO_D) = D$ and $\chi(\cO_x) = 1$,
hence it is also injective.
Therefore, the maps on the factors of the filtration are isomorphisms.
It follows that the original map~\eqref{eq:g0}
is also an isomorphism.

Assume now that $X$ is not complete, and let $X \subset \ol{X}$ be the closure of $X$
in a projective embedding. Let $Z = \ol{X} \setminus X$ be the closed complement of $X$ in $\ol{X}$.
We have the localization sequence of Chow groups \cite[Proposition 1.8]{Fulton-Chow}
\begin{equation*}
\CH_0(Z) \to \CH_0(\ol{X}) \to \CH_0(X) \to 0
\end{equation*}
and since $\ol{X}$ is a rational projective surface we know that $\CH_0(\ol{X}) = \Z$, generated
by any rational point.
As $Z$ is nonempty, the first map in the exact sequence is surjective and we have $\CH_0(X) = 0$.

Thus we have $\CH_2(X) = \Z$, $\CH_1(X) = \Cl(X)$, $\CH_0(X) = 0$ and the same argument as we used
to show~\eqref{eq:g0} proves~\eqref{eq:g0-open}.
\end{proof}

\begin{remark}
Using~\cite[Corollary~1.5]{KS} one can also prove an analogue of this result for~$\rK_0(X)$
under the additional assumption that $X$ has rational singularities:
if $X$ is a quasi-projective rational surface with rational singularities,
then if $X$ is complete, there is an isomorphism
\begin{equation*}
\label{eq:k0}
(\rk, \rc_1, \chi)\colon \rK_0(X) {}\xrightarrow{\ \cong\ }{} \Z \oplus \Pic(X) \oplus \Z.
\end{equation*}
and if $X$ is not complete, then
\begin{equation*}
\label{eq:k0-open}
(\rk, \rc_1)\colon \rK_0(X) {}\xrightarrow{\ \cong\ }{} \Z \oplus \Pic(X).
\end{equation*}
\end{remark}

\subsection{Torsion in rational surfaces}
\label{subsection:torsion-surfaces}

In this subsection we interpret the condition that~$\rG_0(X)$ is torsion-free in terms of
the Brauer group of $X$.

Recall that the Brauer group $\Br(X)$ of a scheme $X$ is the group of
Morita-equivalence classes of Azumaya algebras on $X$ with the operation of tensor product, see~\cite{Milne}.
The Brauer group $\Br(X)$ is closely related to the {\sf cohomological Brauer group} of $X$
\begin{equation*}
\Br'(X) := H^2_\et(X,\Gm);
\end{equation*}
in fact $\Br(X)$ is isomorphic to the subgroup of torsion elements in $\Br'(X)$:
\begin{equation*}
\Br(X) = \Br'(X)_\tors,
\end{equation*}
see~\cite[Cor. 9]{Hoobler} in the surface case and~\cite{deJong} for the general situation.
In particular, if $\Br'(X)$ is a finite group, there is no difference between
the Brauer group and the cohomological Brauer group: $\Br(X) = \Br'(X)$.


\begin{proposition}
\label{prop:torsion}
Let $X$ be a normal projective rational surface with rational singularities.
Then we have the following isomorphisms
\begin{equation*}
\rG_0(X)_{\tors} \simeq \Cl(X)_{\tors}
\end{equation*}
and
\begin{equation*}
\Br'(X) {}={} \Br(X) \simeq \Ext^1(\Cl(X), \Z) {} \simeq \Ext^1(\rG_0(X), \Z).
\end{equation*}
Furthermore, $\Pic(X)$ is free of finite rank and the intersection pairing gives an isomorphism
\begin{equation*}
\Pic(X) \simeq \Cl(X)^\vee.
\end{equation*}
\end{proposition}
\begin{proof}
The isomorphism~\eqref{eq:g0} implies that $\Cl(X)_{\tors} \simeq \rG_0(X)_{\tors}$.

Let $\pi \colon \tX \to X$ be a resolution of singularities.
Denote by
\begin{equation}
\label{eq:cl-tx-x}
\Cl(\tX/X) := \Ker(\Cl(\tX) \xrightarrow{\ \pi_*\ } \Cl(X))
\end{equation}
the kernel of the pushforward map.
Since $\tX$ is smooth, $\Cl(\tX/X)$ is the subgroup of~$\Cl(\tX) = \Pic(\tX)$.
If $E_1, \dots E_m$ are the exceptional curves of $\pi$, then
\begin{equation}
\label{eq:cl-rel}
\Cl(\tX/X) = \bigoplus_{i=1}^m \Z [E_i].
\end{equation}
We consider $\Cl(\tX/X)$ as a sublattice of $\Cl(\tX)$ with respect to the intersection pairing.
We also consider the dual abelian group $\Cl(\tX/X)^\vee$ and denote by $\delta_{E_i}$ the delta-function
corresponding to the generator $E_i$ of $\Cl(\tX/X)$, so that $\delta_{E_i}(E_j) = \delta_{ij}$.

By~\cite[Proposition 1]{Bright} we have an exact sequence
\begin{equation}
\label{eq:pic-sequence}
0 \to \Pic(X) \xrightarrow{\ \pi^*\ } \Pic(\tX) \xrightarrow{\ \IP\ } \Cl(\tX/X)^\vee \to \Br'(X) \xrightarrow{\ \pi^*\ } \Br'(\tX) = 0,
\end{equation}
where $\IP$ is the intersection pairing morphism, that takes a line bundle~$\tilde\cL$ on~$\tX$
to the linear function on $\Cl(\tX/X)$ defined by
\begin{equation}
\label{eq:f_tcl}
f_{\tilde\cL} {} := \IP(\tcL) = \sum_{i=1}^m (\tilde\cL \cdot E_i) \delta_{E_i}.
\end{equation}

On the other hand, by~\eqref{eq:cl-tx-x} we have the following exact sequence
\begin{equation}
\label{eq:cl-pic-cl}
0 \to \Cl(\tX/X) \to \Pic(\tX) \xrightarrow{\ \pi_*\ } \Cl(X) \to 0.
\end{equation}
Its first two terms are free abelian groups (the first is free by~\eqref{eq:cl-rel},
and the second is free since $\tX$ is a smooth rational surface).
Dualizing~\eqref{eq:cl-pic-cl} we obtain an exact sequence
\begin{equation}
\label{eq:cl-sequence}
0 \to \Cl(X)^\vee \to \Pic(\tX)^\vee \to \Cl(\tX/X)^\vee \to \Ext^1(\Cl(X), \Z) \to 0.
\end{equation}

The intersection pairing defines a linear map $\IP \colon \Pic(\tX) \to \Pic(\tX)^\vee$,
which is an isomorphism, since the surface is rational.
Moreover, its composition with the map~\mbox{$\Pic(\tX)^\vee \to \Cl(\tX/X)^\vee$} in~\eqref{eq:cl-sequence}
coincides with the map~$\IP$ in~\eqref{eq:pic-sequence}.
Thus, the exact sequence~\eqref{eq:cl-sequence} can be identified with the sequence~\eqref{eq:pic-sequence}
and in particular, we obtain isomorphisms
\begin{equation*}
\Cl(X)^\vee \simeq \Pic(X)
\qquad\text{and}\qquad
\Ext^1(\Cl(X), \Z) \simeq \Br'(X),
\end{equation*}
where the first isomorphism is induced by the intersection pairing.
In particular, the group~$\Pic(X)$ is free of finite rank, and $\Br'(X)$ is a finite abelian group.
Thus, by~\cite[Cor. 9]{Hoobler} we have
$\Br(X) = \Br'(X) \cong \Ext^1(\Cl(X),\Z)$.
\end{proof}

\begin{remark}
\label{remark:brpic-general}
Exact sequence~\eqref{eq:pic-sequence} gives an isomorphism
\begin{equation*}
\Br(X) \simeq \Coker(\Pic(\wt{X}) \xrightarrow{\ \IP\ } \Cl(\wt{X}/X)^\vee).
\end{equation*}
In Proposition~\ref{proposition:br-cl} below we construct this isomorphism explicitly.
\end{remark}

The following corollary is an immediate consequence of Proposition~\ref{prop:torsion}.

\begin{corollary}
\label{corollary:tf}
If $X$ is a normal projective rational surface with rational singularities then
the following conditions are equivalent:
\begin{enumerate}
\item $\Br(X) = 0$;
\item $\Cl(X)_\tors = 0$;
\item $\rG_0(X)_\tors = 0$.
\end{enumerate}
\end{corollary}

Based on this observation, we suggest the following.

\begin{definition}
Let $X$ be a normal projective rational surface with rational singularities.
We call $X$ {\sf torsion-free} if any of the equivalent conditions of Corollary~\textup{\ref{corollary:tf}} hold.
\end{definition}

Of course, every smooth projective rational surface is torsion-free.
See~\cite{Bright} for a classification of del Pezzo surfaces with du Val singularities that are not torsion-free.
In Section~\ref{section:toric-surfaces} (see Lemma~\ref{lemma:toric-br}) we will explain which toric surfaces are torsion-free.

Summarizing the above discussion, we obtain the following criterion.

\begin{corollary}
Let $X$ be a normal projective rational surface with cyclic quotient singularities
and let $\pi \colon \tX \to X$ be its minimal resolution of singularities.
If~$\Db(\tX)$
admits a semi\-orthogonal decomposition
with components untwisted adherent to the connected components of the exceptional divisor $D$ of $\pi$, then $X$ is torsion-free.
\end{corollary}
\begin{proof}
Assume that $\Db(\tX)$ has such decomposition.
By Corollary~\ref{corollary:dbx} we have a semiorthogonal decomposition~\eqref{eq:dbx-dbk} for~$\Db(X)$,
hence by Lemma~\ref{lemma:semiorth-torsion} the Grothendieck group~$\rG_0(X)$ is torsion-free.
Therefore, $X$ is torsion-free.
\end{proof}

\subsection{Explicit identification of the Brauer group}
\label{subsection:explicit-brauer}

In this section we construct the isomorphism of Remark~\ref{remark:brpic-general} explicitly.
Let $X$ be a normal projective rational surface with rational singularities
and let $\pi \colon \tX \to X$ be its resolution of singularities.
As in the proof of Proposition~\ref{prop:torsion}, we consider the lattice~\eqref{eq:cl-rel}
and denote by $\delta_{E_i}$ the basis of the dual group $\Cl(\tX/X)^\vee$.

Let $\cR$ be an Azumaya algebra on the surface $X$ of rank $r^2$ for some $r > 0$.
Then its pullback~$\tcR := \pi^*\cR$ to the resolution $\tX$ of $X$ is also an Azumaya algebra.
But $\tX$ is a smooth rational surface, so~$\Br(\tX) = 0$, hence~$\tcR$ is Morita-trivial.
This means that there exists a vector bundle~$\tcV$ on $\tX$ such that
\begin{equation}
\label{eq:cr-ce}
\tcR \cong \CEnd(\tcV) \cong \tcV \otimes \tcV^\vee.
\end{equation}
In this case we say that $\tcV$ {\sf splits} or {\sf trivializes} the Azumaya algebra $\tcR$.
Note that such a bundle $\tcV$ has rank $r$ and is defined up to a line bundle twist.

For each irreducible exceptional divisor $E_i$ of the resolution $\pi$ consider the restriction~$\tcV\vert_{E_i}$.
Since $\tcR\vert_{E_i} \cong (\pi^*\cR)\vert_{E_i}$ is a trivial bundle
and $E_i \cong \P^1$, it follows from~\eqref{eq:cr-ce} that
\begin{equation*}
\tcV\vert_{E_i} \cong \cO_{E_i}(b_{E_i})^{\oplus r}
\end{equation*}
for some $b_{E_i} \in \Z$.
Consider the linear function
\begin{equation}
\label{eq:fr}
f_\cR := \sum b_{E_i} \delta_{E_i} \in \Cl(\tX/X)^\vee.
\end{equation}
Denote by $\bar{f}_\cR$ the image of $f_\cR$ in the group $\Coker(\Pic(\tX) \xrightarrow{\ \IP\ } \Cl(\tX/X)^\vee)$.

\begin{proposition}
\label{proposition:br-cl}
The map
\begin{equation}
\label{eq:map-br-coker}
\Br(X) \to \Coker(\Pic(\wt{X}) \xrightarrow{\ \IP\ } \Cl(\tX/X)^\vee),
\qquad
[\cR] \mapsto \bar{f}_\cR
\end{equation}
is well defined, and is an isomorphism of groups.
\end{proposition}

\begin{proof}
First, note that if we replace the vector bundle $\tcV$ in~\eqref{eq:cr-ce} by a line bundle twist~$\tcV \otimes \tcL$,
then the function $f_\cR$ will change by the function $f_\tcL = \IP(\tcL)$ defined by~\eqref{eq:f_tcl},
hence its image~$\bar{f}_\cR$ in $\Coker(\Pic(\wt{X}) \xrightarrow{\ \IP\ } \Cl(\tX/X)^\vee)$ will not change.

Next, let us replace $\cR$ by a Morita-equivalent Azumaya algebra $\cR'$ and let $\cP$ be the bimodule providing a Morita-equivalence.
Then $\pi^*\cP$ provides a Morita equivalence
between the Morita-trivial Azumaya algebras $\pi^*\cR \cong \CEnd(\tcV)$ and~$\pi^*\cR' \cong \CEnd(\tcV')$, hence
\begin{equation*}
\pi^*\cP \cong \tcV \otimes \tcL \otimes (\tcV')^\vee
\end{equation*}
for a line bundle $\tcL$.
Replacing $\tcV$ by $\tcV \otimes \tcL$, we may assume there is no $\tcL$ factor.
If~$\tcV\vert_{E_i} = \cO_{E_i}(b_{E_i})^{\oplus r}$ and $\tcV'\vert_{E_i} = \cO_{E_i}(b'_{E_i})^{\oplus r'}$,
from triviality of $\pi^*\cP$ on $E_i$ it follows that $b_{E_i} = b'_{E_i}$, hence $f_\cR = f_{\cR'}$.
This proves that the map~\eqref{eq:map-br-coker} is well defined.

Similarly, if $\cR \cong \cR_1 \otimes \cR_2$ and the bundles $\tcV_1$ and $\tcV_2$ trivialize the Azumaya algebras~$\pi^*\cR_1$ and $\pi^*\cR_2$,
then the bundle $\tcV = \tcV_1 \otimes \tcV_2$ trivializes $\pi^*\cR$, and
\begin{equation*}
\tcV\vert_{E_i} \cong \tcV_1\vert_{E_i} \otimes \tcV_2\vert_{E_i}
\cong \cO_{E_i}(b^1_{E_i})^{\oplus r_1} \otimes \cO_{E_i}(b^2_{E_i})^{\oplus r_2}
\cong \cO_{E_i}(b^1_{E_i} + b^2_{E_i})^{\oplus r_1r_2},
\end{equation*}
which shows that $f_\cR = f_{\cR_1} + f_{\cR_2}$, hence the map~\eqref{eq:map-br-coker} is linear.

Since both $\Br(X)$ and $\Coker(\Pic(\wt{X}) \to \Cl(\tX/X)^\vee)$ are finite groups of the same order (by Remark~\ref{remark:brpic-general}),
to show that the map~\eqref{eq:map-br-coker} is an isomorphism it is enough to check its injectivity.
So, assume that $\cR$ is an Azumaya algebra on $X$ such that $\bar{f}_\cR = 0$.
This means that the algebra $\pi^*\cR$ can be trivialized by a vector bundle $\tcV$ on~$\tX$
that restricts trivially to each exceptional divisor $E_i$ of $\pi$.
By Lemma~\ref{lemma:descent} we conclude that
\begin{equation*}
\tcV \cong \pi^*\cV
\end{equation*}
for a vector bundle $\cV$ on $X$.
Therefore,
\begin{equation*}
\pi^*\cR \cong \CEnd(\tcV) \cong \CEnd(\pi^*\cV) \cong \pi^*\CEnd(\cV),
\end{equation*}
and since $\pi^*$ is fully faithful, we conclude that $\cR \cong \CEnd(\cV)$.
Clearly, this is an isomorphism of algebras, hence $\cR$ is Morita-trivial.
\end{proof}

The following simple consequence of this result will be very useful later.

\begin{corollary}
\label{corollary:restriction-exists}
Let $X$ be a normal projective rational surface with rational singularities,
let~$\pi \colon \tX \to X$ be its resolution,
and let $f = \sum b_{E_i} \delta_{E_i} \in \Cl(\tX/X)^\vee$.
\begin{enumerate}
\item
There is an integer $r > 0$ and a rank $r$ vector bundle $\tcV$ on $\tX$ such that for all~$i$
\begin{equation*}
\tcV\vert_{E_i} \cong \cO_{E_i}(b_{E_i})^{\oplus r}.
\end{equation*}
\item
If $X$ is torsion-free, there is a line bundle $\tcL$ such that for all~$i$
\begin{equation*}
\tcL\vert_{E_i} \cong \cO_{E_i}(b_{E_i});
\end{equation*}
such line bundle is unique up to a twist by the pullback of a line bundle on $X$.
\end{enumerate}
\end{corollary}
\begin{proof}
Let $\cR$ be the Azumaya algebra corresponding to the image of the function $f$ in the group $\Br(X) = \Coker(\Pic(X) \to \Cl(\tX/X)^\vee)$.
Then $f_\cR$ differs from $f$ by $f_\tcL$ for a line bundle $\tcL$ on $\wt{X}$,
hence we can take $\tcV$ to be an appropriate vector bundle trivializing the algebra $\pi^*\cR$.

If $X$ is torsion-free the map $\IP \colon \Pic(\tX) \to \Cl(\tX/X)^\vee$ is surjective by~\eqref{eq:pic-sequence},
which means the existence of $\tcL$ as required.
Moreover, $\Ker (\IP) = \pi^*(\Pic(X))$, again by~\eqref{eq:pic-sequence}, which gives the uniqueness of $\tcL$.
\end{proof}

The proof of Proposition~\ref{proposition:br-cl} shows that the map
\begin{equation}
\label{eq:def-rb}
\rB \colon \Cl(\tX/X)^\vee \to \Br(X),
\qquad
f = \sum b_{E_i}\delta_{E_i} \mapsto \cR = \pi_*\CEnd(\tcV),
\end{equation}
where $\tcV$ is a vector bundle from Corollary~\ref{corollary:restriction-exists}(1),
is a well-defined surjective homomorphism whose kernel is the image $\IP(\Pic(\tX))$ of $\Pic(\tX)$
under the intersection pairing map.
We call $\rB$ the {\sf Brauer class map}.

\subsection{Resolutions of twisted derived categories}
\label{subsection:resolution-twisted}

Given an Azumaya algebra $\cR$ on a scheme~$X$ we denote by~$\Coh(X,\cR)$
the abelian category of sheaves of right $\cR$-modules on $X$
which are coherent as $\cO_X$-modules.
By definition, the category~$\Coh(X,\cR)$
up to equivalence depends only on the Brauer class $\beta$ of $\cR$.
Accordingly to this, we will usually denote this category by $\Coh(X,\beta)$.

For any class $\beta \in \Br(X)$ if $\cR$ is an Azumaya algebra on $X$ representing it we denote
\begin{equation*}
\Db(X,\beta) = \Db(X,\cR),\qquad
\Dperf(X,\beta) = \Dperf(X,\cR),\qquad
\Dm(X,\beta) = \Dm(X,\cR)
\end{equation*}
the corresponding bounded, perfect
(note that any sheaf of $\cR$-modules which is locally free over $\cO_X$
is automatically locally projective over $\cR$ (see~\cite[Lemma~10.4]{K06}),
hence a complex of $\cR$-modules is perfect if and only if it is perfect as a complex of $\cO_X$-modules),
and bounded above twisted derived categories respectively.

Now let $X$ be a normal projective rational surface with rational singularities and
let~$\pi \colon \tX \to X$ be its minimal resolution.
As in Section~\ref{subsection:explicit-brauer} we denote by $\tcV$ a vector bundle on $\tX$ trivializing the Azumaya algebra $\tcR = \pi^*\cR$.
Note that $\tcV$ is a left $\tcR$-module and $\tcV^\vee$ is a right $\tcR$-module.

We consider the following pair of functors defined as compositions
\begin{align}
\begin{aligned}
\pi_\tcV^*  &\colon \Dm(X,\cR) \to \Dm(\tX, \tcR) \simeq  \Dm(\tX), \qquad & \cG & \mapsto (\pi^*\cG) \otimes_\tcR \tcV,\\
\pi^\tcV_* &\colon \Dm(\tX) \simeq \Dm(\tX, \tcR) \to \Dm(X,\cR), \qquad & \cF & \mapsto \pi_*(\cF \otimes _{\cO_\tX} \tcV^\vee).
\end{aligned}
\end{align}
As before, we consider the pair of functors $(\pi_\tcV^*,\pi^\tcV_*)$ as a categorical resolution
of the twisted derived category of $X$ by the derived category of $\tX$.
Below we check that it has similar properties to the resolution~$(\pi^*,\pi_*)$ of the untwisted derived category.

\begin{lemma}
The pullback functor $\pi_\tcV^*$ is left adjoint to the pushforward functor $\pi^\tcV_*$.
Furthermore, the pullback $\pi_\tcV^*$ preserves the category of perfect complexes
and the pushforward $\pi^\tcV_*$ preserves the bounded category.
Finally,
\begin{equation*}
\pi^\tcV_* \circ \pi_\tcV^* \cong \id_{\Dm(X,\cR)}.
\end{equation*}
In particular, the pullback functor $\pi_\tcV^*$ is fully faithful.
\end{lemma}
\begin{proof}
The adjunction is standard (note that $\cF \otimes _{\cO_\tX} \tcV^\vee \cong \cRHom(\tcV,\cF)$).
The second statement is evident from the definition of the functors.
For the last one note that
\begin{equation*}
\pi^\tcV_*(\pi_\tcV^*(\cG))
\cong \pi_*((\pi^*\cG) \otimes_\tcR \tcV \otimes_{\cO_\tX} \tcV^\vee)
\cong \pi_*((\pi^*\cG) \otimes_\tcR \tcR)
\cong \pi_*(\pi^*\cG)
\cong \cG
\end{equation*}
and we are done.
\end{proof}

Let $E_1, \dots, E_m$ denote the irreducible components of the exceptional divisor of~$\pi$.
We use the notation from Section~\ref{subsection:explicit-brauer}.

\begin{lemma}
\label{lemma:ker-pi-twisted}
Assume that we have $\tcV\vert_{E_i} \cong \cO_{E_i}(b_{E_i})^{\oplus r}$, so that $f_\cR = \sum b_{E_i}\delta_{E_i}$.
An object~\mbox{$\cF \in \Dm(\tX)$} is contained in $\Ker \pi^\tcV_*$ if and only if every cohomology sheaf of~$\cH^t(\cF)$
is an iterated extension of sheaves~$\cO_{E_i}(-1 + b_{E_i})$.
\end{lemma}
\begin{proof}

It follows immediately from Lemma~\ref{lemma:ker-pi} that an object  $\cF$ is contained in $\Ker \pi^\tcV_*$
if and only if every cohomology sheaf of~$\cH^t(\cF)$ is.
So from now on we will assume that~$\cF$ is a pure sheaf in~\mbox{$\Ker \pi^\tcV_*$}
and we must show that $\cF \in \langle \cO_{E_i}(-1+b_{E_i}) \rangle$.

Using again Lemma~\ref{lemma:ker-pi} we see that $\cF \otimes \tcV^\vee \in \langle \cO_{E_i}(-1) \rangle$.
Let $\cF \otimes \tcV^\vee \to \cO_E(-1)$ be an epimorphism, where $E$ is one of the $E_i$.
By adjunction we obtain a nonzero morphism $\cF \to \tcV \otimes \cO_E(-1)$.
Let $\cF'$ be its cone.
Using the defining triangle of $\cF'$
\begin{equation*}
\cF \to \tcV \otimes \cO_E(-1) \to \cF'
\end{equation*}
it is easy to see that $\pi^\tcV_*(\cF') = 0$, hence from the observation at the beginning of the proof
we conclude that $\cH^t(\cF') \in \Ker \pi^\tcV_*$ for each $t$.

Now let $\cG$ be the image of the (nonzero) morphism of sheaves $\cF \to \tcV \otimes \cO_E(-1)$.
Note that~$\cG$ is a nonzero subsheaf in $\tcV \otimes \cO_E(-1)$,
hence $\cG$ is supported scheme-theoretically on the smooth rational curve~$E$.
Furthermore, we have an exact sequence
\begin{equation}
\label{eq:sequence-ch-fprime}
0 \to \cH^{-1}(\cF') \to \cF \to \cG \to 0
\end{equation}
which implies that $\cG \in \Ker \pi^\tcV_*$.
It follows that $\cG \otimes \tcV^\vee \cong \cO_E(-1)^{\oplus s}$ for some~$s > 0$, hence
\begin{equation*}
\cG \cong \cO_E(-1+b_E)^{\oplus s/r}.
\end{equation*}
On the other hand, the sum of lengths of $\cH^{-1}(\cF')$ at generic points of $E_i$ is less than that for~$\cF$
(since $s > 0$),
hence by induction we have the containment
\begin{equation*}
\cH^{-1}(\cF') \in \langle \cO_{E_i}(-1+b_{E_i}) \rangle.
\end{equation*}
Now the statement follows from~\eqref{eq:sequence-ch-fprime}.
\end{proof}

\begin{lemma}
\label{lemma:im-pi-twisted}
For any $\cG \in \Db(X,\cR)$ there exists $\cF \in \Db(\tX)$ such that $\cG \cong \pi^\tcV_*(\cF)$.
\end{lemma}
\begin{proof}
Analogous to the proof of Corollary~\ref{lemma:im-pi}.
\end{proof}


\begin{lemma}
\label{lemma:descent-twisted}
Let $\cF \in \Db(\tX)$.
The following properties are equivalent:
\begin{enumerate}
\item
for each $i$ one has $\cF\vert_{E_i} \in \langle \cO_{E_i}(b_{E_i}) \rangle$;
\item
for each $i$ one has $\Ext^\bullet(\cF\vert_{E_i}, \cO_{E_i}(-1 + b_{E_i})) = 0$;
\item
there exists $\cG \in \Dperf(X,\cR)$ such that $\cF \cong \pi_\tcV^*\cG$;
\item
one has $\pi^\tcV_*\cF \in \Dperf(X,\cR)$ and $\cF \cong \pi_\tcV^*(\pi^\tcV_*\cF)$.
\end{enumerate}
If additionally $\cF$ is a pure sheaf, or a locally free sheaf, then so is $\pi^\tcV_*\cF$.
\end{lemma}
\begin{proof}
Analogous to the proof of Lemma~\ref{lemma:descent}.
\end{proof}

\subsection{Grothendieck groups of twisted derived categories}\label{subsec:Groth-twisted}

Let $X$ be a normal projective rational surface with rational singularities.
Let $\beta \in \Br(X)$ be a Brauer class.
We denote by $\rG_0(X, \beta) = \rK_0(\Db(X, \beta))$ the Grothendieck group of twisted coherent sheaves on $X$.
By Proposition~\ref{proposition:br-cl} there is an element
\begin{equation*}
f = \sum_{i=1}^m b_{E_i} \delta_{E_i} \in \Cl(\wt{X}/X)^\vee
\end{equation*}
such that $\beta = \rB(f)$, where $\rB$ is the Brauer class map defined in~\eqref{eq:def-rb}.

To study $\rG_0(X, \beta)$
we consider the following subgroup of $\rG_0(\wt{X})$:
\begin{equation}\label{eq:cl-beta}
\rG_0(\wt{X}/X,\beta)
= \big\langle [\OO_{E_1}(-1 + b_{E_1})], \dots, [\OO_{E_m}(-1 + b_{E_m})] \big\rangle \subset \rG_0(\wt{X})
\end{equation}
(see Proposition~\ref{prop:g0-quotient} below for a conceptual interpretation of this subgroup).
Under the isomorphism $(\rk, \rc_1, \chi) \colon \rG_0(\wt{X}) \simeq \Z \oplus \Cl(\wt{X}) \oplus \Z$ of Lemma \ref{lem:k0-g0-rational},
this subgroup can be written as
\begin{equation*}
\rG_0(\wt{X}/X,\beta) \simeq \big\langle (0, [E_1], b_{E_1}), \dots, (0, [E_m], b_{E_m}) \big\rangle \subset \Z \oplus \Cl(\wt{X}) \oplus \Z.
\end{equation*}
Note that when all $b_{E_i} = 0$, so that $\beta = 0$ we have $\rG_0(\wt{X}/X,\beta) = \Cl(\wt{X}/X) \subset \rG_0(X)$.
Also note a slight abuse of notation: strictly speaking $\rG_0(\wt{X}/X,\beta) \subset \rG_0(\wt{X})$
depends not only on $\beta$ but also on the choice of $f \in \Cl(\wt{X}/X)^\vee$ representing $\beta$.
Note however that $\rG_0(\wt{X}/X,\beta) \subset \rG_0(\wt{X})$ is well-defined up to multiplication
by a class of a line bundle on $\wt{X}$.

\begin{proposition}
\label{prop:g0-quotient}
We have an exact sequence of abelian groups
\begin{equation}
\label{eq:g0-sequence}
0 \to \rG_0(\wt{X}/X,\beta) \to \rG_0(\tX) \xrightarrow{\ \pi^\tcV_*\ } \rG_0(X,\beta) \to 0.
\end{equation}
In particular, $\rG_0(X,\beta)$ is a finitely generated abelian group.
\end{proposition}

\begin{remark}
If all $b_{E_i} = 0$, then $\beta = 0$ and~\eqref{eq:g0-sequence} follows
from~\eqref{eq:cl-pic-cl} and Lemma~\ref{lem:k0-g0-rational} as $\rG_0(\wt{X}/X, 0)$ corresponds
to $0 \oplus \Cl(\wt{X}/X) \oplus 0 \subset \Z \oplus \Pic(\wt{X}) \oplus \Z \simeq \rG_0(\wt{X})$.
\end{remark}

\begin{proof}
The morphism $\pi^\tcV_* \colon \rG_0(\tX) \to  \rG_0(X,\beta)$ is surjective by Lemma~\ref{lemma:im-pi-twisted}
and the subgroup $\rG_0(\wt{X}/X,\beta)$ is contained in its kernel by Lemma~\ref{lemma:ker-pi-twisted},
so it remains to identify the kernel of $\pi^\tcV_*$ with the image of the first arrow in~\eqref{eq:g0-sequence}.
Let $E_1, \dots, E_m$ be the components of the exceptional divisor of~$\pi$.
Consider the open subscheme
\begin{equation*}
U := X \setminus \left( \bigcup_{i=1}^m \pi(E_i) \right) = \tX \setminus \left( \bigcup_{i=1}^m E_i \right)
\end{equation*}
and the natural embeddings $u \colon U \to X$ and $\tilde{u} \colon U \to \tX$, so that $\pi \circ \tilde{u} = u$.
The Azumaya algebra $\cR_U := u^*\cR \cong \tilde{u}^*\pi^*\cR \cong \tilde{u}^*\tcR$ is Morita trivial,
and trivialized by the vector bundle $\tcV_U := \tilde{u}^*\tcV$, so we have a functor
\begin{equation*}
u_\tcV^* \colon \Db(X,\cR) \to \Db(U), \qquad
\cG \mapsto (u^*\cG) \otimes_{\cR_U} \tcV_U,
\end{equation*}
such that $u_\tcV^* \circ \pi^\tcV_* \cong \tilde{u}^* \colon \Db(\tX) \to \Db(U)$.
Therefore, we have a commutative diagram
\begin{equation*}
\xymatrix{
\rG_0(\tX) \ar[rr]^-{\pi^\tcV_*} \ar[dr]_{\tilde{u}^*} &&
\rG_0(X,\beta) \ar[dl]^{u_\tcV^*} \\
& \rG_0(U)
}
\end{equation*}
giving an exact sequence
\begin{equation*}
0 \to \Ker \pi^\tcV_* \to \Ker \tilde{u}^* \xrightarrow{\ \pi^\tcV_*\ } \Ker u_\tcV^*.
\end{equation*}
On the other hand, using isomorphisms \eqref{eq:g0}, \eqref{eq:g0-open} of Lemma~\ref{lem:k0-g0-rational}
we have
\begin{equation*}
\Ker \tilde{u}^* = 0 \oplus \Cl(\tX/X) \oplus \Z[\cO_\tx] \subset \Z \oplus \Cl(\tX) \oplus \Z = \rG_0(\tX),
\end{equation*}
where $\tx$ is a point on $\tX$.
So, it remains to note that this subgroup is freely generated by the elements $[\cO_{E_i}(-1 + b_{E_i})]$ and by~$[\cO_\tx]$,
and that $\pi^\tcV_*([\cO_\tx])$ is a non-torsion element of~$\rG_0(X,\beta)$
(for the last fact just note that the composition
\begin{equation*}
\rG_0(\tX) \xrightarrow{\ \pi^\tcV_*\ } \rG_0(X,\beta)
{} = \rG_0(X,\cR) \xrightarrow{\ \mathrm{Forget}\ } \rG_0(X)
\xrightarrow{\ \chi\ } \Z
\end{equation*}
(where $\mathrm{Forget}$ is the linear map induced by the functor forgetting the structure of $\cR$-module)
takes $[\cO_\tx]$ to~\mbox{$\chi(\pi_*(\cO_\tx \otimes \tcV^\vee)) = \rk(\tcV)$}).
\end{proof}

In the next proposition we relate the torsion part of $\rG_0(X,\beta)$ to the Brauer group of $X$
providing a generalization of Proposition~\ref{prop:torsion} to the twisted case.
We use more or less the same proof as that of Proposition~\ref{prop:torsion}, but instead of
$\Cl(X)$ and $\Pic(X)$ we use $\rG_0(X)$ and $\rK_0(X)$ respectively.
Recall the direct sum decomposition
\begin{equation*}
(\rk, \rc_1, \chi) \colon \rK_0(\tX) {} = \rG_0(\tX) \simeq \Z \oplus \Pic(\tX) \oplus \Z
\end{equation*}
described in Lemma~\ref{lem:k0-g0-rational}, and consider on it the {\sf Mukai pairing} defined by:
\begin{equation}\label{eq:def-MP}
\MP((r, D, s), (r', D', s')) := rs' + r's - D \cdot D'
\end{equation}
Note that it is unimodular; in particular it gives an identification $\rK_0(\tX)^\vee \cong \rK_0(\tX)$.
Note also that for elements $\xi_1 = (0,D_1,s_1)$ and $\xi_2 = (0,D_2,s_2)$ of $\rG_0(\wt{X}/X,\beta)$ we have
\begin{equation}
\label{eq:mp-ip}
\MP(\xi_1,\xi_2) = - D_1 \cdot D_2 = - \IP(\rc_1(\xi_1),\rc_1(\xi_2)),
\end{equation}
so the two pairings agree up to a sign.
In particular, the map
\begin{equation*}
\rc_1 \colon \rG_0(\wt{X}/X,\beta) \to \Cl(\tX/X)
\end{equation*}
is an isomorphism of lattices up to a sign.

\begin{proposition}
\label{prop:g0-twisted}
We have a natural isomorphism
\begin{equation*}
\Ext^1(\rG_0(X,\beta), \Z) \simeq \Br(X) / \langle \beta \rangle.
\end{equation*}
In particular $\rG_0(X,\beta)$ is torsion-free if and only if $\beta$ is a generator of $\Br(X)$.
\end{proposition}
\begin{proof}
Using~\eqref{eq:g0-sequence} as a free resolution for $\rG_0(X,\beta)$,
we obtain
\begin{equation*}
\rG_0(\tX)^\vee \to \rG_0(\wt{X}/X,\beta)^\vee \to \Ext^1(\rG_0(X, \beta),\Z) \to 0.
\end{equation*}
Using the Mukai pairing to identify $\rG_0(\tX)^\vee = \rK_0(\tX)^\vee$ with $\rK_0(\tX)$, we rewrite this as
\begin{equation*}
\rK_0(\tX) \xrightarrow{\ \MP\ } \rG_0(\wt{X}/X,\beta)^\vee \to \Ext^1(\rG_0(X, \beta),\Z) \to 0,
\end{equation*}
where the first map takes $[\cF] = (r,D,s)$ to
\begin{equation}
\label{eq:mp-explicit}
\sum \MP([\cF],[\cO_{E_i}(-1 + b_{E_i})]) \delta_{E_i} =
\sum (r \cdot b_{E_i} - D \cdot E_i) \delta_{E_i} =
r \sum b_{E_i}\delta_{E_i} - \IP(D).
\end{equation}
Here $\delta_{E_i} \in \rG_0(\wt{X}/X, \beta)^\vee$ is defined
as the dual basis for the basis $[\OO_{E_i}(-1 + b_i)]$ of~$\rG_0(\wt{X}/X, \beta)$.
Note that we have an isomorphism
\begin{equation*}
\rG_0(\wt{X}/X, \beta)^\vee \cong \Cl(\tX/X)^\vee
\end{equation*}
that identifies their bases $\delta_{E_i}$.
Thus, to compute the cokernel of $\MP$, we should take successively
the quotients of $\Cl(\tX/X)^\vee$ by the images of the three summands of
\begin{equation*}
\rK_0(\tX) = \Z \oplus \Pic(\tX) \oplus \Z.
\end{equation*}
It follows from~\eqref{eq:mp-explicit} that the third summand is mapped to zero.
Furthermore, the map on the second summand agrees up to sign with the map~$\IP$,
hence by Remark~\ref{remark:brpic-general} the quotient is isomorphic to~$\Br(X)$
via the Brauer class map~$\rB$ defined in~\eqref{eq:def-rb}.
Finally, comparing~\eqref{eq:def-rb} with~\eqref{eq:mp-explicit} we see that
the composition $\rB \circ \MP$ takes the generator~$1$ of the first summand of $\rK_0(\tX)$ to~$\rB(\sum b_{E_i} \delta_{E_i}) = \beta$.
Therefore, the cokernel is isomorphic to $\Br(X) / \langle \beta \rangle$.
\end{proof}

\begin{remark}
Propositions \ref{prop:g0-quotient} and \ref{prop:g0-twisted} show how to compute $\rG_0(X, \beta)$
using the sublattice $\rG_0(\wt{X}/X,\beta) \subset \rG_0(\wt{X}) = \rK_0(\wt{X})$.
This sublattice can be also used to compute the group~$\rK_0(X, \beta) = \rK_0(\Dperf(X, \beta))$.
Indeed one can show that there is a natural exact sequence
\begin{equation*}
0 \to \rK_0(X, \beta) \xrightarrow{\ \pi_\tcV^*\ } \rK_0(\wt{X}) \xrightarrow{\ \MP\ }
\rG_0(\wt{X}/X,\beta)^\vee \to \Br(X) / \langle \beta \rangle \to 0.
\end{equation*}
Injectivity of $\pi_\cV^*$ is nontrivial; it can be established as in \cite[Theorem 2.19]{Pavic-Shinder}.
\end{remark}

\subsection{Semiorthogonal decompositions of twisted derived categories}
\label{subsection:adherent-twisted}

Assume that~$X$ is a normal projective rational surface with cyclic quotient singularities,
and let~\mbox{$\pi \colon \tX \to X$} be its minimal resolution
with the exceptional divisor
\begin{equation*}
D = \bigsqcup_{i=1}^n D_i = \bigsqcup_{i=1}^n \left( \bigcup_{p=1}^{m_i} E_{i,p} \right),
\end{equation*}
so that $D_i = \cup_{p=1}^{m_i} E_{i,p}$ for each $i$ is a chain (possibly empty) of smooth rational curves.
The next result is a twisted analogue of Corollary~\ref{corollary:dbx}.

\begin{theorem}
\label{theorem:dbx-twisted}
Let $X$ be a normal projective rational surface with cyclic quotient singularities
and let $\pi \colon \tX \to X$ be its minimal resolution.
Let
\begin{equation*}
\Db(\tX) = \langle \tcA_1, \tcA_2, \dots, \tcA_n \rangle
\end{equation*}
be a semiorthogonal decomposition, such that every component $\tcA_i$ is $(b_{i,p})$-twisted adherent to a chain of rational curves $D_i$
where $D = \sqcup D_i$ is the exceptional locus of $\pi$.
Let $\tcV$ be a vector bundle on $\tX$ such that
\begin{equation*}
\tcV\vert_{E_{i,p}} \cong \cO_{E_{i,p}}(b_{i,p})^{\oplus r}
\end{equation*}
for some positive integer $r$ and all $i$ and $p$, see Corollary~\textup{\ref{corollary:restriction-exists}(1)}.
Let
\begin{equation*}
\beta = \rB\left(\sum b_{i,p} \delta_{E_{i,p}}\right) \in \Br(X)
\end{equation*}
be the corresponding Brauer class.
Let~$K_i$ be the Kalck--Karmazyn algebras associated with the components $D_i$ of $D$, see~\eqref{eq:def-k}.
Define the functors
\begin{equation*}
\gamma^\tcV_i := \pi^\tcV_* \circ \tilde\gamma_i \circ \rho_i^* \colon \Dm(K_i\md) \to \Dm(X,\beta).
\end{equation*}
Then the functors $\gamma^\tcV_i$ are fully faithful, preserve boundedness, and induce a semiorthogonal decomposition
\begin{equation}
\label{eq:dbx-dbk-twisted}
\Db(X,\beta) = \langle \Db(K_1\md), \Db(K_2\md), \dots, \Db(K_n\md) \rangle.
\end{equation}
If $\pi$ is crepant along $D_j$ for $j > i$ then $\gamma^\tcV_i$ also induces a fully faithful functor
\begin{equation*}
\Dperf(K_i\md) \xrightarrow{\ \ } \Dperf(X,\beta),
\end{equation*}
and if $\pi$ is crepant along $D_j$ for $j \ge 2$
there is a semiorthogonal decomposition
\begin{equation}
\label{eq:perfx-perfk-twisted}
\Dperf(X,\beta) = \langle \Dperf(K_1\md), \Dperf(K_2\md), \dots, \Dperf(K_n\md) \rangle.
\end{equation}
\end{theorem}
\begin{proof}
The proof repeats the proof of Corollary~\ref{corollary:dbx} in the twisted setting.
\end{proof}

\begin{corollary}
\label{corollary:beta-generator}
If $\Db(\tX)$ has a semiorthogonal decomposition such that every its component is $(b_{i,p})$-twisted adherent
to a connected component $D_i$ of $\pi$, the corresponding Brauer class $\beta = \rB( \sum b_{i,p} \delta_{E_{i,p}} ) \in \Br(X)$
is a generator of $\Br(X)$.
\end{corollary}
\begin{proof}
If such semiorthogonal decomposition exists, then using~\eqref{eq:dbx-dbk-twisted} and arguing
as in Lemma~\ref{lemma:semiorth-torsion} we see that~$\rG_0(X,\beta)$ is torsion-free.
Then Proposition~\ref{prop:g0-twisted} proves that~$\beta$ generates $\Br(X)$.
\end{proof}


\begin{remark}
\label{remark:components-twist}
Let $\Db(\tX) = \langle \tcA_1, \tcA_2, \dots, \tcA_n \rangle$ be a semiorthogonal decomposition
which is adherent to the exceptional divisor $D = \sqcup D_i$ with a twist $(b_{i,p})$.
Let $\tcL$ be a line bundle on $\tX$.
Then $\Db(\tX) = \langle \tcA_1 \otimes \tcL, \tcA_2 \otimes \tcL, \dots, \tcA_n \otimes \tcL \rangle$ is a semiorthogonal decomposition
which is adherent to the exceptional divisor $D = \sqcup D_i$ with a twist $(b_{i,p} + \tcL \cdot E_{i,p})$.
Note that
\begin{equation*}
\rB\left( \sum (b_{i,p} + \tcL \cdot E_{i,p}) \delta_{E_{i,p}} \right) =
\rB\left( \sum b_{i,p} \delta_{E_{i,p}} + \IP(\tcL) \right) =
\rB\left( \sum b_{i,p} \delta_{E_{i,p}} \right),
\end{equation*}
so the Brauer class corresponding to the new decomposition coincides with the original Brauer class $\beta$.
If $\tcV$ is a vector bundle
as in Corollary~\ref{corollary:restriction-exists}
and $\cR$ is an Azumaya algebra on $X$ such that $\pi^*\cR \cong \CEnd(\tcV)$
then we also have $\pi^*\cR \cong \CEnd(\tcV \otimes \tcL)$, hence
we have a commutative diagram
\begin{equation*}
\xymatrix{
\tcA_i \ar[rr]^-{-\otimes\tcL} \ar[dr]_(.3){\pi^\tcV_*} &&
\tcA_i \otimes \tcL \ar[dl]^(.3){\pi^{\tcV \otimes \tcL}_*}
\\
&
\Db(X,\cR),
}
\end{equation*}
which shows that with these choices we obtain the same semiorthogonal decomposition of $\Db(X,\beta) {} = \Db(X,\cR)$ in the end.
On the other hand, if $\cL$ is a line bundle on~$X$, we can replace the functor $\pi^\tcV_*$ by $\pi^{\tcV \otimes \pi^*\cL}_*$.
Then we have a commutative diagram
\begin{equation*}
\xymatrix{
&
\tcA_i \ar[dr]^{\pi^\tcV_*} \ar[dl]_{\pi^{\tcV \otimes \pi^*\cL}_*}
\\
\Db(X,\cR) \ar[rr]^-{-\otimes\cL} &&
\Db(X,\cR),
}
\end{equation*}
which shows that a different choice of the resolution functor results in a twist
of the resulting semiorthogonal decomposition of $\Db(X,\cR)$.
%
%
\end{remark}

\section{Application to toric surfaces}
\label{section:toric-surfaces}

In this section we apply the results of previous sections to projective toric surfaces.
We refer to~\cite{Danilov, Fulton} for general information about toric varieties.
While reading this section it is instructive to keep Example~\ref{ex:hirzebruch} in mind.

\subsection{Notation}

A toric surface is an irreducible normal surface $X$ endowed with an action of a two-dimensional torus $\rT \cong \Gm^2$ with a free $\rT$-orbit.
In particular, every toric surface is rational.
We only consider toric surfaces which are projective.
Note that for each resolution $\tX$ of such surface the condition~\eqref{eq:hotx} is satisfied.

We denote by~$\rM \cong \Z^2$ the lattice of characters of $\rT$, and by $\rN {}:={} \rM^\vee \cong \Z^2$ its dual lattice.
A projective toric surface $X$ is determined by a complete fan $\Sigma$ in $\rN \otimes \R$.
We denote the primitive generators of the rays (one-dimensional cones) in the fan by
\begin{equation}
\label{eq:fan}
v_1,v_2,\dots,v_n \in \rN.
\end{equation}
We assume that $v_i$ are indexed in the counterclockwise order on the plane $\rN \otimes \R$.
We also set
\begin{equation*}
v_{n+1} = v_1.
\end{equation*}
The rays of $\Sigma$
correspond to irreducible torus-invariant divisors $C_i \subset X$.
Since $X$ is projective, each $C_i$ is isomorphic to $\P^1$.

Similarly, two-dimensional cones in the fan $\Sigma$ correspond to $\rT$-invariant points on~$X$.
We denote by $x_i \in X$ the $\rT$-invariant point corresponding to the cone~$(v_i,v_{i+1})$ of the fan, so that
\begin{equation}
\label{eq:x-c-c}
x_i = C_i \cap C_{i+1},
\end{equation}
and let $U_i$
be the toric chart containing $x_i$.
Since the vector $v_i$ is primitive, we can choose a basis in $\rN$ such that $v_i = (1,0)$.
Then we can write $v_{i+1} = (r_i - a_i,r_i)$, where
\begin{equation}
\label{eq:ri}
r_i = \det(v_i,v_{i+1})
\end{equation}
is positive and $a_i \in \Z$.
Moreover, changing the second basis vector in $\rN$ and taking into account the primitivity of $v_{i+1}$,
we can assume that
\begin{equation*}
0 \le a_i < r_i
\qquad\text{and}\qquad
\gcd(r_i,a_i)  = 1.
\end{equation*}
Then the dual cone to $\R_{\ge 0} \cdot v_i + \R_{\ge 0} \cdot v_{i+1} {} \subset \rN \otimes \R$ is isomorphic to the cone
generated by $(1,0)$, $(a_i, r_i)$ in the dual space $\rM \otimes \R$,
and one has $U_i \cong \A^2 / \bmu_{r_i}$, where $\bmu_{r_i}$ acts on~$\A^2$ with weights $(1,a_i)$,
so that $x_i$ is a cyclic quotient singularity of type $1/{r_i}(1,a_i)$.
In particular $X$ has cyclic quotient singularities.
We call~$r_i$ {\sf the order} of the $\rT$-invariant point~$x_i$.

Note that if we change the orientation of $\rN \otimes \R$, that is if we replace the counterclockwise
ordering of vectors~$v_i$ by the clockwise ordering, the construction above will describe the singular point~$x_i$
as of type~$1/{r_i}(1, a_i')$ where~$a_i'$ is the inverse of~$a_i$ modulo~$r_i$
(cf.~Remark~\ref{remark:kk-duality}).

\subsection{The Brauer group of toric surfaces}

As it was explained in Section~\ref{subsection:torsion-surfaces}
the Brauer group of a smooth projective toric surface is zero;
the following lemma describes the Brauer group of a \emph{singular} toric surface in terms of its singularities.
See~\cite{DeMeyerFord, DeMeyerFordMiranda} for more general statements about the Brauer group of a toric variety.

Denote by $\upsilon$ the canonical map
\begin{equation*}
\upsilon \colon \Z^n \to \rN,
\end{equation*}
that takes the $i$-th basis vector of $\Z^n$ to the corresponding vector $v_i \in \rN$.

\begin{lemma}[\cite{DeMeyerFord}, Corollary 2.9(c)]\label{lemma:toric-br}
Let $\rN_\Sigma {} := \Im(\upsilon) \subset \rN$ be the sublattice generated by all ray generators $v_i$
of the fan~$\Sigma$ of a projective toric surface $X$.
Then
\begin{equation*}
\Br(X) \cong \rN/\rN_\Sigma.
\end{equation*}
Moreover, $\Br(X)$ is a cyclic group of order equal to~$\gcd(r_1,\dots,r_{n})$, the greatest common divisor
of the orders $r_i$ of the $\rT$-invariant points of $X$.
In particular, if~$X$ has a smooth $\rT$-invariant point then $X$ is torsion-free.

Furthermore, $\Pic(X) \cong \Ker(\upsilon)$ and $\Cl(X) \cong \Coker(\upsilon^T \colon \rM \to \Z^n)$.
\end{lemma}

\begin{remark}\label{remark:toric-gcd}
One can also show
that for every $1 \le i \le n$
\begin{equation*}
\gcd(r_1, \dots, r_n) = \gcd(r_1, \dots, \widehat{r_i}, \dots, r_n).
\end{equation*}
\end{remark}

\begin{example}
\label{ex:p123}
Let $w_1,w_2,w_3 \ge 1$ be pairwise coprime.
Then the weighted projective plane $\P(w_1,w_2,w_3)$
has three torus-invariant points of
types $1/w_1(w_2,w_3)$, $1/w_2(w_3,w_1)$, and $1/w_3(w_1,w_2)$ of orders~$w_1$, $w_2$, and~$w_3$ respectively,
in particular it is torsion-free.
One interesting special case is~$\P(1,2,3)$ which is given by the fan
\begin{equation*}
v_1 = (1,1),\quad
v_2 = (-2,1),\quad
v_3 = (1,-1).
\end{equation*}
It has one $A_1$ and one $A_2$ singularity, $\Pic(X) = \Cl(X) = \Z$, $\Br(X) = 0$.
\end{example}

\begin{example}
\label{ex:p1p1z2}
Let $X = (\P^1 \times \P^1) \big/ \bmu_2$, where $\bmu_2$ is acting diagonally
and the action on each factor is given by $[x:y] \mapsto [-x:y]$.
Then $X$ is a projective toric surface with the fan given by
\begin{equation*}
v_1 = (1,1),\quad
v_2 = (-1,1),\quad
v_3 = (-1,-1),\quad
v_4 = (1,-1).
\end{equation*}
It has four $A_1$ singularities and $\Pic(X) = \Z^2$, $\Br(X) = \Z/2$, $\Cl(X) = \Z^2 \oplus \Z/2$.
In particular, $X$ is not torsion-free.
\end{example}

\begin{example}
\label{ex:p2z3}
Let $X = \P^2 \big/ \bmu_3$, where $\bmu_3$ is acting on $\P^2$ with three different weights.
Then $X$ is a projective toric surface with the fan given by
\begin{equation*}
v_1 = (1,1),\quad
v_2 = (-2,1),\quad
v_3 = (1,-2).
\end{equation*}
It has three $A_2$ singularities and $\Pic(X) = \Z$, $\Br(X) = \Z/3$, $\Cl(X) = \Z \oplus \Z/3$.
In particular, $X$ is not torsion-free.
\end{example}

The fans of the above toric surfaces are presented below.
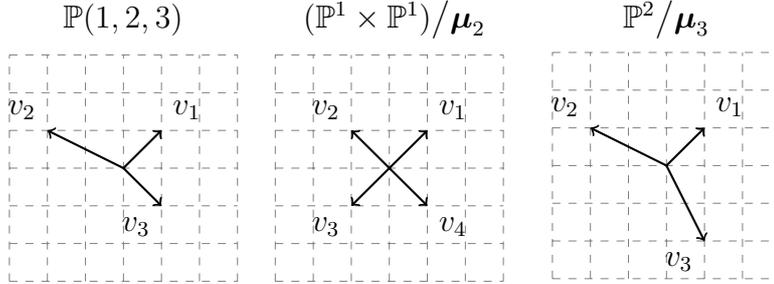
\begin{center}
\begin{tabular}{ccc}
$\P(1,2,3)$ & $(\P^1 \times \P^1) \big/ \bmu_2$ & $\P^2 \big/ \bmu_3$
\\[1ex]
\hfill
\begin{tikzpicture}[xscale = .5pt, yscale = .5pt]
\draw[step=1,gray,very thin,dashed] (-3,-3) grid (3,3);
\draw[thick,->] (0,0)--(1,1) node[above right] {$v_1$};
\draw[thick,->] (0,0)--(-2,1) node[above left] {$v_2$};
\draw[thick,->] (0,0)--(1,-1) node[below left] {$v_3$};
\end{tikzpicture}
\hfill&\hfill
\begin{tikzpicture}[xscale = .5pt, yscale = .5pt]
\draw[step=1,gray,very thin,dashed] (-3,-3) grid (3,3);
\draw[thick,->] (0,0)--(1,1) node[above right] {$v_1$};
\draw[thick,->] (0,0)--(-1,1) node[above left] {$v_2$};
\draw[thick,->] (0,0)--(-1,-1) node[below left] {$v_3$};
\draw[thick,->] (0,0)--(1,-1) node[below right] {$v_4$};
\end{tikzpicture}
\hfill&\hfill
\begin{tikzpicture}[xscale = .5pt, yscale = .5pt]
\draw[step=1,gray,very thin,dashed] (-3,-3) grid (3,3);
\draw[thick,->] (0,0)--(1,1) node[above right] {$v_1$};
\draw[thick,->] (0,0)--(-2,1) node[above left] {$v_2$};
\draw[thick,->] (0,0)--(1,-2) node[below left] {$v_3$};
\end{tikzpicture}
\hfill
\end{tabular}
\end{center}

\subsection{Minimal resolution}

Let $X$ be a toric surface and let
$\pi \colon \tX \to X$ be the minimal resolution of singularities of $X$.
Then $\tX$ is also a toric surface, with an action of the same torus $\rT$.
The fan~$\tilde\Sigma$ of~$\tX$ is a refinement of $\Sigma$.

To be more precise, $\tilde\Sigma$ is obtained from $\Sigma$ as follows:
for each two-dimensional cone~$(v_i,v_{i+1})$ of $\Sigma$
consider the convex hull of all nonzero integral points of the cone
generated by~$v_i$ and~$v_{i+1}$, i.e.,
\begin{equation*}
\rP_{i,i+1} := \operatorname{Conv}\Big((\R_{\ge 0}v_i + \R_{\ge 0}v_{i+1}) \cap (\rN \setminus \{0\})\Big),
\end{equation*}
where the sum in the right side is the Minkowski sum of two rays (see~\cite[8.4]{Danilov}
or~\cite[Exercise (a) on page 46]{Fulton}).
Consider all integral points on the boundary of~$\rP_{i,i+1}$ that do not lie on two infinite rays segments $\R_{\ge 1}v_i$ and $\R_{\ge 1}v_{i+1}$.
Let
\begin{equation*}
v_{i,1},\dots,v_{i,m_i}
\end{equation*}
be these points indexed in the counterclockwise order.
Also set
\begin{equation*}
v_{i,0} = v_i.
\end{equation*}
Then $v_{i,p}$, $1 \le i \le n$, $0 \le p \le m_i$, are all generators of rays in the fan $\tilde\Sigma$.
Note that if~$\det(v_i,v_{i+1}) = 1$, i.e., if $x_i$ is a smooth point of $X$, we have $m_i = 0$.

The fans of the resolutions of toric surfaces in Examples~\ref{ex:p123}, \ref{ex:p1p1z2}, \ref{ex:p2z3} are presented below.
\begin{center}
\begin{tabular}{ccc}
$\P(1,2,3)$ & $(\P^1 \times \P^1) \big/ \bmu_2$ & $X = \P^2 \big/ \bmu_3$
\\[1ex]
\hfill
\begin{tikzpicture}[xscale = .5pt, yscale = .5pt]
\draw[step=1,gray,very thin,dashed] (-3,-3) grid (3,3);
\draw[very thick,->] (0,0)--(1,1) node[above right] {$v_{1,0}$};
\draw[thick,->] (0,0)--(0,1) node [above] {\tiny$v_{1,1}$};
\draw[thick,->] (0,0)--(-1,1) node [above left] {\tiny$v_{1,2}$};
\draw[very thick,->] (0,0)--(-2,1) node[left] {$v_{2,0}$};
\draw[very thick,->] (0,0)--(1,-1) node[below left] {$v_{3,0}$};
\draw[thick,->] (0,0)--(1,0) node [right] {\tiny$v_{3,1}$};
\end{tikzpicture}
\hfill&\hfill
\begin{tikzpicture}[xscale = .5pt, yscale = .5pt]
\draw[step=1,gray,very thin,dashed] (-3,-3) grid (3,3);
\draw[very thick,->] (0,0)--(1,1) node[above right] {$v_{1,0}$};
\draw[thick,->] (0,0)--(0,1) node [above] {\tiny$v_{1,1}$};
\draw[very thick,->] (0,0)--(-1,1) node[above left] {$v_{2,0}$};
\draw[thick,->] (0,0)--(-1,0) node [left] {\tiny$v_{2,1}$};
\draw[very thick,->] (0,0)--(-1,-1) node[below left] {$v_{3,0}$};
\draw[thick,->] (0,0)--(0,-1) node [below] {\tiny$v_{3,1}$};
\draw[very thick,->] (0,0)--(1,-1) node[below right] {$v_{4,0}$};
\draw[thick,->] (0,0)--(1,0) node [right] {\tiny$v_{4,1}$};
\end{tikzpicture}
\hfill&\hfill
\begin{tikzpicture}[xscale = .5pt, yscale = .5pt]
\draw[step=1,gray,very thin,dashed] (-3,-3) grid (3,3);
\draw[very thick,->] (0,0)--(1,1) node[above right] {$v_{1,0}$};
\draw[thick,->] (0,0)--(0,1) node [above] {\tiny$v_{1,1}$};
\draw[thick,->] (0,0)--(-1,1) node [above left] {\tiny$v_{1,2}$};
\draw[very thick,->] (0,0)--(-2,1) node[left] {$v_{2,0}$};
\draw[thick,->] (0,0)--(-1,0) node [left] {\tiny$v_{2,1}$};
\draw[thick,->] (0,0)--(0,-1) node [left] {\tiny$v_{2,2}$};
\draw[very thick,->] (0,0)--(1,-2) node[below left] {$v_{3,0}$};
\draw[thick,->] (0,0)--(1,-1) node [right] {\tiny$v_{3,1}$};
\draw[thick,->] (0,0)--(1,0) node [right] {\tiny$v_{3,2}$};
\end{tikzpicture}
\hfill
\end{tabular}
\end{center}

The exceptional divisor of the resolution $\pi$ is the union
\begin{equation*}
D = \bigsqcup_{i=1}^n \left( \bigcup_{p=1}^{m_i}  E_{i,p} \right)
\end{equation*}
of the irreducible toric divisors $E_{i,p}$ corresponding to the rays $v_{i,p}$ of $\tilde\Sigma$ that are not in~$\Sigma$,
i.e., with $1 \le p \le m_i$.
Its connected components are
\begin{equation*}
D_i = \bigcup_{p=1}^{m_i}  E_{i,p}
\end{equation*}
(so that if the point $x_i$ on $X$ is smooth then $D_i$ is empty).
As usual we denote by
\begin{equation*}
d_{i,p} := - E_{i,p}^2
\end{equation*}
the self-intersections of the exceptional divisors.
Furthermore, we denote by $E_{i,0}$ the irreducible toric divisors on $\tX$ corresponding to the rays $v_{i,0}$
(they are the strict transforms of the toric divisors $C_i$ on $X$).


\subsection{Adherent exceptional collections}

There is a standard way \cite{Hille} to construct a full exceptional collection of line bundles on a smooth
toric surface.
One should choose a ray in the fan and a direction (counterclockwise or clockwise) and starting with
any line bundle
add at each step the divisor corresponding to the next ray in the fan.

In case of the minimal resolution $\tX$ of a toric surface $X$ to make this collection (twisted) adherent
to the connected components $D_i$ of the exceptional divisor~$D$,
one can start for example with $\cO_\tX(E_{1,0})$, add $E_{1,1}$ at the first step, and go in the counterclockwise direction.

This procedure gives the collection~$\cL_{i,p}$, $1 \le i \le n$, \mbox{$0 \le p \le m_i$},
defined by
\begin{align}
\label{eq:def-lij2}
\cL_{i,0} &= \cO_\tX\left(\sum_{j=1}^{i-1} \sum_{q=0}^{m_j} E_{j,q} + E_{i,0}\right),\\
\label{eq:def-lip}
\cL_{i,p} &= \cL_{i,0}(E_{i,1} + \dots + E_{i,p}).
\end{align}
Then we have the following equality
\begin{equation}
\label{eq:Lip-Eip}
\cL_{i,0} \cdot E_{k,p} =
\delta_{i,k}\delta_{p,1} + \delta_{k,n}\delta_{p,m_n},
\qquad \text{if $1 \le i \le k \le n$}.
\end{equation}
Indeed, if
$i \le k$ and $1 \le p \le m_k$ the only summands in the right hand side of~\eqref{eq:def-lij2} that
have nontrivial intersection with the curve $E_{k,p}$
are the last summand $E_{i,0}$ (if~$i = k$ and~$p = 1$) and,
since the curves form a cycle, the first summand $E_{1,0}$ (if~$k = n$ and~$p = m_n$).
This allows us to deduce the following.

%

\begin{proposition}
\label{proposition:toric-sequence}
The collection of line bundles $\cL_{i,p}$, $1 \le i \le n$, $0 \le p \le m_i$, defined by~\eqref{eq:def-lij2} and~\eqref{eq:def-lip}
is a full exceptional collection consisting of $n$ blocks and its $i$-th block
\begin{equation*}
\tcA_i = \langle \cL_{i,0}, \cL_{i,1}, \dots, \cL_{i,m_i} \rangle
\end{equation*}
is adherent to the connected component $D_i$ of the exceptional divisor $D$ with the twist
\begin{equation}
\label{eq:aip-toric}
b_{i,p} =
\begin{cases}
2 - d_{i,p}, & (i,p) \ne (n,m_n) \\
3 - d_{n,m_n}, & (i,p) = (n,m_n). \\
\end{cases}
\end{equation}
In particular, if $m_n = 0$, i.e., if there are no exceptional curves between~$E_{n,0}$ and~$E_{1,0}$,
the formula~\eqref{eq:aip-toric} simplifies to $b_{i,p} = 2 - d_{i,p}$ for all $1 \le i \le n, 1 \le p \le m_i$.


\end{proposition}
\begin{proof}
The collection is full and exceptional by~\cite[Theorem~5.1]{Hille},
and adherence follows from~\eqref{eq:def-lip} by definition.
Comparing~\eqref{eq:cli0-oeip} with~\eqref{eq:Lip-Eip} for $k = i$ we get
\begin{equation*}
d_{i,p} + b_{i,p} - 2 + \delta_{p,1} = \delta_{p,1} + \delta_{i,n}\delta_{p,m_n}.
\end{equation*}
which gives the formula~\eqref{eq:aip-toric} for the twist.
\end{proof}

\begin{remark}
One can use the same method to construct a full exceptional collection of line bundles with $n$ blocks
adherent to components $D_i$ of the exceptional locus $D$ of the minimal resolution
for any normal projective rational surface~$X$ with cyclic quotient singularities,
if the chains $D_i$ can be included into a \emph{cycle of smooth rational curves} summing up to $-K_\tX$.
\end{remark}

We consider the element
\begin{equation*}
f := {} \sum_{i,p} b_{i,p} \delta_{E_{i,p}} \in \Cl(\tX/X)^\vee
\end{equation*}
corresponding to the twist~\eqref{eq:aip-toric}.
Note that by adjunction formula $K_\tX \cdot E_{i,p} = d_{i,p} - 2$, hence $f + \IP(K_\tX) = \delta_{E_{n,m_n}}$.
This means that the Brauer class corresponding to the function $f$ under the Brauer class map \eqref{eq:def-rb} is equal to
\begin{equation}
\label{eq:B-delta}
\rB\left( \sum b_{i,p}\delta_{E_{i,p}} \right) = \rB(\delta_{E_{n,m_n}}) \in \Br(X).
\end{equation}
In particular, if $m_n = 0$ (i.e., if $D_n = \varnothing$) then $\rB(f) = 0$.

\begin{remark}
\label{remark:untwisting}
More generally, if $X$ is torsion-free, $\rB(f) = 0$ (because $\Br(X) = 0$),
and we are able to explicitly untwist the exceptional collection of Proposition~\ref{proposition:toric-sequence} as follows.
Let $\cM$ be a line bundle on $\tX$ such that
for all $1 \le i \le n$, $1 \le p \le m_i$
\begin{equation}
\label{eq:M-def}
\cM \cdot E_{i,p} =
\begin{cases}
0, & (i,p) \neq (n,m_n) \\
-1, & (i,p) = (n,m_n). \\
\end{cases}
\end{equation}
Such a line bundle exists by Corollary \ref{corollary:restriction-exists} (2).
Note that as usual, if $x_n$ is smooth then the second case in
\eqref{eq:M-def} does not occur and we can take $\cM = \OO_{\tX}$.
Now using~\eqref{eq:aip-toric} it follows that if $\cL_{i,p}$ are defined by~\eqref{eq:def-lij2} and~\eqref{eq:def-lip},
the full exceptional collection
\begin{equation}
\label{eq:M-seq}
\{ \cM_{i,p} = \cL_{i,p} \otimes \cM(K_{\wt{X}})\}_{i,p}
\end{equation}
consists of $n$ blocks \emph{untwisted} adherent to $D_1, \dots, D_n$ respectively.
Note also that it follows easily from~\eqref{eq:Lip-Eip} and~\eqref{eq:M-def} that
\begin{equation}
\label{eq:mi0-ekp}
\cM_{i,0} \cdot E_{k,p} = \delta_{i,k} \delta_{p,1} + d_{k,p} - 2
\end{equation}
for all $i \le k$ and $1 \le p \le m_k$.
\end{remark}

The next theorem summarizes our results on semiorthogonal decompositions for singular surfaces in the toric case.
We use the notation and conventions introduced above.

\begin{theorem}
\label{theorem:toric}
Let $X$ be a projective toric surface with $\rT$-invariant points $x_1,\dots,x_n$,
and let $\pi \colon \tX \to X$ be its minimal resolution with the exceptional divisor $D = \sqcup_{i=1}^n D_i$, where $\pi(D_i) = x_i$.
Let $K_i$ be the Kalck--Karmazyn algebra corresponding to the chain of rational curves $D_i$.
Let
\begin{equation*}
\beta = \rB ( \delta_{E_{n,m_n}} ) \in \Br(X)
\end{equation*}
be the corresponding Brauer class.
Then there is a semiorthogonal decomposition
\begin{equation*}
\Db(X,\beta) = \langle \Db(K_1\md), \Db(K_{2}\md), \dots, \Db(K_n\md) \rangle.
\end{equation*}
Moreover, if the $\rT$-invariant points $x_j$ for $j \ge 2$ are Gorenstein, this decomposition induces a semiorthogonal decomposition
\begin{equation*}
\Dperf(X,\beta) = \langle \Dperf(K_1\md), \Dperf(K_{2}\md), \dots, \Dperf(K_n\md) \rangle.
\end{equation*}
\end{theorem}

\begin{proof}
This is a special case of Theorem~\ref{theorem:dbx-twisted} in the toric situation.
\end{proof}

When $X$ is torsion-free so that $\Br(X) = 0$ we get the following:

\begin{corollary}\label{cor:toric}
If $X$ is a projective torsion-free toric surface,
there is a semiorthogonal decomposition
\begin{equation*}
\Db(X) = \langle \Db(K_1\md), \Db(K_{2}\md), \dots, \Db(K_n\md) \rangle.
\end{equation*}
If the $\rT$-invariant points $x_j$ for $j \ge 2$ are Gorenstein, this decomposition
induces a semiorthogonal decomposition
\begin{equation*}
\Dperf(X) = \langle \Dperf(K_1\md), \Dperf(K_{2}\md), \dots, \Dperf(K_n\md) \rangle.
\end{equation*}
\end{corollary}

\begin{remark}
If $X$ is not torsion-free, the Brauer class $\beta$ is always non-trivial by Corollary~\ref{corollary:beta-generator}.
\end{remark}

Recall that weighted projective planes are torsion-free (Example~\ref{ex:p123}).

\begin{example}[cf.\ Example~\ref{ex:hirzebruch}]
\label{Example:P11n}
If $X$ is the weighted projective plane $\P(1,1,d)$, we obtain a semiorthogonal decomposition
\begin{equation*}
\Db(X) = \langle \Db(\Bbbk[z_1, \dots, z_{d-1}]/(z_1,\dots, z_{d-1})^2), \Db(\Bbbk), \Db(\Bbbk) \rangle
\end{equation*}
and a similar decomposition for the category~$\Dperf(X)$.
Recall that the isomorphism $K_{d,1} \simeq \Bbbk[z_1, \dots, z_{d-1}]/(z_1,\dots, z_{d-1})^2$
has been explained in Example~\ref{example:k}(2).
Note that we placed the non-Gorenstein singular point of $X$ to the first position
to achieve the decomposition of $\Dperf(X)$.
\end{example}

\begin{example}
\label{Example:P123}
If $X$ is the weighted projective plane $\P(1,2,3)$, see Example~\ref{ex:p123},
we obtain a semiorthogonal decomposition
\begin{equation*}
\Db(X) = \langle \Db(\Bbbk), \Db(\Bbbk[z]/z^2), \Db(\Bbbk[z]/z^3) \rangle
\end{equation*}
and again a similar decomposition for~$\Dperf(X)$.
\end{example}

\begin{example}
If $X$ is the surface of Example~\ref{ex:p1p1z2}, we obtain
\begin{equation*}
\Db(X,\beta) = \langle \Db(\Bbbk[z]/z^2), \Db(\Bbbk[z]/z^2), \Db(\Bbbk[z]/z^2), \Db(\Bbbk[z]/z^2) \rangle,
\end{equation*}
where $\beta$ is the non-trivial element of $\Br(X) = \Z/2$,
and a similar decomposition for~$\Dperf(X)$.
\end{example}

\begin{example}
If $X$ is the surface of Example~\ref{ex:p2z3}, we obtain
\begin{equation*}
\Db(X,\beta) = \langle \Db(\Bbbk[z]/z^3), \Db(\Bbbk[z]/z^3), \Db(\Bbbk[z]/z^3) \rangle,
\end{equation*}
where $\beta$ is a non-trivial element of $\Br(X) = \Z/3$,
and again a similar decomposition for~$\Dperf(X)$.
\end{example}

\subsection{Special Brauer classes}

The exceptional collection of Proposition~\ref{proposition:toric-sequence},
as well as the Brauer class~$\beta$ on~$X$ and the semiorthogonal decomposition of $\Db(X,\beta)$ of Theorem~\ref{theorem:toric}
depend on some choices.
First, they depend on the choice of a cyclic order of rays of the fan of $X$
(this is equivalent to a choice of orientation of the plane~$\rN \otimes \R$).
Second, they depend on the choice of the linear order~$x_1,\dots,x_n$ of torus-invariant points on $X$
compatible with the chosen cyclic order.
Changing these choices we obtain a semiorthogonal decomposition of a differently twisted category
with the same components up to reordering.
The next lemma explains how twist changes.

We denote by $\beta_i$ the Brauer class corresponding to the choice of the point $x_i$ as \emph{the last} point,
so that the linear ordering of the torus-invariant points is
\begin{equation*}
x_{i+1}, \dots, x_{n-1}, x_n, x_1, \dots, x_{i-1}, x_{i-2}, x_i
\end{equation*}
and by $\beta_i'$ the Brauer class corresponding to the same choice of the last point
with the opposite cyclic order,
so that the linear ordering is
\begin{equation*}
x_{i-1}, x_{i-2}, \dots, x_1, x_n, x_{n-1} \dots, x_{i+1}, x_i.
\end{equation*}

\begin{lemma}
\label{lemma:delta-functions}
If for each $1 \le i \le n$ the point $x_i$ is a cyclic quotient singularity
of type~$\frac1{r_i}(1,a_i)$
then the following relations are satisfied
by the classes $\beta_1, \dots, \beta_n$, and~$\beta_1', \dots, \beta_n' \in \Br(X)$: for every
$1 \le i \le n$ we have
\begin{equation*}
\beta_i = a_i \beta_i^{'},
\qquad
\beta_{i+1} = -a_{i+1} \beta_i.
\end{equation*}
In particular, if $X$ is Gorenstein then
$\beta_i = \beta_1$, and
$\beta_i' = -\beta_1$ for all $1 \le i \le n$.
\end{lemma}
\begin{proof}
By~\eqref{eq:B-delta} we have
\begin{equation*}
\beta_i = \rB(\delta_{E_{i,m_i}}),
\qquad
\beta_i' = \rB(\delta_{E_{i,1}}).
\end{equation*}
Using standard recursions for determinants of tridiagonal matrices is easy to check that
\begin{equation*}
\IP \left( \sum_{p=2}^{m_i} \tridet(d_{i,2},\dots,d_{i,p-1}) E_{i,p} \right) =
\delta_{E_{i,1}} - \tridet(d_{i,2},\dots,d_{i,m_i}) \delta_{E_{i,m_i}}.
\end{equation*}
Using~\eqref{eq:r-a-d} and $\rB \circ \IP = 0$, we conclude that $\beta_i = a_i \beta_i'$.
Similarly
\begin{equation*}
\IP(E_{i+1,0}) = \delta_{E_{i,m_i}} + \delta_{E_{i+1,1}},
\end{equation*}
hence $\beta_{i+1}' = -\beta_i$.
Finally, if $X$ is Gorenstein then $a_i = r_i - 1$, and since $\Br(X)$
is a cyclic group of order $\gcd(r_1, \dots, r_n)$ by Lemma \ref{lemma:toric-br} and Remark \ref{remark:toric-gcd},
so that $r_i$ acts trivially on $\Br(X)$, we see that $\beta_i = \beta_1$, $\beta_i' = -\beta_1$ for all $1 \le i \le n$.
\end{proof}

In the Gorenstein case we thus obtain a bunch of semiorthogonal decompositions of~$\Db(X,\beta_1)$ and $\Db(X,-\beta_1)$.
Using Remark~\ref{remark:components-twist}
one can show that all these decompositions of~$\Db(X,\beta_1)$ are the same up to line bundle twist,
and all decompositions of~$\Db(X,-\beta_1)$ are obtained from the above decompositions by dualization.

\section{Reflexive sheaves}
\label{section:reflexive}

An alternative approach to construct a semiorthogonal decomposition of $\Db(X)$
was suggested by Kawamata in \cite{Kawamata}.
Starting with an object $F \in \Coh(X)$, take $G_0 = F$ and consider a sequence $G_i$ of
iterated non-trivial extensions
\begin{equation*}
0 \rightarrow F \rightarrow G_i \rightarrow G_{i-1} \rightarrow 0,
\end{equation*}
in $\Coh(X)$.
If it terminates with some  maximal element $\wt{G} = G_m$, i.e.,
\begin{equation*}
\Ext^1_X(\wt{G},F) = 0,
\end{equation*}
then $\wt{G}$ is said to be a \emph{maximal iterated extension} of $F$.
When such an object exists it is unique, see~\cite[Corollary 3.4]{Kawamata},
and can be interpreted as the versal noncommutative deformation of $F$ over $\End_X(\wt{G})$; see~\cite[Corollary 4.11]{Kawamata}.
Furthermore, Kawamata proves that, under appropriate conditions the sheaf $\wt{G}$ generates
an admissible subcategory of $\Db(X)$, and in some examples
(for the weighted projective planes~$\P(1,1,n)$ and~$\P(1,2,3)$, see~\cite[Examples 5.5, 5.7, 5.8]{Kawamata})
he constructs a semiorthogonal decomposition into derived categories of finite dimensional algebras.
In this section we explain the relation between our approach and that of Kawamata.

\subsection{Criteria of reflexivity and purity}
\label{subsection:criteria}

Let $X$ be a normal surface.
Recall that by definition a reflexive sheaf is a coherent sheaf $\FF$ on $X$ satisfying
\begin{equation*}
(\FF^\vee)^\vee \simeq \FF
\end{equation*}
via the natural morphism (here the duality is underived);
an equivalent definition is that
\begin{equation*}
\cF \simeq R^0u_* \cF_U
\end{equation*}
where $u\colon U \to X$ is the embedding of an open subset with zero-dimensional complement,
and $\cF_U$ is a locally free sheaf on $U$.
Furthermore, recall that the set of isomorphism classes
of reflexive sheaves on $X$ is a group with respect to the operation
\begin{equation*}
(\cF_1 , \cF_2)
\mapsto \left((\cF_1 \otimes \cF_2)^\vee\right)^\vee
\end{equation*}
(the tensor product and duality are underived), and this group is isomorphic
to the class group $\Cl(X)$ via the first Chern class map~\eqref{eq:c1-G0}.
If $D$ is a Weil divisor on $X$ we write $\OO(D)$ for
the reflexive sheaf with the first Chern class $D$; equivalently
$\OO(D)$ may be defined as $R^0 u_* \OO(D \cap U)$, where $U$
is the nonsingular locus of $X$ as above.

We will be interested in coherent sheaves on $X$ obtained by pushing forward line bundles from a resolution~$\tX$,
and we will rely on the following criterion for reflexivity.

\begin{lemma}
\label{lemma:reflexivity-criterion}
Let $(X,x)$ be a cyclic quotient surface singularity
and let $\pi \colon \tX \to X$ be its minimal resolution
with the chain $E_1,\dots,E_m$ of exceptional divisors with self-intersections $E_i^2 = -d_i$.
Set $D = \sum E_j$.

If a line bundle $\LL$ on $\wt{X}$ satisfies
\begin{equation}
\label{eq:reflexivity-conditions}
\cL(D) \cdot E_j \le 0
\quad\text{for all $1 \le j \le m$,\quad and } \qquad
\sum_{j=1}^m \cL(D) \cdot E_j < 0,
\end{equation}
then $R^0\pi_*(\cL)$ is a reflexive sheaf of rank~$1$ on $X$.
\end{lemma}
\begin{proof}
Let
\begin{equation*}
U = X \setminus \{x\} = \tX \setminus D.
\end{equation*}
Let $u \colon U \to X$ and $\tilde{u} \colon U \to \tX$ be the embeddings and set $\cL_U := \tilde{u}^*\cL$.
Since $\tilde{u}$ is an affine morphism, we have
\begin{equation*}
\tilde{u}_*(\cL_U) {} = R^0\tilde{u}_*(\cL_U) \cong \ilim \cL(tD),
\end{equation*}
where in the right-hand side the colimit is taken with respect to the sequence of the natural embeddings $\cL \to \cL(D) \to \cL(2D) \to \dots$.
Therefore, there is an exact sequence
\begin{equation*}
0 \to \cL \to \tilde{u}_*(\cL_U) \to \ilim \left( \frac{\cL(tD)}\cL \right) \to 0.
\end{equation*}
Pushing it forward to $X$ we obtain an exact sequence
\begin{equation}
\label{eq:pushforward-sequence}
0 \to R^0\pi_*(\cL) \to R^0\pi_*(\tilde{u}_*(\cL_U)) \to \ilim R^0\pi_* \left( \frac{\cL(tD)}\cL \right)
\end{equation}
(the pushforward functor commutes with the colimit since $\pi$ is quasicompact).

Let us show that the last term in~\eqref{eq:pushforward-sequence} is zero.
We prove that $R^0\pi_* \left( \frac{\cL(tD)}\cL \right)$ = 0 for all $t \ge 0$ by induction.
The short exact sequence
\begin{equation*}
0 \to \frac{\cL((t-1)D)}\cL \to \frac{\cL(tD)}\cL \to \cL(tD)\vert_D \to 0
\end{equation*}
shows that it is enough to check that $H^0(D,\cL(tD)\vert_D) = 0$ for all $t \ge 1$.
Since~$D$ is a chain of curves $E_j$, it is enough to check that all the degrees $\cL(tD) \cdot E_j$ are nonpositive
and their sum is negative.
For $t = 1$ this is ensured by~\eqref{eq:reflexivity-conditions}.
For $t > 1$ the same follows from~\eqref{eq:reflexivity-conditions} combined with inequalities
\begin{equation*}
D \cdot E_j =
2 - \delta_{j1} - \delta_{jm} - d_j \le 0,
\end{equation*}
that hold true since $d_j \ge 2$ for all $j$ by minimality of the resolution.

Now we conclude from~\eqref{eq:pushforward-sequence} that
\begin{equation*}
R^0\pi_*(\cL) \cong R^0\pi_*(\tilde{u}_*(\cL_U)) \cong R^0u_*(\cL_U).
\end{equation*}
Since $X$ is a normal surface, it follows that $R^0\pi_*(\cL)$ is reflexive.
\end{proof}

\begin{corollary}\label{cor:reflexivity}
In the notation of Lemma~\textup{\ref{lemma:reflexivity-criterion}}, if the multidegree vector $(\cL \cdot E_j)_{j=1}^m$
is equal to one of the vectors
\begin{align*}
(d_1 - 2, \; d_2 - 2, \; \dots, \; d_{m-1} - 2, \; d_m - 2), \\
(d_1 - 2, \; d_2 - 2, \; \dots, \; d_{m-1} - 2, \; d_m - 1), \\
(d_1 - 1, \; d_2 - 2, \; \dots, \; d_{m-1} - 2, \; d_m - 2),
\end{align*}
then the derived pushforward $\pi_* \LL$ is a reflexive sheaf of rank~$1$ on $X$.
\end{corollary}
\begin{proof}
Let us write $l_j = \cL \cdot E_j$.
It is easy to see that the condition (\ref{eq:reflexivity-conditions}) is equivalent
to the following condition: if $m > 1$ the degrees $l_j$ satisfy
\begin{equation*}
l_1 \le d_1 - 1, \quad l_j \le d_j - 2 \quad (2 \le j \le m-1), \quad l_m \le d_m - 1
\end{equation*}
and requiring that one of these inequalities is strict; and if $m = 1$ then $l_1 < d_1$.
This condition holds by assumption and Lemma \ref{lemma:reflexivity-criterion} implies that $R^0\pi_* \cL$ is a reflexive sheaf.

To show that higher direct images vanish we note that for every $1 \le j \le m$ we have~$K_{\wt{X}} \cdot E_j = d_j - 2$
so that by assumption on the degrees $l_j$ we deduce that~$\cL(-K_{\wt{X}})$ is~$\pi$-nef.
By the Kawamata--Viehweg vanishing, see~\cite[Theorem~1-2-3]{KMM},
we have the required vanishing~$R^p \pi_* \cL = 0$ for all $p > 0$.
\end{proof}

\begin{example}
For the $\frac12(1,1)$ singularity $X = \Spec(\Bbbk[x^2, xy, y^2])$
let $E$ be the single exceptional curve on the minimal resolution $\wt{X}$, so that $d = 2$.
Let $\cL$ be a line bundle on $\wt{X}$, and let $l$ be the degree of $\LL$ restricted to $E$.
By Corollary \ref{cor:reflexivity}, for~$l = 0$ and~$l = 1$, $\pi_* \cL$ is a reflexive sheaf
of rank~1. In fact one can see that the same is true for~$l = -1$.

However, for $l {}\ge{} 2$, $\pi_* \cL$ is a pure sheaf which is not reflexive,
while for $l \le -2$, $R^1 \pi_* \cL$ is a non-zero torsion sheaf so that $\pi_* \cL$ is not a sheaf.
\end{example}

\begin{remark}
In the case of a Gorenstein cyclic quotient singularity of type~\mbox{$\frac1{m}(1,m-1)$} with $m > 1$
the exceptional divisor of the minimal resolution is a chain $E_1,\dots,E_{m-1}$ and
we have $d_j = 2$ for all $1 \le j \le m-1$.
The condition of the~Corollary~\ref{cor:reflexivity} is that~$\cL$ restricts trivially to all exceptional curves
except possibly either $E_1$ or $E_m$ where it may have degree~1.
\end{remark}

\begin{remark}
Other criteria for reflexivity can be found in the literature,
e.g. \cite[\S A1, Theorem, part~(b)]{WunramReflexive}, \cite[Lemma~2.1]{E85}.
\end{remark}

\subsection{Extension of reflexive rank~$1$ sheaves}
\label{subsection:reflexive-sheaves}

Now we relate the components of the semiorthogonal decomposition~\eqref{eq:dbx-dbk}
to the subcategories defined by Kawamata in~\cite{Kawamata}.
So, we assume that $X$ is a normal projective rational surface
with cyclic quotient singularities and~$\pi \colon \tX \to X$ is its minimal resolution.
Let $D_i$ be a connected component of the exceptional divisor $D$ of $\pi$ and
let $x_i = \pi(D_i) \in X$ be a quotient cyclic singularity of type $\frac1{r_i}(1,a_i)$.
Assume that the subcategory $\tcA_i \subset \Db(\tX)$ is \emph{untwisted} adherent to~$D_i$,
let $\cL_{i,0}$ be the corresponding line bundle on $\tX$ so that~\eqref{eq:tcai-l-o} holds
and let~$\gamma_i \colon \Db(K_i\md) \to \Db(X)$ be the functor of Theorem~\ref{theorem:sod-cai}.
We define a complex of sheaves
\begin{equation}
\label{eq:def-ri}
R_i := \gamma_i(\Bbbk) {} \in \Db(X),
\end{equation}
the image of the simple $K_i$-module $\Bbbk$ under the functor $\gamma_i$,
and
\begin{equation}
\label{eq:def-mi}
M_i := \gamma_i(K_i) {} \in \Db(X),
\end{equation}
the image of the free $K_i$-module $K_i$ under the same functor.
Since $K_i$ is an iterated extension of $\Bbbk$, $M_i$ is an iterated extension of $R_i$.
Here, given $F, G \in \Db(X)$ by an extension of $G$ by $F$ we mean
a cone of any morphism $G[-1] \to F$, so that if $F$ and $G$ are coherent sheaves
extensions are precisely those in the sense of abelian categories.
In particular, if $R_i \in \Db(X)$ is a pure sheaf, then so is $M_i$.
Denote by $\langle R_i \rangle$ and $\langle M_i \rangle$ the minimal triangulated
subcategory of $\Db(X)$ containing~$R_i$ and~$M_i$ respectively.

\begin{proposition}
\label{prop:semiort-ri-mi}
There is a semiorthogonal decomposition
\begin{equation*}
\Db(X) = \big\langle \langle R_1 \rangle,  \dots, \langle R_n \rangle \big\rangle
\end{equation*}
with $\langle R_i \rangle \simeq \Db(K_i\md)$,
and if $x_2, \dots, x_n \in X$ are Gorenstein then there is a semiorthogonal decomposition
\begin{equation*}
\Dperf(X) = \big\langle \langle M_1 \rangle,  \dots, \langle M_n \rangle \big\rangle
\end{equation*}
with $\langle M_i \rangle \simeq \Dperf(K_i\md)$.
\end{proposition}
\begin{proof}
Consider the semiorthogonal decomposition of Corollary~\ref{corollary:dbx}.
We have
\begin{equation*}
\AA_i \simeq \gamma_i(\Db(K_i\md)) = \gamma_i(\langle \Bbbk \rangle) \simeq \langle R_i \rangle,
\end{equation*}
so that the first decomposition of the proposition
follows from~\eqref{eq:dbx-dbk}.
For the second decomposition we use~\eqref{eq:perfx-perfk} together with the fact
that $\Dperf(K_i\md)$ is generated by~$K_i$ so that the essential
image of $\Dperf(K_i\md)$ in~$\Db(X)$ is~$\langle M_i \rangle$.
\end{proof}

As $\gamma_i$ is fully faithful it is immediate to see that the objects $R_i$ and $M_i$ satisfy
\begin{equation}\label{eq:ext-mi}
\Ext_X^\bullet(M_i,R_i) = \Bbbk,
\qquad
\Ext_X^\bullet(M_i,M_i) \cong K_i,
\end{equation}
where the vector spaces in the right-hand side are concentrated in degree $0$.
In this sense decompositions of Proposition~\ref{prop:semiort-ri-mi} provide an analog
of a full exceptional collection for the singular surface $X$.
In particular if $D_i = \varnothing$, then $R_i = M_i$ is actually an  exceptional object.

As we already noticed, $M_i$ is an iterated extension of $R_i$.
As $\Ext_X^1(M_i,R_i) = 0$ this extension is maximal;
in the language of \cite[Definition 5.1]{Kawamata} $M_i$ is said to be \emph{relative exceptional}.
We are going to investigate under which conditions $R_i$ (resp.~$M_i$) are reflexive (resp.\ locally free)
sheaves on $X$.

Recall the locally free sheaves $\cP_{i,0} \in \tcA_i$ defined in Theorem~\ref{theorem:hille-ploog}.

\begin{proposition}\label{prop:ri-mi}
{$(i)$} For any $1 \le i \le n$ we have natural isomorphisms of complexes
\begin{equation*}
R_i \cong \pi_*(\cL_{i,0}),
\qquad
M_i \cong \pi_*(\cP_{i,0}).
\end{equation*}

\noindent{$(ii)$}
Each $R_i$ is a locally free sheaf of rank~$1$ at $X \setminus \{x_i, \dots, x_n\}$ and a reflexive sheaf of rank~$1$ at $x_i$.
Each $M_i$ is a locally free sheaf of rank~$r_i$ at $X \setminus \{x_{i+1}, \dots, x_n\}$.

\noindent{$(iii)$}
If singular points $x_{i+1}, \dots, x_n$ are Gorenstein, then $R_i$ is locally free of rank~$1$ on~$X \setminus \{x_i\}$
and reflexive of rank~$1$ at $x_i$, and $M_i$ is a locally free sheaf on $X$.

\noindent{$(iv)$}
If $R_i$ is a reflexive sheaf on $X$, then $M_i$ is the unique maximal iterated extension of $R_i$
in the sense of Kawamata~\cite{Kawamata}.
\end{proposition}
\begin{proof}
{$(i)$}
We rely on~\eqref{eq:tgamma-p-s}, \eqref{eq:rho-s0} and the first diagram in~\eqref{eq:gamma-tgamma} to compute
\begin{equation*}
\begin{aligned}
R_i &= \gamma_i(\Bbbk) &&\cong \gamma_i(\rho_{i*}(S_{i,0})) &&\cong \pi_*(\tilde\gamma_i(S_{i,0})) &&\cong \pi_*(\cL_{i,0}), \\
M_i &= \gamma_i(K_i)   &&\cong \gamma_i(\rho_{i*}(P_{i,0})) &&\cong \pi_*(\tilde\gamma_i(P_{i,0})) &&\cong \pi_*(\cP_{i,0}). \\
\end{aligned}
\end{equation*}

{$(ii)$ and $(iii)$}
We start by noticing that since $\pi$ induces an isomorphism on the nonsingular locus of $X$, both $R_i$ and $M_i$ are locally free sheaves
on $X \setminus \{x_1, \dots, x_n\}$.

Let us now check that $R_i$ is reflexive of rank~$1$ in the neighbourhood of $x_i$.
Set
\begin{equation*}
\tV_i := \tX \setminus \bigcup_{j \ne i} D_j
\qquad\text{and}\qquad
V_i := \pi(\tV_i),
\end{equation*}
so that $V_i$ is an open neighborhood of $x_i$ and
the restriction of $\pi$ to $\tV_i$ (which we still denote by $\pi$) is a resolution of the singularity~$(V_i,x_i)$.

It follows from~\eqref{eq:cli0-oeip} where all $b_{i,p} = 0$ that
the condition of Corollary~\ref{cor:reflexivity} for the morphism~\mbox{$\pi \colon \tV_i \to V_i$} is satisfied so
that~$R_i$ is a reflexive sheaf of rank~$1$ on the neighborhood $V_i$ of $x_i$.

To show that $R_i$ (resp. $M_i$) is locally free at $x_j$ for some $j$ by Lemma~\ref{lemma:descent}
we need to check that $\Ext^\bullet(\cL_{i,0}, \OO_{E_{j,p}}(-1)) = 0$ (resp. $\Ext^\bullet(\cP_{i,0}, \OO_{E_{j,p}}(-1)) = 0$)
for all~$1 \le p \le m_j$.
For $j < i$, or for $j > i$ under the Gorenstein condition, the $\Ext$-groups are vanishing by~\eqref{eq:semiorthogonality-tcai-tcaj-ker}.

Finally, for $j = i$, we have for every $1 \le p \le m_i$
\begin{equation*}
\Ext^\bullet(\cP_{i,0},\cO_{E_{i,p}}(-1)) =
\Ext^\bullet(\tilde\gamma_i(P_{i,0}),\tilde\gamma_i(S_{i,p})) =
\Ext^\bullet_{\Lambda_i}(P_{i,0},S_{i,p}) = 0
\end{equation*}
by~\eqref{eq:ext-p-s} and the correspondence between projective and simple $K_i$-modules,
so that~$M_i$ is locally free at~$x_i$.

Each $M_i$ has rank~$r_i$ since by Lemma~\ref{lemma:kk} we have $\dim_{\Bbbk}(K_i) = r_i$
so that $M_i$ has a filtration with subquotients being $r_i$ copies of $R_i$, and $R_i$ has rank~1.

{$(iv)$}
We know that $M_i$ is an iterated extension of $R_i$ and~\eqref{eq:ext-mi} shows that it is maximal.
The uniqueness follows, see~\cite[Corollary 3.4]{Kawamata}.
\end{proof}

We note that if a point $x_i$ is nonsingular, then $K_i = \Bbbk$ and $M_i = R_i$.

\subsection{Toric case}

Assume now that $X$ is a torsion-free toric surface
and let $\pi\colon \wt{X} \to X$ be the minimal resolution.
Let $\MM$ be a line bundle on $\wt{X}$ satisfying conditions~\eqref{eq:M-def} and set
\begin{equation*}
C := \pi_*(\rc_1(\cM)) \in \Cl(X).
\end{equation*}
We note since $\cM$ is well-defined up to twist by a line bundle from $\pi^*(\Pic(X))$,
the class $C$ is well-defined up to adding a Cartier divisor class on $X$.

The full exceptional collection~\eqref{eq:M-seq}
provides a semiorthogonal decomposition of the category~$\Db(\wt{X})$
untwisted adherent to the components~$D_i$ of the exceptional locus~$D$ of~$\pi$.
Recall the objects~$R_i$ and~$M_i$ of~$\Db(X)$ associated in~\eqref{eq:def-ri} and~\eqref{eq:def-mi}
to this semiorthogonal decomposition.


\begin{proposition}
\label{proposition: toric maximal extension}
Let $X$ be a torsion-free projective toric surface with torus-invariant points $x_1,\dots,x_n$,
and with torus-invariant divisors $C_1, \dots, C_n$ with the ordering fixed by the condition~\eqref{eq:x-c-c}.

\noindent{$(i)$}
For any $1 \le i \le n$ the object $R_i$ is a reflexive sheaf of rank~$1$ on $X$, explicitly
\begin{equation}
\label{eq:Ri-expl}
R_i \simeq \OO(K_X + C + C_1 + \dots + C_i),
\end{equation}
Moreover, for every $j < i$ the sheaf $R_i$ is locally free at the torus-invariant point~$x_j$.
If~\mbox{$x_2, \dots, x_n$} are Gorenstein, then each $R_i$ is locally free at $X \setminus \{x_i\}$.

\noindent{$(ii)$}
For any $1 \le i \le n$ the object $M_i$ is a reflexive sheaf on $X$ and is locally free at all~$x_j$ with $j \le i$.
If $x_2, \dots, x_n$ are Gorenstein, then each $M_i$ is locally free on~$X$.

\noindent{$(iii)$}
Each $M_i$ is the unique maximal iterated extension of $R_i$.
\end{proposition}
\begin{proof}
Using Proposition~\ref{prop:ri-mi} we only have to check that~$R_i$ is
a reflexive sheaf of rank~1 at the points~$x_{i+1}, \dots, x_n$ and to prove~\eqref{eq:Ri-expl}.
Recall that by Proposition~\ref{prop:ri-mi}$(i)$ we have~$R_i = \pi_*\MM_{i,0}$.
To prove reflexivity we rely on the criterion of Corollary~\ref{cor:reflexivity}.
By~\eqref{eq:mi0-ekp} we have
\begin{equation*}
\cM_{i,0} \cdot E_{k,p} = d_{k,p} - 2
\end{equation*}
for $i < k \le n$,
and Corollary \ref{cor:reflexivity} applies to show that $R_i = \pi_*(\cM_{i,0})$ is reflexive at $x_k$.

To get the expression~\eqref{eq:Ri-expl} for $R_i$, note that a reflexive sheaf is determined by its first Chern class
and that the functor $\pi_*$ commutes with $\rc_1$ by~\eqref{eq:c1-G0}, hence
\begin{equation*}
\rc_1(R_i) =
\rc_1(\pi_*(\cM_{i,0})) =
\pi_*(\rc_1(\cM_{i,0})) =
\pi_*(\rc_1(\cL_{i,0} \otimes \cM(K_\tX))).
\end{equation*}
Taking into account~\eqref{eq:M-seq} and~\eqref{eq:def-lij2}, and using equalities
$\pi_*(K_\tX) = K_X$, $\pi_*(E_{i,p})  = 0$ for $p {}\ge{} 1$ and $\pi_*(E_{i,0}) = C_{i}$, and the definition of~$C$,
we deduce~\eqref{eq:semiorth-toric-explicit}.
%
%
%
\end{proof}

We now explain how one can compute $C$, and
hence the reflexive sheaves
\eqref{eq:Ri-expl}
in some important special cases.

\begin{proposition}
\label{prop:toric-gcd}
Let $X$ be a projective toric surface with $\rT$-invariant points
$x_1, \dots, x_n$ of orders $r_1,\dots,r_n$ and $\rT$-invariant divisors $C_1, \dots, C_n$
with convention~\eqref{eq:x-c-c}.
Assume that~$\gcd(r_n,r_1) = 1$.
Let~$s$ be an integer such that
\begin{equation*}
s \equiv 0 \mod r_1
\qquad\text{and}\qquad
s \equiv -1 \mod r_n.
\end{equation*}
Then $C = sC_1$ is a Weil divisor class~on~$X$ corresponding to a line bundle~$\cM$ on~$\tX$ satisfying~\eqref{eq:M-def}.
In particular, reflexive sheaves defined by~\eqref{eq:Ri-expl} generate~\eqref{eq:dbx-dbk}.
\end{proposition}
\begin{proof}
Let us check that $\cM$ can be chosen in the form
\begin{equation*}
\cM \cong \cO_\tX(F_n + sE_{1,0} + F_1),
\end{equation*}
where $F_n$ is a linear combination of $E_{n,p}$, $1 \le p \le m_n$, and~$F_1$ is a linear combination of~$E_{1,p}$, $1 \le p \le m_1$.
Indeed, any divisor of the form $F_n + sE_{1,0} + F_1$ has trivial intersection with $E_{i,p}$ for all $2 \le i \le n-1$
and~$1 \le p \le m_i$.
Furthermore, for any $F_n$ we have
\begin{equation*}
F_n \cdot E_{1,p} = 0
\qquad\text{and}\qquad
sE_{1,0} \cdot E_{1,p} = s\delta_{1p},
\end{equation*}
and since the determinant of the intersection matrix $\tridiag(d_{1,1},\dots,d_{1,m_1})$ of~$E_{1,q}$
is equal to~$r_1$ by~\eqref{eq:r-a-d} and~$r_1$ divides~$s$,
it follows that there exists (a unique) $F_1$ such that $F_1 \cdot E_{1,p} = -s\delta_{p,1}$.
Similarly, for any $F_1$ we have
\begin{equation*}
F_1 \cdot E_{n,p} = 0
\qquad\text{and}\qquad
sE_{1,0} \cdot E_{n,p} = s\delta_{p,m_n} = (s+1)\delta_{p,m_n} - \delta_{p,m_n},
\end{equation*}
and since the determinant~$r_n$ of the intersection matrix $\tridiag(d_{n,1},\dots,d_{n,m_n})$ of~$E_{n,q}$ divides~$s + 1$
there exists (a unique) $F_n$ such that $F_n \cdot E_{n,p} = -(s + 1)\delta_{p,m_n}$.

With the choices of $F_1$ and $F_n$ made as above, we have~\eqref{eq:M-def}.
It remains to note that~$\pi_*(F_1) = \pi_*(F_n) = 0$, hence $C = \pi_*(\rc_1(\cM)) = \pi_*(sE_{1,0}) = sC_1$.
\end{proof}

It is curious that the quality of the collection $R_i$ and $M_i$ of objects in $\Db(X)$
may depend on the choice of indexing of torus-invariant points.
For instance, if only one of these points is non-Gorenstein,
then it makes sense to choose indexing such that it is the point $x_1$ (like we did in Example~\ref{Example:P11n}).
With this choice all the sheaves $M_i$ become locally free.

The reflexive sheaves $R_i$ have especially simple form
when one of the $\rT$-invariant points of $X$ is smooth
and we choose an ordering to make this point the last point.

\begin{corollary}
\label{cor:toric-refl-smooth-point}
Let $X$ be a projective toric surface with $\rT$-invariant points
$x_1, \dots, x_n$ and $\rT$-invariant divisors $C_1, \dots, C_n$.
Assume that $x_n$ is smooth.
Then we have the following semiorthogonal decomposition of $\Db(X)$:
\begin{equation}
\label{eq:semiorth-toric-explicit}
\Db(X) = \big\langle \langle \OO_X(K_X + C_1) \rangle, \dots,
\langle \OO_X(K_X + C_1 + \dots + C_n) \rangle \big\rangle.
\end{equation}
\end{corollary}

Note that since $K_X = -\sum_{i=1}^n C_i$,
the last reflexive sheaf in the semiorthogonal decomposition
\eqref{eq:semiorth-toric-explicit} is trivial.

\begin{proof}
Since $x_n$ is smooth, so that $r_n = 1$, we can take $s = 0$ in Proposition \ref{prop:toric-gcd} and then obtain $C = 0$.
We plug this
into the semiorthogonal decomposition given by Proposition~\ref{proposition: toric maximal extension},
and~\eqref{eq:semiorth-toric-explicit} follows.
\end{proof}

%


The following two examples generalize~\cite[Example~5.7, Example~5.8]{Kawamata}
and refine the prediction of Kawamata
presented at the end of~\cite[Section~5]{Kawamata}.

\begin{example}
\label{ex:pabc}
Let us construct three reflexive sheaves
which provide a semiorthogonal decomposition of $\P(w_1,w_2,w_3)$
for any pairwise coprime positive integers $w_i$.
Using Lemma~\ref{lemma:toric-br} it is easy to compute that
$\Cl(X) \cong \Z$, and if $\cO(1)$ is the ample generator of~$\Cl(X)$, then~$\Pic(X) \cong \Z$ and is generated by $\cO(w_1w_2w_3)$.
We order the torus invariant points in such a way that $x_i$ has order $w_i$.
Then
\begin{equation*}
\cO(C_1) \cong \cO(w_2),
\qquad
\cO(C_2) \cong \cO(w_3),
\qquad
\cO(C_3) \cong \cO(w_1).
\end{equation*}
Let $s$ be an integer such that $s \equiv 0 \mod w_1$ and $s \equiv -1 \mod w_3$.
Then by Proposition~\ref{prop:toric-gcd} we have
\begin{equation*}
R_1 = \cO(sw_2-w_1-w_3),\qquad
R_2 = \cO(sw_2-w_1),\qquad
R_3 = \cO(sw_2).
\end{equation*}
Note that permuting $w_1,w_2,w_3$ we get various semiorthogonal
decompositions of the same category $\Db(\P(w_1,w_2,w_3))$.
\end{example}

\begin{example}
\label{ex:p1ab}
As a special case of the previous example, let $X:=\mathbb{P}(1,a,b)$ for coprime integers $a,b > 0$.
If we order the singular points such that $x_1$ has type~$\frac1b(1,a)$, $x_2$ has type~$\frac1a(1,b)$, and $x_3$ is smooth,
then
\begin{equation*}
\cO(C_1) \cong \cO(a),
\qquad
\cO(C_2) \cong \cO(1),
\qquad
\cO(C_3) \cong \cO(b),
\end{equation*}
and by Corollary~\ref{cor:toric-refl-smooth-point} we have
\begin{equation*}
R_1 = \cO(- b - 1),\qquad
R_2 = \cO(- b),\qquad
R_3 = \cO.
\end{equation*}
By Proposition~\ref{proposition: toric maximal extension} we have reflexive sheaves $M_1, M_2, M_3$
constructed as maximal iterated extensions of the rank~$1$ reflexive sheaves $R_1, R_2, R_3$ on $X$.

Furthermore, $M_1$ has rank $b$, $\End_X(M_1) \cong K_{b,a}$,
$M_2$ has rank $a$, $\End_X(M_2) \cong K_{a, b'}$, where $b' \equiv b^{-1} \mod a$,
and $M_3 = R_3 = \cO$, $\End_X(M_3) \cong \Bbbk$,
and by Proposition~\ref{prop:semiort-ri-mi} there is a semiorthogonal decomposition
\begin{equation*}
\Db(\mathbb{P}(1,a,b)) =
\Big\langle \langle R_1 \rangle, \langle R_2 \rangle, \langle R_3 \rangle \Big\rangle =
\langle \Db(K_{b,a}), \Db(K_{a,b'}), \Db(\Bbbk) \rangle.
\end{equation*}
In this case $M_3 = R_3$ and $M_2$ are locally free while $M_1$ is locally free at $x_1$ and $x_3$,
but is locally free at $x_2$ if and only if $b \equiv -1 \mod a$ and otherwise is only reflexive.
\end{example}


\appendix

\section{Semiorthogonal decomposition of perfect complexes}

The next result was explained to us by Sasha Efimov.

\begin{theorem}
\label{theorem:efimov}
If $X$ is a projective scheme over $\Bbbk$ and $\cA \subset \Db(X)$ is right admissible
then $\cA^\perf = \cA \cap \Dperf(X)$ is left admissible in $\Dperf(X)$.
\end{theorem}

The proof of this result takes the rest of this section and uses the machinery of DG-categories.
We refer to~\cite[Section~3]{KL} and references therein for all prerequisites.

Let $\dg{\cD}$ be an essentially small DG-category.
Recall that a {\sf left DG-module} over $\dg{\cD}$ is a DG-functor $M \colon \dg{\cD} \to \cD(\Bbbk)$
to the DG-category of complexes of $\Bbbk$-vector spaces.
We will say that a DG-module is {\sf finite-dimensional}, resp.\ {\sf acyclic}, if for every $\cF \in \dg{\cD}$
the complex $M(\cF)$ is perfect, i.e., has finite dimensional total cohomology, resp.\ acyclic.

We denote by $\dg{\cD}\md$, $\dg{\cD}\md_\fd$, and $\dg{\cD}\md_\ac$ the DG-categories of \emph{left} DG-modules,
left finite-dimensional DG-modules, and left acyclic DG-modules over~$\dg{\cD}$, respectively.
Furthermore, we denote by
\begin{equation*}
\bD(\dg{\cD}) = [\dg{\cD}\md/\dg{\cD}\md_\ac],
\qquad
\bD_\fd(\dg{\cD}) = [\dg{\cD}\md_\fd/\dg{\cD}\md_\ac],
\end{equation*}
the homotopy categories of the corresponding Drinfeld quotients,
i.e., the derived category of $\dg{\cD}$ and the finite-dimensional derived category of $\dg{\cD}$ respectively.
We note that~$\bD(\dg{\cD})$ is triangulated and $\bD_\fd(\dg{\cD})$ is a full triangulated subcategory in $\bD(\dg{\cD})$.

We say that a full subcategory $\dg{\cA}$ of a DG-category $\dg{\cD}$ is right (resp.\ left) admissible
if its homotopy category $[\dg{\cA}] \subset [\dg{\cD}]$
is right (resp.\ left) admissible, i.e. admits a right (resp.\ left) adjoint functor.

The proof of the theorem is based on the following observation.

\begin{proposition}
\label{proposition:dg-modules-admissible}
Let $\dg{\cD}$ be an essentially small DG-category.
If $\dg{\cA} \subset \dg{\cD}$ is right admissible then $\bD(\dg{\cA}) \subset \bD(\dg{\cD})$ and $\bD_\fd(\dg{\cA}) \subset \bD_\fd(\dg{\cD})$ are right admissible.
\end{proposition}
\begin{proof}
Let $\alpha \colon \dg{\cA} \to \dg{\cD}$ be the embedding functor.
By $\Res_\alpha \colon \bD(\dg{\cD}) \to \bD(\dg{\cA})$ and~$\Ind_\alpha \colon \bD(\dg{\cA}) \to \bD(\dg{\cD})$
we denote the {\sf restriction of scalars} and {\sf induction} functors:
\begin{equation*}
\Res_\alpha(M)(\cF) = M(\alpha(\cF)),
\qquad
\Ind_\alpha(N) = \Delta_{\dg{\cD}} \otimes_{\dg{\cA}}^{\mathbf{L}} N,
\end{equation*}
where in the second formula the diagonal $\dg{\cD}$-bimodule $\Delta_{\dg{\cD}}$ is considered
as a right $\dg{\cA}$-module by restriction of scalars.
Note that $\Res_\alpha$ is the right adjoint of $\Ind_\alpha$
and full faithfulness of~$\alpha$ implies that of $\Ind_\alpha$, see~\cite[Proposition~3.9]{KL}.
Therefore
\begin{equation}
\label{eq:res-ind}
\Res_\alpha \circ \Ind_\alpha \cong \id_{\bD(\dg{\cA})}
\end{equation}
and in particular, $\bD(\dg{\cA})$ is right admissible in $\bD(\dg{\cD})$.

Now assume that $M$ is a finite-dimensional $\dg{\cD}$-module.
Then $\Res_\alpha(M)$ is a finite-dimensional $\dg{\cA}$-module by definition.

Let us also show that $\Ind_\alpha(N)$ is finite-dimensional for any finite-dimensional $\dg{\cA}$-module $N$.
Let $\dg{\cA}^\perp \subset \dg{\cD}$ be the full DG-subcategory formed by objects with classes
in the right orthogonal $[\dg{\cA}]^\perp \subset [\dg{\cD}]$ of $[\dg{\cA}] \subset [\dg{\cD}]$.
Since $\dg{\cA}$ is right admissible, every object of $\dg{\cD}$ is an extension of an object from $\dg{\cA}$ and an object from~$\dg{\cA}^\perp$.
So, to show that $\Ind_\alpha(N)$ is finite-dimensional it is enough to check that
its values on the objects of $\dg{\cA}$ and of~$\dg{\cA}^\perp$ are finite-dimensional.
In other words, if $\alpha' \colon \dg{\cA}^\perp \to \dg{\cD}$ is the embedding of $\dg{\cA}^\perp$,
it is enough to check that $\Res_\alpha(\Ind_\alpha(N))$ and $\Res_{\alpha'}(\Ind_\alpha(N))$ are both finite-dimensional.
The first follows immediately from~\eqref{eq:res-ind}.
For the second just note that the composition~$\Res_{\alpha'} \circ \Ind_\alpha$ is given
by the derived tensor product with the~$\dg{\cA}^\perp$-$\dg{\cA}$-bimodule
that is equal to the restriction of the diagonal bimodule~$\Delta_{\dg{\cD}}$.
However, this restriction is quasiisomorphic to zero (by semiorthogonality of~$\dg{\cA}$ and~$\dg{\cA}^\perp$),
hence~$\Res_{\alpha'} \circ \Ind_\alpha = 0$, and finite-dimensionality of $\Ind_\alpha(N)$ follows.

From these observations we see that the restriction and induction functors preserve the categories of finite-dimensional DG-modules,
and since they form an adjoint pair and~\eqref{eq:res-ind} holds, we conclude that~$\bD_\fd(\dg{\cA})$
is right admissible in $\bD_\fd(\dg{\cD})$.
\end{proof}

Next we show that the category of left finite-dimensional DG-modules over $\dg{\cD}$
in some cases can be identified with a certain subcategory of $\dg{\cD}$.
We write $\dg{\cD}^\op$ for the opposite DG-category of~$\dg{\cD}$.
For any DG-category $\dg{\cD}$ there is a natural {\sf Yoneda functor}
\begin{equation*}
\rh \colon [\dg{\cD}^\op] \to \bD(\dg{\cD}),
\qquad
\cF \mapsto \rh_\cF(-) := \Hom_{\dg{\cD}}(\cF,-),
\end{equation*}
that takes an object $\cF$ to the corresponding representable left DG-module $\rh_\cF$.
Note that the Yoneda functor is fully faithful, due to the DG-version of the Yoneda Lemma.

We denote by $\dg{\cD}^\hf \subset \dg{\cD}$ the full DG-subcategory formed
by all {\sf homologically finite-dimensional} objects of~$\dg{\cD}$.
By definition this is just the preimage of the category of left finite-dimensional DG-modules over $\dg{\cD}$ under the Yoneda functor, i.e.,
\begin{equation*}
\dg{\cD}^\hf :=
\{ \cF \in \dg{\cD} \mid \rh_\cF \in \bD_\fd(\dg{\cD}) \}.
\end{equation*}

Recall that the minimal subcategory of $\bD(\dg{\cD})$ containing all representable DG-modules and closed under shifts,
cones and homotopy direct summands is called the category of {\sf perfect DG-modules}
(and its objects are called perfect DG-modules over~$\dg{\cD}$).
We will say that a DG-category $\dg{\cD}$ is {\sf perfectly pretriangulated} if the Yoneda functor induces
an equivalence between $[\dg{\cD}^\op]$ and the category of perfect DG-modules over~$\dg{\cD}$.
Recall that a DG-category $\dg{\cD}$ is {\sf smooth} if the diagonal DG-bimodule $\Delta_{\dg{\cD}}$
is perfect.

\begin{lemma}
\label{lemma:hf-fd}
If $\dg{\cD}$ is a smooth perfectly pretriangulated DG-category
then the Yoneda functor induces an equivalence $[(\dg{\cD}^\hf)^\op] \cong \bD_\fd(\dg{\cD})$ of triangulated categories.
\end{lemma}
\begin{proof}
As we already mentioned, the Yoneda functor $\rh$ is fully faithful.
Moreover, by definition it takes $[(\dg{\cD}^\hf)^\op]$ to $\bD_\fd(\dg{\cD})$.
So, it only remains to show that its restriction to $[(\dg{\cD}^\hf)^\op]$ is
homotopically essentially surjective onto $\bD_\fd(\dg{\cD})$.
Let $M$ be a finite-dimensional left DG-module over $\dg{\cD}$.
Since~$\Delta_{\dg{\cD}}$ is perfect, we conclude that
\begin{equation*}
M \cong \Delta_{\dg{\cD}} \otimes_{\dg{\cD}}^{\mathbf{L}} M
\end{equation*}
is a perfect DG-module over $\dg{\cD}$.
Furthermore, since $\dg{\cD}$ is perfectly pretriangulated, we have $M \cong \rh_\cF$ for some $\cF \in \dg{\cD}$.
Since the DG-module $M$ is finite-dimensional, we see that~$\cF \in \dg{\cD}^\hf$.
\end{proof}

In a combination with Proposition~\ref{proposition:dg-modules-admissible}, this gives the following.

\begin{corollary}
\label{corollary:ca-hf-admissible}
Let~$\dg{\cA} \subset \dg{\cD}$ be a right admissible DG-subcategory
in a smooth perfectly pretriangulated DG-category $\dg{\cD}$.
Then $\dg{\cA}^\hf \subset \dg{\cD}^\hf$ is left admissible.
\end{corollary}
\begin{proof}
Note that $\dg{\cA}$ is also smooth and perfectly pretriangulated.
Therefore, from Proposition~\ref{proposition:dg-modules-admissible} and Lemma~\ref{lemma:hf-fd}
we deduce that $(\dg{\cA}^\hf)^\op {} \simeq \bD_\fd(\dg{\cA})$ is right admissible in~$(\dg{\cD}^\hf)^\op$,
hence~$\dg{\cA}^\hf$ is left admissible in~$\dg{\cD}^\hf$.
\end{proof}

Finally, we observe in geometric situations that
the categories $\dg{\cA}^\hf$ and $\dg{\cD}^\hf$ are closely related to the category of perfect complexes.
In the next lemma we consider~$\Db(X)$ as a DG-category (via any appropriate DG-enhancement)
and $\Dperf(X)$
as its DG-subcategory.
To emphasize this we will write $\dg{\cD}^b(X)$ and $\dg{\cD}^\perf(X)$ for these DG-categories.

\begin{lemma}
\label{lemma:hf}
Let $X$ be a projective scheme over $\Bbbk$.
For the DG-category $\dg{\cD} = \dg{\cD}^b(X)$ we have~\mbox{$\dg{\cD}^\hf = \dg{\cD}^\perf(X)$}.
Moreover, if $\dg{\cA} \subset \dg{\cD}^b(X)$ is right admissible then we have~$\dg{\cA}^\hf = \dg{\cA}^\perf$.
\end{lemma}
\begin{proof}
Since $X$ is projective, the category $\dg{\cD}^\hf$ coincides with the subcategory of homologically bounded objects in $\dg{\cD}$
as defined in~\cite[Definition~1.6]{Orl06}.
Therefore, the  first part is just~\cite[Proposition~1.11]{Orl06}
and the second part is deduced from this by~\cite[Proposition~1.10]{Orl06}.
\end{proof}

\begin{proof}[Proof of Theorem~\textup{\ref{theorem:efimov}}]
The category $\dg{\cD} = \dg{\cD}^b(X)$
is smooth (even when $X$ is singular, see~\cite{Lunts}) and perfectly pretriangulated.
Let $\dg{\cA} \subset \dg{\cD}$ be the full DG-subcategory formed by objects with classes in $\cA \subset [\dg{\cD}] = \Db(X)$.
By Corollary~\ref{corollary:ca-hf-admissible} the subcategory $\dg{\cA}^\hf \subset \dg{\cD}^\hf$ is left admissible.
By Lemma~\ref{lemma:hf} this gives the required result.
\end{proof}

\begin{remark}
One can also prove Theorem~\ref{theorem:efimov} replacing the machinery of DG-categories by the technique of~\cite{Neeman}.
The crucial result here is~\cite[Theorem 0.2]{Neeman} identifying the category~$\Dperf(X)^{op}$ with the category of triangulated functors $\Db(X) \to \Db(\Bbbk)$
for any proper scheme~$X$ over~$\Bbbk$.
\end{remark}

%
%
%

\providecommand{\arxiv}[1]{{\tt{arXiv:#1}}}

\end{document}